\documentclass[reqno,a4paper,11pt]{amsart}
\usepackage{mathrsfs}
\usepackage{amssymb}
\usepackage{libertine}
\usepackage[libertine]{newtxmath}
\usepackage[utf8]{inputenc}
\usepackage[english]{babel}
\usepackage{url}
\usepackage[backref=page, colorlinks=true]{hyperref}
\hypersetup{linkcolor=blue, citecolor=blue, filecolor=black, urlcolor=blue}%MidnightBlue}
\usepackage{comment}

\usepackage{cancel}

\renewcommand*{\backrefalt}[4]{   %a second more readable format for the backrefs 
	\ifcase #1 (Not cited.)        %a second more readable format for the backrefs
	\or        (Cited on page~#2.) %a second more readable format for the backrefs
	\else      (Cited on pages~#2.)%a second more readable format for the backrefs
	\fi}

\newcommand{\crosses}[1]{%
	\ifcase#1\relax
	\or
	\rslash\or
	\rslash\mskip-5.5mu\rslash\or
	\rslash\mskip-5.5mu\rslash\mskip-5.5mu\rslash%
	\fi
}
\newcommand{\rslash}{\raisebox{.15ex}{/}}
\usepackage[usenames,dvipsnames]{xcolor}
\usepackage{color}
\usepackage{geometry}
\usepackage[all]{xy}
\usepackage{enumitem}

\usepackage{slashed}
\usepackage{bbm}
\usepackage{cancel}

\usepackage{cleveref}
\usepackage{tikz}
\usepackage{tikz-cd}

\usepackage{xargs}  % Use more than one optional parameter in a new commands
\usepackage{soul}

\usepackage{amsmath}
\usepackage{graphicx}
\usepackage{amsthm}
\usepackage{latexsym}
\usepackage{amsfonts}
\usepackage{bbm}
\usepackage{mathrsfs}
\usepackage{cases}

\makeatletter
\@namedef{subjclassname@2020}{%
	\textup{2020} Mathematics Subject Classification}
\makeatother

\numberwithin{equation}{section}

\theoremstyle{plain}
\newtheorem{lemma}{Lemma}[section]
\newtheorem{proposition}[lemma]{Proposition}
\newtheorem{proposition/definition}[lemma]{Proposition/Definition}
\newtheorem*{theorem*}{Theorem}
\newtheorem{theorem}[lemma]{Theorem}
\newtheorem{corollary}[lemma]{Corollary}

\theoremstyle{definition}
\newtheorem{definition}[lemma]{Definition}
\newtheorem{remark}[lemma]{Remark}
\newtheorem{example}[lemma]{Example}

\DeclareRobustCommand{\SkipTocEntry}[5]{}

\DeclareMathOperator{\id}{id}

\DeclareMathOperator{\rank}{rank}

\DeclareMathOperator{\Hom}{Hom}

\DeclareMathOperator{\gr}{Gr}

\newcommand{\pr}{\text{pr}}
\newcommand{\hooklongrightarrow}{\lhook\joinrel\longrightarrow}

\newcommand{\calD}{\mathcal{D}}
\newcommand{\calE}{\mathcal{E}}

\newcommand{\calH}{\mathcal{H}}

\newcommand{\calL}{\mathcal{L}}
\newcommand{\calM}{\mathcal{M}}

\newcommand{\calQ}{\mathcal{Q}}

\newcommand{\calX}{\mathcal{X}}

\newcommand{\scrL}{\mathscr{L}}

\newcommand{\bbD}{\mathbb{D}}

\newcommand{\bbN}{\mathbb{N}}

\newcommand{\bbR}{\mathbb{R}}

\newcommand{\bbT}{\mathbb{T}}

\newcommand{\bbZ}{\mathbb{Z}}

\newcommand{\frakX}{\mathfrak{X}}

\newcommand{\fraka}{\mathfrak{a}}

\newcommand{\frakg}{\mathfrak{g}}
\newcommand{\frakh}{\mathfrak{h}}

\newcommand{\frakm}{\mathfrak{m}}

\newcommand{\rmd}{\mathrm{d}}

\newcommand{\rmh}{\mathrm{h}}

\renewcommand{\phi}{\varphi}

\newcommand{\ldsb}{[\![}
\newcommand{\rdsb}{]\!]}
\newcommand{\ldab}{\langle\!\langle}
\newcommand{\rdab}{\rangle\!\rangle}
\renewcommand{\theta}{\vartheta}

\allowdisplaybreaks
\DeclareRobustCommand{\SkipTocEntry}[5]{}

\title[]{The Deformation $L_\infty$ algebra of a Dirac--Jacobi structure}

\author{Alfonso Giuseppe Tortorella}
\address{Centro de Matemática da Universidade do Porto, Rua do Campo Alegre 867, 4197-007 Porto, Portugal
	\newline
	Geometry Section, Department of Mathematics, KU Leuven, Celestijnenlaan 200B - 3001 Leuven, Belgium
}
\email{\href{mailto:alfonsogiuseppe.tortorella@kuleuven.be}{\underline{\smash{alfonsogiuseppe.tortorella@kuleuven.be}}}\qquad\href{mailto:alfonso.tortorella@fc.up.pt}{\underline{\smash{alfonso.tortorella@fc.up.pt}}}}


\makeatletter
\@namedef{subjclassname@2020}{%
	\textup{2020} Mathematics Subject Classification}
\makeatother


\begin{document}

\begin{abstract}
	We develop the deformations theory of a Dirac--Jacobi structure within a fixed Courant--Jacobi algebroid.
	Using the description of split Courant--Jacobi algebroids as degree $2$ contact $\bbN Q$ manifolds and Voronov's higher derived brackets, each Dirac--Jacobi structure is associated with a cubic $L_\infty$ algebra for any choice of a complementary almost Dirac--Jacobi structure.
	This $L_\infty$ algebra governs the deformations of the Dirac--Jacobi structure: there is a one-to-one correspondence between the MC elements of this $L_\infty$ algebra and the small deformations of the Dirac-Jacobi structure.
	Further, by Cattaneo and Sch\"atz's equivalence of higher derived brackets, this $L_\infty$ algebra does not depend (up to $L_\infty$-isomorphisms) on the choice of the complementary almost Dirac--Jacobi structure.
	These same ideas apply to get a new proof of the independence of the $L_\infty$ algebra of Dirac structure from the choice of a complementary almost Dirac structure (a result proved using other techniques by Gualtieri, Matviichuk and Scott).
\end{abstract}

\maketitle

\tableofcontents

\section{Introduction}
\label{sec:intro}

%background on deformation theory
Originated with Riemann’s 1857 memoir on abelian functions, the study of deformation problems plays a central role in algebra, geometry and mathematical physics.
In the late 50s, the deformation theory of complex structures was developed by Kodaira, Spencer and Kuranishi.
Deformation theory of algebraic structures was initiated, in the 60s, by Gerstenhaber (for associative algebras) and Richardson and Nijenhuis (for Lie algebras).
All this contributed to consolidate the philosophy (by Quillen, Drinfeld, Deligne,\ldots) according to which every deformation problem is controlled by a differential graded Lie algebra (dgLa for short) or an $L_\infty$ algebra, i.e.~a graded Lie algebra up to homotopy.
A dgLa (resp.~an $L_\infty$ algebra) is said to control the deformations of a structure if there is a one-to-one correspondence between the small deformations of the given structure and the Maurer–Cartan (MC) elements of the dgLa (resp.~$L_\infty$ algebra).
Moreover, dgLa's (resp.~$L_\infty$-algebras) which are equivalent as $L_\infty$-algebras govern equivalent deformation problems with equivalent moduli spaces.
However, in general, it is not trivial at all to construct the $L_\infty$ algebra governing the deformation problem of a given structure.

%aim of the paper
This paper aims at developing the deformation theory of Dirac--Jacobi structures.
These structures generalizes Dirac structures and have an intrinsic potential to highlight the rich interplay between several geometric structures like (pre-)contact, Jacobi, and almost contact structures, as well as generalized complex structures on odd-dimensional manifolds (often called \emph{generalized contact structures}~\cite{iglesias2005contact,poon2011generalized,vitagliano2016generalized}).
As detailed below, Dirac--Jacobi structures can be understood as the ``contact versions'' of the Dirac structures.
This interpretation is conceptually grounded on the close relation between symplectic/Poisson and contact/Jacobi geometry.
Indeed, on the one hand, contact structures can be seen, via \emph{symplectization}, as the odd-dimensional analogue of symplectic structures.
On the other hand, Jacobi structures (independently introduced by Kirillov~\cite{kirillov1976local} and Lichnerowicz~\cite{lichnerowicz1978varietes}) can be seen as the (possibly degenerate) contravariant generalization of contact structures exactly like Poisson structures are the (possibly degenerate) contravariant generalization of symplectic structures.
This point of view consists in adopting the line bundle approach of~\cite{vitagliano2018dirac} and seeing contact/Jacobi geometry as symplectic/Poisson geometry on the \emph{gauge (or Atiyah) algebroid} $DL$ of a line bundle $L\to M$.

%background on Courant Algebroids and Dirac Structures
The origin of the notion of Dirac structure goes back to Courant and Weinstein's work on the geometric interpretation of Dirac's theory of constrained mechanical systems and Dorfman's work on integrable systems.
In this setting, a Dirac structure is a Lagrangian subbundle of the generalized tangent bundle $\bbT M:=TM\oplus T^\ast M$ satisfying an integrability condition phrased in terms of the Courant bracket or equivalently (as we will do in the sequel) in terms its skew-symmetrization, i.e.~the Dorfman bracket.
Only later, in Liu, Weinstein and Xu's work on the double of a Lie bialgebroid, Courant algebroids emerged as the natural framework for the study of Dirac structures.
Loosely speaking, a Courant algebroid is a vector bundle $E\to M$ equipped with a non-degenerate symmetric product $\ldab-,-\rdab$, an anchor $\rho\colon E\to TM$ and a Loday bracket $\ldsb-,-\rdsb$ on $\Gamma(E)$ satisfying a series of compatibility conditions.
Then a Dirac structure is a subbundle $A\subset E$ which is Lagrangian wrt the product $\ldab-,-\rdab$ and involutive wrt the bracket $\ldsb-,-\rdsb$.
There is also an alternative, but equivalent, way to describe Courant algebroids that was introduced by Roytenberg.
Indeed, a Courant algebroid can be seen as a degree $2$ symplectic $\bbN Q$ manifold.% $\calM$ equipped with a cohomological degree $1$ Hamiltonian vector field $X_\Theta$ or equivalently an MC element $\Theta$ of $(C^\infty(\calM)[2],\{-,-\})$, where $\{-,-\}$ is the corresponding degree $-2$ non-degenerate Poisson structure on $\calM$.
In this way, for instance, the standard Courant algebroid $\bbT M$ gets identified with the degree $2$ symplectic $\bbN$ manifold $T^\ast[2]TM[1]$ equipped with the cohomological degree $1$ Hamiltonian vector field associated to de Rham vector field $\rmd_{\text{dR}}$ on $TM[1]$ which can be seen, by ``Hamiltonian lift'', as an MC element of $(C^\infty(T^\ast[2]TM[1])[2],\{-,-\})$.

%review of deformation theory of Dirac structures
Dirac geometry has a built-in potential to unify several important geometric structures such as (pre)symplectic, Poisson, complex and generalized complex structures.
This motivated the interest received over the years by their deformation problem. 
In~\cite{liu1997manintriples}, given a Lie bialgebroid $(A,A^\ast)$, Liu, Weinstein and Xu observed that the shifted de Rham complex of the Lie algebroid $A$ enriched with the Gerstenhaber bracket of the Lie algebroid $A^\ast$ forms a dgLa $(\Omega^\bullet(A)[1],\rmd_A,[-,-]_{A^\ast})$ which controls the small deformations of the Dirac structure $A$ within the Courant algebroid $A\oplus A^\ast$, i.e.~the double of $(A,A^\ast)$.
However, unlike in the double of a Lie bialgebroid, a Dirac structure $A$ doesn't admit, in general, a complementary Dirac structure: the most that one can get is a complementary almost Dirac structure.
In~\cite{fregier_zambon_2015,Keller2007formaldeformationsDirac,roytenberg2002quasi}, given a Courant algebroid $E$ and a Dirac structure $A\subset E$, it is observed that each choice of an almost Dirac structure $B\subset E$ complementary to $A$ gives rise to a cubic $L_\infty$ algebra $(\Omega^\bullet(A)[1],\rmd_A,[-,-]_B,[-,-,-]_B)$ controlling the deformation problem of $A$.
Later on, the uniquess (up to $L_\infty$-isomorphisms) of the above $L_\infty$ algebra was proved in~\cite{gualtieri2020deformation}.
Precisely, given a Dirac structure $A\subset E$, they proved that different choices of the complementary almost Dirac structure $B\subset E$ give rise to $L_\infty$ algebras that are canonically $L_\infty$-isomorphic.
Further, the deformation theory of Dirac structures has also found direct applications to the study of deformation problems of related geometric structures like generalized complex structures~\cite{gualtieri2011generalized}, generalized K\"ahler structures~\cite{goto2010deformations}, pre-symplectic structures\cite{SZDirac}, and symplectic foliations~\cite{GTZ2021symplecticfoliations}. 

%background on Courant--Jacobi algebroids and Dirac--Jacobi structures

The natural setting for studying Dirac--Jacobi structures is provided by Courant--Jacobi algebroids, or equivalently contact-Courant algebroids.
Following Grabowski~\cite{Grabowski2013gradedcontactCourant}, a Courant--Jacobi algebroid can be described as a degree $2$ contact $\bbN$ manifold $\calM$, with underlying line bundle $\calL\to\calM$, equipped with a cohomological degree $1$ contact vector field $X_\Theta$ or equivalently an MC element $\Theta$ of $(\Gamma(\calL)[2],\{-,-\})$,  where $\{-,-\}$ is the corresponding degree $-2$ non-degenerate Jacobi structure on $\calL\to\calM$.
In terms of more classical data, a Courant--Jacobi algebroid $(E;L)$ consists of a vector bundle $E\to M$ and a line bundle $L\to M$ equipped with a non-degenerate $L$-valued symmetric product $\ldab-,-\rdab$, a vector bundle morphism $\nabla\colon E\to DL$ and a Loday bracket $\ldsb-,-\rdsb$ on $\Gamma(E)$ satisfying a series of compatibility conditions.
In this way, the latter generalizes the notion of Courant algebroid.
In particular, the omni-Lie algebroid $\bbD L:=DL\oplus J^1L$ of a line bundle~\cite{chen2010courant}, the prototypical example of Courant--Jacobi algebroid, identifies with the degree $1$ contact $\bbN$ manifold $J^1[2]L_{DL}$, with underlying line bundle $\calL:=J^1[2]L_{DL}\times_M L_{DL}\to J^1[2]L_{DL}$, equipped with the MC element obtained as ``Hamiltonian lift'' of $\rmd_{D}$, the der-differential~\cite{rubtsov1980cohomology} seen as cohomological degree $1$ derivation of the pull-back line bundle $L_{DL}:=DL[1]\times_ML\to {DL}[1]$.
Given a Courant--Jacobi algebroid $(E;L)$, a Dirac--Jacobi structure is a subbundle $A\subset E$ which is Lagrangian wrt the product $\ldab-,-\rdab$ and involutive wrt the bracket $\ldsb-,-\rdsb$.
This is actually the line bundle theoretic version of the notion of Dirac--Jacobi structures originally introduced in~\cite{grabowski2002graded} generalizing Dirac structures.

%aim of the paper and sketch of its main results
In this paper we aim at developing the deformation theory of Dirac--Jacobi structures.
In doing so we meet a substantial difference with the analogous problem for Dirac structures.
Indeed, the deformation problem of Dirac structures has been handled by the combination of two different approaches.
On the one hand, for each choice of a complementary almost Dirac structure, the $L_\infty$ algebra controlling the small deformation of a given Dirac structure $A$ was first constructed in~\cite{fregier_zambon_2015,roytenberg2002quasi} using the interpretation of Courant algebroids as degree $2$ symplectic $\bbN Q$ manifolds and Voronov's higher derived brackets. 
On the other hand, the approach to Dirac structures via their corresponding pure spinors (for the Clifford algebra of the ambient Courant algebroid) has been successfully used in~\cite{gualtieri2011generalized} to prove that different choices of the complementary almost Dirac structure give rise to canonically isomorphic $L_\infty$ algebras.
However, a similar spinorial description of Dirac--Jacobi structures is still unavailable.
So, in this paper, the deformation theory of Dirac--Jacobi structures will be entirely developed using only the interpretation of Courant--Jacobi algebroids as degree $2$ contact $\bbN Q$ manifolds and Voronov's higher derived brackets.
Below we sketch the main results: their statements will be made precise and better explained in the body of the paper.

\begin{theorem*}[{\bf A}]
	For each split Courant--Jacobi algebroid $(A\oplus A^\dagger;L)$, there is an associated curved cubic $L_\infty$ algebra $(\Omega^\bullet(A;L)[1],\{\mu_k\})$ which fully encodes the split Courant--Jacobi algebroid structure.
\end{theorem*}
This property of split Courant--Jacobi algebroids, prompts the construction the \emph{deformation $L_\infty$ algebra} of a Dirac--Jacobi structure $A$ in a Courant--Jacobi algebroid $(E;L)$.
\begin{theorem*}[{\bf B}]
	Each almost Dirac--Jacobi structure $B$ complementary to $A$ in $E$ determines a cubic $L_\infty$ algebra structure $\{\lambda^B_k\}$ enriching the shifted de Rham complex $(\Omega^\bullet(A;L)[1],\rmd_{A,L})$.
\end{theorem*}
\begin{theorem*}[{\bf C}]
	Different choices of the complementary almost Dirac--Jacobi structure give rise to canonically isomorphic $L_\infty$ algebras.
\end{theorem*}
Here we use the equivalence of higher derived brackets~\cite{CS} to prove the independence (up to $L_\infty$ isomorphisms) of the $L_\infty$ algebra $(\Omega^\bullet(A;L)[1],\{\lambda_k^B\})$ on the choice of $B$.
This immediately leads to an alternative proof of the analogous results for Dirac structures (already proved in~\cite{gualtieri2020deformation} using different methods).
Finally, we address the deformation problem.
\begin{theorem*}[{\bf D}]
	For each complementary almost Dirac--Jacobi structure $B$, the cubic $L_\infty$ algebra $(\Omega^\bullet(A;L)[1],\{\lambda_k^B\})$ controls the small deformations of $A$.
	Indeed, there is a canonical bijection between:
	\begin{itemize}
		\item those Dirac--Jacobi structures $A^\prime\subset E$ that are close to $A$ w.r.t.~$B$ in sense that $E=A^\prime\oplus B$, and
		\item the MC elements of $(\Omega^\bullet(A;L)[1],\{\lambda_k^B\})$, i.e.~those $\eta\in\Omega^2(A;L)$ satisfying the MC equation
		\begin{equation*}
			\rmd_{A,L}\eta+\frac{1}{2}\lambda_2^B(\eta,\eta)+\frac{1}{6}\lambda_3^B(\eta,\eta,\eta)=0.
		\end{equation*}
	\end{itemize}
\end{theorem*}
As a by-product of the latter, we also identify the infinitesimal deformations of a Dirac--Jacobi structure $A$ and find sufficient criteria for the existence of obstructions.
Additionally, the latter result immediately suggests applications to the deformation problems of related geometric structures like pre-contact structures, regular Jacobi structures (namely locally conformally symplectic and contact foliations), and generalized complex structures on odd-dimensional manifolds.
The investigation of these related deformation problems will be pursued in future work.

%structure of the paper
\textbf{Structure of the paper.}
Section~\ref{sec:Courant-Jacobi_algebroids} reviews the line bundle approach to Jacobi algebroids, Courant--Jacobi algebroids and Dirac--Jacobi structures.
% including the Atiyah algebroid, the first jet bundle and the omni-Lie algebroid of a line bundle.
Section~\ref{sec:degree_2_NQ-manifolds} describes in detail the canonical contact structure of the shifted first jet bundle $J^1[2]L_\calQ$ of a line bundle generated in degree $0$ over an $\bbN$ manifold $\calQ$.
For the aims of the paper, we focus our attention on the case of an $\bbN$ manifold of degree $1$, i.e.~$\calQ=A[1]$, for some vector bundle $A\to M$.
In particular, in this setting, we construct a contact version of the Legendre transform in Theorem~\ref{theor:contact_Legendre}.
As special case of the identification of Courant--Jacobi algebroids with degree $2$ contact $\bbN$ manifolds, Section~\ref{sec:split-CJ_algebroids} identifies split Courant algebroids $(A\oplus A^\ast;L)$ with the degree $2$ contact $\bbN$ manifold $J^1[2]L_A$, equipped with a cohomological degree $1$ contact vector field $X_\Theta$.
In Theorem~\ref{theor:higher_derived_brackets:splitCJ} we use this identification to construct the curved cubic $L_\infty$ algebra associated with the split Courant--Jacobi algebroid.
In the final Section~\ref{sec:deformation_theory:DJ}, we address the deformation problem of a Dirac--Jacobi structure $A$.
First, we construct the deformation $L_\infty$ algebra of a Dirac--Jacobi structure $A$.
Specifically, in Theorem~\ref{theor:deformation_L_infty_algebra}, we prove that each choice of a complementary almost Dirac--Jacobi structure $B$ gives rise to a cubic $L_\infty$ algebra and, in Theorem~\ref{theor:GMS}, we show that different choices of $B$ give rise to canonically isomorphic $L_\infty$ algebras.
Then, in Theorem~\ref{theor:DJ_deformation:MC-elements}, we prove that this $L_\infty$ algebra controls the deformation problem of the given Dirac--Jacobi structure $A$.
Additionally, we also identify the infinitesimal deformations and find sufficient criteria for the existence of obstructions.
Finally, for the reader's convenience, Appendix~\ref{app:L_infty_algebras} reviews the identification of (morphisms of) $L_\infty$ algebras with (morphisms of) codifferential coalgebras and recalls, without proofs, the results on the equivalence of higher derived brackets from~\cite{CS}.

%prerequisites
%\textbf{We assume the reader is familiar with the fundamentals of Lie algebroids}
\textbf{Notation.}
We assume the reader is familiar with the fundamentals of Lie algebroids.
For any vector bundle $E\to M$, we denote by $DE\to M$ the \emph{gauge (or Atiyah) algebroid} of $E$.
The sections of $DE$ are the \emph{derivations of $E$}, i.e.~those linear first order differential operators $\Delta\in\Gamma((J^1E)^\ast\otimes E)$ whose symbol $\sigma(\Delta)\in\Gamma(TM\otimes\operatorname{End}(E))$ takes values in the vector subbundle $TM\simeq TM\otimes\langle\id_E\rangle\subset TM\otimes\operatorname{End}(E)$.
Then the Lie bracket on $\calD E:=\Gamma(DE)$ is the ordinary commutator and the anchor is the well-defined symbol map $DE\to TM,\ \Delta,\mapsto\sigma(\Delta)$.
In particular, $DL=(J^1L)^\ast\otimes L$, for any line bundle $L\to M$.
Another relevant tool is the language of $L_\infty$ algebras~\cite{lada1995strongly}.
In the body of the paper we will always work with the equivalent notion of $L_\infty[1]$ algebra (cf., e.g.,~\cite{fiorenza2007cones}) in which all the multibrackets are degree $1$ graded symmetric.
Additionally, the paper requires also familiarity with graded geometry, for which we refer the reader to~\cite{mehta2006supergroupoids}.

\section{Jacobi Algebroids, Courant--Jacobi algebroids and Dirac--Jacobi structures}
\label{sec:Courant-Jacobi_algebroids}

In order to set what will be our general framework, this section reviews the line bundle approach to Jacobi algebroids, Courant--Jacobi algebroids and Dirac--Jacobi structures.
For more details on, specifically, the line bundle approach to (pre-)contact and Jacobi structures we refer the reader to~\cite[Section~2]{blaga2020contact}, \cite[Section~2]{schnitzer2021weak} and references therein.

\subsection{Jacobi Algebroids}
\label{sec:Jacobi_algebroids}
%Jacobi algebroids
Jacobi algebroids were first introduced in~\cite{grabowski2001jacobi}  and~\cite{Iglesias2001generalizedLiebialgbds} under the name generalized Lie algebroids.
Here, following~\cite[Definition 2.5]{le2018deformations}, we recall a more general definition, adapted to the realm of non-necessarily trivial line bundles.

\begin{definition}
	\label{def:Jacobi algebroid}
	A \emph{Jacobi algebroid} over a manifold $M$ consists of a Lie algebroid $A\to M$, with Lie bracket $[-,-]$ and anchor $\rho$, and a line bundle $L\to M$ endowed with a flat $A$-connection $\nabla:A\to DL$.
\end{definition}

\begin{remark}
	\label{rem:Kirillov_algebroid}
	Jacobi algebroids are equivalent to Grabowski's Kirillov algebroids~\cite[Section 8]{Grabowski2013gradedcontactCourant}.
\end{remark}

%%The Jacobi algebroid of a Jacobi manifold
%\begin{example}
%	\label{ex:Jacobi_algebroid:Jacobi_manifold}
%	Let $L\to M$ be a line bundle.
%	Each Jacobi structure $J=\{-,-\}\in\calD^2 L$ determines a Jacobi algebroid $(J^1L;L)$ with Lie bracket $[-,-]_J$, anchor $\rho_J$ and representation $\nabla^J$ defined by setting
%	\begin{equation*}
%		[\alpha,\beta]_J=\calL_{J^\sharp\alpha}\beta-\calL_{J^\sharp\beta}\alpha-\rmd_D\langle J,\alpha\wedge\beta\rangle,\quad \rho_J(\alpha)=(\sigma\circ J^\sharp)(\alpha),\quad \nabla^J_\alpha=J^\sharp\alpha,
%	\end{equation*}
%for all $\alpha,\beta\in\Gamma(J^1L)$, where $J^\sharp:J^1L\to(J^1L)^\ast\otimes L\simeq DL$ is the sharp map of $J\in\Gamma(\wedge^2(J^1L)^\ast\otimes L)$ and we are using the Cartan calculus on the der-complex $(\Omega^\bullet_D(L),\rmd_D)$.
%\end{example}

%the gauge algebroid with its tautological representation
\begin{example}
	\label{ex:tautological_representation_DL}
	For any line bundle $L\to M$, the gauge algebroid $DL\to M$ has a tautological representation $\nabla$ on $L$ given by the identity map, i.e.~$\nabla_\Delta=\Delta$ for all $\Delta\in\calD L$.
	The gauge algebroid $DL$ with its tautological representation on $L$ represent the prototypical example of a Jacobi algebroid.
\end{example}

%almost_Jacobi_algebroid
In this paper we also need to work with the ``almost'' version of a Jacobi algebroid which is obtained from Definition~\ref{def:Jacobi algebroid} removing the Jacobi identity of $[-,-]$ and the flatness of $\nabla$ as in the following.
\begin{definition}
	\label{def:almost_Jacobi_algebroid}
	An \emph{almost Jacobi algebroid} over a manifold $M$ consists of an almost Lie algebroid $A\to M$, a line bundle $L\to M$ and a VB morphism $\nabla:A\to DL,\ u\mapsto\nabla_u,$ over $\id_M$, such that
	\begin{equation*}
		[u,fv]=(\sigma(\nabla_u)f)v+f[u,v],\qquad\text{for all $u,v\in\Gamma(A)$ and $f\in C^\infty(M)$.}
	\end{equation*}
\end{definition}

Denote by $\bbR_M$ the trivial line bundle over manifold $M$, i.e.~$\bbR_M=M\times\bbR$.
Then $\Gamma(\bbR_M)=C^\infty(M)$ and $\Delta f=\sigma(\Delta)f+f(\Delta 1)$ for all $\Delta\in\calD\bbR_M:=\Gamma(D\bbR_M)$ and $f\in C^\infty(M)$.
Hence there is a canonical identification $\bbD\bbR_M\overset{\sim}{\to}TM\oplus\bbR_M,\ \delta\mapsto(\sigma(\delta),\delta(1))$ that we will understand in the following.

%Lie algebroids as Jacobi algebroids
\begin{remark}
	\label{rem:Lie_algebroid}
	Each (almost) Lie algebroid $A\to M$, with structure maps $([-,-],\rho)$, identifies with the (almost) Jacobi algebroid $(A;\bbR_M)$, with structure maps $([-,-],\nabla)$, where $\nabla=\rho:A\to TM\subset D\bbR_M$.
\end{remark}

\begin{remark}[\textbf{The de Rham differential}]
	\label{rem:deRham_differential}
Let $(A;L)$ be an almost Jacobi algebroid.
Denote by $\Omega^\bullet(A)$ the graded commutative algebra of forms on $A$ and by $\Omega^\bullet(A;L)$ the graded $\Omega^\bullet(A)$-module of $L$-valued forms on $A$, so that
\begin{equation*}
	\Omega^\bullet(A):=\Gamma(\wedge^\bullet A^\ast)\qquad\text{and}\qquad \Omega^\bullet(A;L):=\Gamma(\wedge^\bullet A^\ast\otimes L).
\end{equation*}
Then one can construct the degree $1$ graded $\bbR$-linear map $\rmd_{A,L}:\Omega^\bullet(A;L)\to\Omega^\bullet(A;L)$ associated with $(A;L)$ by the following Koszul-like formula
\begin{equation}
	\label{eq:de_Rham_differential}
	\begin{aligned}
	(\rmd_{A,L}\alpha)(X_0,\ldots,X_k)=&\sum_i(-)^i\nabla_{X_i}(\alpha(X_1,\ldots,\widehat{X_i},\ldots,X_k))
	\\
	&+\sum_{i<j}(-)^{i+j}\alpha([X_i,X_j],X_1,\ldots,\widehat{X_i},\ldots,\widehat{X_j},\ldots,X_k),
\end{aligned}
\end{equation}
for all $\alpha\in\Omega^k(A;L)$ and $X_0,X_1,\ldots,X_k\in\Gamma(A)$.
A similar formula defines the degree $1$ graded linear map $d_A:\Omega^\bullet(A)\to\Omega^\bullet(A)$ associated with the almost Lie algebroid $A$.
Consequently, as it is easy to check, $\rmd_{A,L}$ is also characterized by the following two conditions
\begin{equation}
	\label{eq:de_Rham_differential:bis}
	(\rmd_{A,L}\lambda)(X)=\nabla_X\lambda\qquad\text{and}\qquad \rmd_{A,L}(\omega\cdot\eta)=(\rmd_A\omega)\cdot\eta+(-)^{k}\omega\cdot(\rmd_{A,L}\eta),
\end{equation}
for all $X\in\Gamma(A)$, $\lambda\in\Gamma(L)$, $\omega\in\Omega^k(A)$ and $\eta\in\Omega^\bullet(A;L)$.
By the Leibniz rule in the RHS of Equation~\eqref{eq:de_Rham_differential:bis} one gets that $\rmd_{A,L}$ is a degree $1$ graded derivation of the graded $\Omega^\bullet(A)$-module $\Omega^\bullet(A;L)$ whose symbol is the degree $1$ graded derivation $\rmd_A$ of the graded algebra $\Omega^\bullet(A)$.
Finally, as it is easy to prove, notice that
\begin{equation*}
	\text{$\rmd_{A,L}$ is cohomological, i.e.~$\rmd_{A,L}^2=0$}\ \Longleftrightarrow\ \text{$(A;L)$ is a Jacobi algebroid.}
\end{equation*}
In this case, $\rmd_{A,L}$ is the \emph{Lie algebroid de Rham differential} of $A$ with values in $L$ and there exists a Cartan calculus on $\Omega^\bullet(A;L)$.
In addition to $\rmd_{A,L}$, the structural operations of this calculus are given, for any $X\in\Gamma(A)$, by the \emph{contraction} and the \emph{Lie derivative} along $X$
\begin{equation*}
	\iota_X\colon\Omega^\bullet(A;L)\to\Omega^{\bullet-1}(A;L),\qquad\scrL_X:=[\rmd_{A,L},\iota_X]\colon\Omega^\bullet(A;L)\to\Omega^\bullet(A;L),
\end{equation*}
where $[-,-]$ is the graded commutator.
These operations are related through the following identities
\begin{equation*}
	\label{eq:Cartan_identities}
	\begin{gathered}{}
		[\iota_X,\scrL_Y]=\iota_{[X,Y]},\quad [\scrL_X,\scrL_Y]=\scrL_{[X,Y]},%\\
		\quad
		[\rmd_{A,L},\scrL_X]=[\iota_X,\iota_Y]=0,\qquad\text{for all}\ X,Y\in\Gamma(A).
	\end{gathered}
\end{equation*}
\end{remark}

\begin{example}
	\label{ex:der-differential}
	Applying the construction from Remark~\ref{rem:deRham_differential} to $(DL;L)$, i.e.~the Jacobi algebroid formed by the gauge algebroid $DL$ and its tautological representation on $L$, one obtains the \emph{der-complex}~\cite{rubtsov1980cohomology} of \emph{$L$-valued Atiyah forms on $M$}, i.e.~the Lie algebroid de Rham complex of $DL$ with coefficients in $L$, that we will denote by $(\Omega_D^\bullet(L),\rmd_D)$.
	Then, as pointed out in~\cite[Prop.~3.3]{vitagliano2018dirac}, the \emph{$L$-valued presymplectic Atiyah forms}, i.e.~the $2$-cocycles in $(\Omega_D^\bullet(L),\rmd_D)$, describe exactly the pre-contact structures on the manifold $M$, with underlying line bundle $L\to M$.
	So, one gets the following identification
	\begin{equation}
		\label{eq:precontact_structures}
		\{\text{precontact structures with underlying $L\to M$}\}\overset{\sim}{\longrightarrow}\operatorname{MC}(\Omega^\bullet_D(L)[1],\rmd_D,0),
		%\left\{\begin{matrix}
			%\text{precontact structures on $M$}\\
			%\text{with underlying $L\to M$}
		%\end{matrix}\right\}\overset{\sim}{\longrightarrow}\operatorname{MC}(\Omega^\bullet_D(L)[1],\rmd_D,0).
	\end{equation}
	and the deformations of such a precontact structures are governed by the dgLa $(\Omega^\bullet_D(L)[1],\rmd_D,0)$.
\end{example}

%Gerstenhaber--Jacobi bracket
\begin{remark}[\textbf{The Gerstenhaber--Jacobi bracket}]
	\label{rem:GJ_bracket}
	Let $(A;L)$ be an almost Jacobi algebroid.
	Denote by $A^\dagger\to M$ the \emph{$L$-twisted dual of $A$} defined by $A^\dagger:=A^\ast\otimes L$, so that
	\begin{equation*}
		\Omega^\bullet(A^\dagger):=\Gamma(\wedge^\bullet (A\otimes L^\ast))\qquad\text{and}\qquad \Omega^\bullet(A^\dagger;L):=\Gamma(\wedge^\bullet (A\otimes L^\ast)\otimes L),
	\end{equation*}
	In particular, $\Omega^0(A^\dagger;L)=\Gamma(L)$ and $\Omega^1(A^\dagger;L)=\Gamma(A)$.
	Then there exists a unique degree $0$ graded skew-symmetric $\bbR$-linear map $[-,-]_{A,L}:\Omega^\bullet(A^\dagger;L)[1]\times\Omega^\bullet(A^\dagger;L)[1]\to\Omega^\bullet(A^\dagger;L)[1]$ such that
	\begin{itemize}
		\item it is a graded derivation of the graded $\Omega^\bullet(A^\dagger)$-module $\Omega^\bullet(A^\dagger;L)[1]$ in each entry separately,
		\item it extends both $[-,-]$ and $\nabla$ in the sense that, for all $u,v\in\Gamma(A)$ and $\lambda\in\Gamma(L)$,
		\begin{equation*}
			[u,v]_{A,L}=[u,v]\quad\text{and}\quad[u,\lambda]_{A,L}=\nabla_u\lambda.
		\end{equation*}
	\end{itemize}
	The latter is called the \emph{almost Gerstenhaber--Jacobi bracket} associated with $(A;L)$.
	Finally, notice that
		\begin{equation*}
			\text{$(\Omega^\bullet(A^\dagger;L)[1],[-,-]_{A,L})$ is a graded Lie algebra}\Longleftrightarrow\text{$(A;L)$ is a Jacobi algebroid.}
		\end{equation*}
	In this case, $[-,-]_{A,L}$ is called the \emph{Gerstenhaber-Jacobi bracket} associated with the Jacobi algebroid $(A;L)$, and makes the pair $(\Omega^\bullet(A^\dagger),\Omega^\bullet(A^\dagger;L))$ into a \emph{Gerstenhaber--Jacobi algebra}~\cite[Def.~A.5]{le2017bfv}.
\end{remark}

\begin{example}
	\label{ex:SJ_brackets}
	Applying the construction from Remark~\ref{rem:GJ_bracket} to $(DL;L)$, one recovers the \emph{Schouten--Jacobi bracket} $[-,-]_{\sf SJ}$ that makes $(\calD^\bullet L)[1]$ into a graded Lie algebra, where $\calD^\bullet L:=\Gamma(\wedge^\bullet(J^1L)^\ast\otimes L)$ is the graded space of first order multi-differential operators from $L$ to itself.
	Then, as first pointed out in~\cite[Theor.~1.b, (28), (29)]{grabowski2001jacobi} for trivial line bundles, the \emph{Jacobi structures} on the line bundle $L\to M$ identify with the Maurer--Cartan elements of the graded Lie algebra $(\calD^\bullet(L)[1],[-,-]_{\sf SJ})$ (see also~\cite[Lemma~ 2.8(2)]{le2018deformations}).
	So, one gets the following identification
	\begin{equation}
		\label{eq:Jacobi_structures}
		\{\text{Jacobi structures on $L\to M$}\}\overset{\sim}{\longrightarrow}\operatorname{MC}((\calD^\bullet L)[1],0,[-,-]_{\sf SJ}),
		%\left\{\begin{matrix}
		%\text{precontact structures on $M$}\\
		%\text{with underlying $L\to M$}
		%\end{matrix}\right\}\overset{\sim}{\longrightarrow}\operatorname{MC}(\Omega^\bullet_D(L)[1],\rmd_D,0).
	\end{equation}
	and the deformations of such a Jacobi structure $J$ are governed by the dgLa $(\calD^\bullet L)[1],\rmd_J,[-,-]_{\sf SJ})$, where $\rmd_J:=[J,-]_{\sf SJ}:\calD^\bullet L\to \calD^\bullet L$ is the Jacobi--Lichnerowicz differential determined by $J$.
\end{example}

\begin{definition}
	\label{def:Jacobi bialgebroid}
	A \emph{Jacobi bialgebroid} $(A,A^\dagger;L)$ over a manifold $M$ consists of a pair of Jacobi algebroids $(A;L)$ and $(A^\dagger;L)$ over $M$ such that the de Rham differential $\rmd_{A,L}$ is a degree $1$ graded derivation of the Gerstenhaber--Jacobi bracket $[-,-]_{A^\dagger,L}$, i.e.~the following graded Leibniz rule holds
	\begin{equation*}
		\rmd_{A,L}[\eta,\omega]_{A^\dagger,L}=[\rmd_{A,L}\eta,\omega]_{A^\dagger,L}+(-)^{|\eta|}[\eta,\rmd_{A,L}\omega]_{A^\dagger,L},
	\end{equation*}
	for all homogeneous $\eta,\omega\in\Omega^\bullet(A^\dagger;L)[1]\simeq\Omega^\bullet(A;L)[1]$, where $|\eta|$ denotes the degree of $\eta$.
\end{definition}

\begin{remark}
	As it is easy to see, Definition~\ref{def:Jacobi bialgebroid} is symmetric in $A$ and $A^\dagger$.
	Indeed, $A^{\dagger\dagger}\simeq A$ canonically, and $(A,A^\dagger;L)$ is a Jacobi bialgebroid iff $(A^\dagger,A;L)$ is a Jacobi bialgebroid.
\end{remark}

%Jacobi algebroids as Jacobi bialgebroids
\begin{remark}
	\label{rem:Jacobi_algebroid|Jacobi_bialgebroids}
	Each Jacobi algebroid $(A;L)$ identifies with the Jacobi bialgebroid $(A,A^\dagger;L)$, where $(A^\dagger;L)$ is the trivial Jacobi algebroid, i.e.~the Lie bracket on $\Gamma(A^\dagger)$ and $\nabla:A^\dagger\to DL$ are both zero.
\end{remark}

\subsection{Courant--Jacobi Algebroids}
\label{sec:CJ_algebroids}
Following Grabowski~\cite{Grabowski2013gradedcontactCourant}, we recall the notion of Courant--Jacobi algebroid slightly generalizing the similar notion earlier introduced by Grabowski and Marmo~\cite[Def.~1]{grabowski2002graded}

\begin{definition}
	\label{def:Courant-Jacobi_algebroid}
	A \emph{Courant--Jacobi algebroid} on a manifold $M$ consists of a vector bundle $E\to M$ and a line bundle $L\to M$ equipped with:
	\begin{itemize}
		\item a non-degenerate $L$-valued symmetric $C^\infty(M)$-bilinear product $\ldab-,-\rdab:E\otimes E\to L$,
		\item a \emph{Loday bracket}~\cite{yksPoisson1996} on $\Gamma(E)$, i.e.~a (not necessarily skew-symmetric) $\bbR$-bilinear map $\ldsb-,-\rdsb:\Gamma(E)\times\Gamma(E)\to\Gamma(E)$ satisfying the Jacobi identity written in Leibniz form
		\begin{equation}
			\label{eq:def:Courant-Jacobi_algebroids:Jacobi_identity}
			\ldsb e_1,\ldsb e_2,e_3\rdsb\rdsb=\ldsb\ldsb e_1,e_2\rdsb, e_3\rdsb+\ldsb e_2,\ldsb e_1,e_3\rdsb\rdsb,\qquad\text{for all}\ e_1,e_2,e_3\in\Gamma(E),
		\end{equation}
		\item a VB morphism $\nabla:E\to DL,\ e\mapsto\nabla_e,$ over $\id_M$ mapping sections of $E$ to derivations of $L$,
	\end{itemize}
	such that the following compatibility conditions are satisfied, for all $e_1,e_2,e_3\in\Gamma(E)$,
	\begin{align}
		\ldab\ldsb e_1,e_2\rdsb,e_2\rdab&=\ldab e_1,\ldsb e_2,e_2\rdsb\rdab,\label{eq:def:Courant-Jacobi_algebroids:product}\\
		\nabla_{e_1}\ldab e_2,e_2\rdab&=2\ldab\ldsb e_1,e_2\rdsb,e_2\rdab\label{eq:def:Courant-Jacobi_algebroids:nabla}.
	\end{align}
\end{definition}

\begin{remark}
	\label{rem:contac_Courant_algebroid}
	Courant--Jacobi algebroids are equivalent to Grabowski's contact Courant algebroids (see Definition 12.1 and Theorem 12.1 in~\cite{Grabowski2013gradedcontactCourant}).
	Moreover, Definition~\ref{def:Courant-Jacobi_algebroid} can be rephrased replacing the Loday bracket $\ldsb-,-\rdsb$ on $\Gamma(E)$ with its skew-symmetrization (cf., e.g.,~\cite{das2021contact}).
\end{remark}

The next proposition is the Jacobi version of an analogous result valid for Courant algebroids (cf, e.g., \cite[Theorem 2.1]{yks2005allthat}) and its proof can be found in~\cite[Section~12]{Grabowski2013gradedcontactCourant}.

\begin{proposition}
	\label{prop:Courant-Jacobi_algebroid}
	In any Courant--Jacobi algebroid $(E;L)$ the following additional identities hold:
	\begin{enumerate}[label=(\arabic*)]
		\item
		\label{enumitem:prop:Courant-Jacobi_algebroid:flat}
		the VB morphism $\nabla:E\to DL$ induces a morphism of Loday algebras from $\Gamma(E)$ to $\calD L$, i.e.
		\begin{equation}
			\label{eq:flat}
			\nabla_{\ldsb e_1,e_2\rdsb}=[\nabla_{e_1},\nabla_{e_2}],\quad\text{for all $e_1,e_2\in\Gamma(E)$}
		\end{equation}
		\item
		\label{enumitem:prop:Courant-Jacobi_algebroid:Leibniz}
		the Loday bracket $\ldsb-,-\rdsb$ satisfies the following Leibniz rule
	\begin{equation}
		\label{eq:Leibniz}
		\ldsb e_1,fe_2\rdsb=(\sigma(\nabla_{e_1})f)e_2+f\ldsb e_1,e_2\rdsb,\quad\text{for all $e_1,e_2\in\Gamma(E)$ and $f\in C^\infty(M)$}.
	\end{equation}
	In other words, the Loday algebra derivation $Z_e:=\ldsb e,-\rdsb:\Gamma(E)\to\Gamma(E)$ is also a derivation of the vector bundle $E$, i.e.~$Z_e\in\calD E$, and $\sigma(Z_e)=\sigma(\nabla_e)$, for any $e\in\Gamma(E)$
	\end{enumerate}
\end{proposition}

\begin{remark}
	By Proposition~\ref{prop:Courant-Jacobi_algebroid}\ref{enumitem:prop:Courant-Jacobi_algebroid:flat}, if $L=\bbR_M$, then Definition~\ref{def:Courant-Jacobi_algebroid} reduces to Grabowski and Marmo's notion of Courant--Jacobi structure~\cite[Def.~1]{grabowski2002graded} and to the equivalent notion of generalized Courant algebroid by Nunes da Costa and Clemente-Gallardo~\cite[Def.~3.4]{nunesdacosta2004dirac}.
	%\textcolor{red}{However, there are interesting examples of Courant--Jacobi algebroids with non-trivial underlying line bundles. For instance\ldots}
\end{remark}

\begin{remark}
	\label{rem:Courant_algebroid}
	Courant--Jacobi algebroids generalizes the notion of Courant algebroids.
	Indeed, each Courant algebroid $E$, with structure maps $(\ldab-,-\rdab,\ldsb-,-\rdsb,\rho)$ identifies with the Courant--Jacobi algebroid $(E;\bbR_M)$, with structure maps $(\ldab-,-\rdab,\ldsb-,-\rdsb,\nabla)$, where $\nabla=\rho\colon E\to TM\subset D\bbR_M$.
\end{remark}

\begin{remark}
	In a Courant--Jacobi algebroid $(E;L)$, for any subbundle $V\subset E$, we denote by $V^\perp$ its orthogonal wrt $\ldab-,-\rdab$.
	The subbundle $V$ is \emph{isotropic} (resp.~\emph{coisotropic}) if $V\subset V^\perp$ (resp.~$V^\perp\subset V$) and so, in particular, $\rank V\leq\frac{1}{2}\rank E$ (resp.~$\rank V\geq\frac{1}{2}\rank E$).
	Moreover, $V$ is \emph{Lagrangian} if $V=V^\perp$ and so, in particular, $\rank V=\frac{1}{2}\rank E$.
\end{remark}

%\color{orange}
%\begin{example}
%	\label{ex:omni-Lie_algebroid}
%	For any line bundle $L\to M$, the \emph{omni-Lie algebroid} $\bbD L:=DL\oplus J^1 L$ is the prototypical example of a Courant--Jacobi algebroid.
%	It is equipped with the pairing $\ldab-,-\rdab$, Loday bracket $\ldsb-,-\rdsb$ and representation $\nabla$, which are defined on sections $\delta_i+\alpha_i\in\Gamma(\bbD L)$, with $i=1,2$, as follows:
%	\begin{equation*}
%	\ldab \delta_1+\alpha_1,\delta_2+\alpha_2\rdab=\iota(\delta_2)\alpha_1+\iota(\delta_1)\alpha_2,\quad\nabla_{\delta_1+\alpha_1}=\delta_1,\quad\ldsb \delta_1+\alpha_1,\delta_2+\alpha_2\rdsb=[\delta_1,\delta_2]+\calL(\delta_1)\alpha_2-\iota(\delta_2)\rmd_D\alpha_1.
%	\end{equation*}
%	where we have used the standard Cartan calculus on the der-complex $(\Omega_D^\bullet(L),\rmd_D)$ \textcolor{red}{(add a reference)}.
%\end{example}

%The double of a Jacobi bialgebroids
\begin{example}%[The double of a Jacobi bialgebroid]
	\label{ex:Jacobi_bialgebroid}
	Let $(A,A^\dagger;L)$ be a Jacobi bialgebroid.
	Denote by $[-,-]_A$ and $\nabla^A$ (resp.~$[-,-]_{A^\dagger}$ and $\nabla^{A^\dagger}$) the structural operations of the Jacobi algebroid $(A;L)$ (resp.~$(A^\dagger;L)$).
	Then the \emph{double of $(A,A^\dagger;L)$} is the Courant--Jacobi algebroid $(A\oplus A^\dagger;L)$ whose structural operations are the following:
	\begin{itemize}
		\item the non-degenerate $L$-valued symmetric product $\ldab-,-\rdab$ on $A\oplus A^\dagger$ is given by
		\begin{equation*}
			%\ldab u+\alpha,v+\beta\rdab=\frac{1}{2}(\alpha(v)+\beta(u)),
			\ldab u+\alpha,v+\beta\rdab=\alpha(v)+\beta(u),
		\end{equation*}
		\item the Loday bracket $\ldsb-,-\rdsb$ on $\Gamma(A\oplus A^\dagger)=\Gamma(A)\oplus\Gamma(A^\dagger)$ is given by
		\begin{align*}
			\ldsb u+\alpha,v+\beta\rdsb=([u,v]_A-\iota_\beta\rmd_{A^\dagger,L}u+\scrL_\alpha v+)+(\scrL_u \beta-\iota_v\rmd_{A,L}\alpha+[\alpha,\beta]_{A^\dagger}),
		\end{align*}
		%\begin{align*}
			%\ldsb u+\alpha,v+\beta\rdsb&=[u,v]_A+\scrL_\alpha v-\scrL_\beta u-\textcolor{red}{\frac{1}{4}}\rmd_{A^\dagger,L}(\alpha(v))+\textcolor{red}{\frac{3}{4}}\rmd_{A,L}(\beta(u))\\
			%&\phantom{=\ }+[\alpha,\beta]_{A^\dagger}+\scrL_u\beta-\scrL_v\alpha+\textcolor{red}{\frac{3}{4}}\rmd_{A,L}(\alpha(v))-\textcolor{red}{\frac{1}{4}}\rmd_{A,L}(\beta(u)),
		%\end{align*}
		\item the VB morphism $\nabla\colon A\oplus A^\dagger\to DL$ is given by $			\nabla_{u+\alpha}=\nabla^A_u+\nabla^{A^\dagger}_\alpha$, for all $u+\alpha\in\Gamma(A\oplus A^\dagger)$.
	\end{itemize}
	In view of Remark~\ref{rem:Jacobi_algebroid|Jacobi_bialgebroids}, this construction produces, in particular, the double of a Jacobi algebroid $(A;L)$.
	For instance, if applied to $(DL;L)$, i.e.~the Jacobi algebroid formed by the gauge algebroid with its tautological representation (cf.~Example~\ref{ex:tautological_representation_DL}), one gets that the double of $(DL;L)$ is nothing but the \emph{omni-Lie algebroid $\bbD L=DL\oplus J^1L$} of the line bundle $L\to M$.
	The omni-Lie algebroid was introduced by Chen and Liu~\cite{Chen2010omniLie} as a generalization of the Courant algebroid $\calE^1(M)$ by Wade~\cite{wade2000conformal}.
%	 is equipped with the pairing $\ldab-,-\rdab$, Loday bracket $\ldsb-,-\rdsb$ and representation $\nabla$, which are defined on sections $\delta_i+\alpha_i\in\Gamma(\bbD L)$, with $i=1,2$, as follows:
%	\begin{equation*}
%		\ldab \delta_1+\alpha_1,\delta_2+\alpha_2\rdab=\iota(\delta_2)\alpha_1+\iota(\delta_1)\alpha_2,\quad\nabla_{\delta_1+\alpha_1}=\delta_1,\quad\ldsb \delta_1+\alpha_1,\delta_2+\alpha_2\rdsb=[\delta_1,\delta_2]+\calL(\delta_1)\alpha_2-\iota(\delta_2)\rmd_D\alpha_1.
%	\end{equation*}
%	where we have used the standard Cartan calculus on the der-complex $(\Omega_D^\bullet(L),\rmd_D)$ \textcolor{red}{(add a reference)}.
\end{example}
%\color{black}

\subsection{Dirac--Jacobi Structures}
\label{sec:DJ_structures}

Pursuing the line bundle approach to Dirac--Jacobi geometry, we introduce here the notion of Courant--Jacobi algebroid which slightly generalizes the similar notion earlier introduced by Grabowski and Marmo~\cite[Section~4]{grabowski2002graded} (see also~\cite{nunesdacosta2004dirac}).

\begin{definition}
	\label{def:Dirac-Jacobi_structure}
	Let $(E;L)$ be a Courant--Jacobi algebroid.
	A \emph{Dirac--Jacobi structure in $(E;L)$}, also called an \emph{$(E;L)$-Dirac--Jacobi structure on $M$}, is a subbundle $A\subset E$ such that
	\begin{itemize}
		\item it is \emph{Lagrangian} wrt the product, i.e.~$A=A^\perp$, and
		\item it is \emph{involutive} wrt the Loday bracket, i.e.~$\ldsb\Gamma(A),\Gamma(A)\rdsb\subset\Gamma(A)$.
	\end{itemize}
	Removing involutivity, a plain Lagrangian subbundle $A\subset E$ is called an \emph{almost Dirac--Jacobi structure}.
\end{definition}

\begin{remark}
	Assume that the underlying line bundle is the trivial one, i.e.~$L=\bbR_M$.
	Then Definition~\ref{def:Dirac-Jacobi_structure} reduces to Grabowski and Marmo's notion of a Dirac--Jacobi structure~\cite[Section~4]{grabowski2002graded} (see also~\cite[Definition 3.8.]{nunesdacosta2004dirac} where the authors consider the same notion in a generalized Lie bialgebroid).
	%\textcolor{red}{However, there are interesting examples of Courant--Jacobi algebroids with non-trivial underlying line bundles. For instance\ldots}
\end{remark}

\begin{remark}
	\label{rem:Dirac_structures}
	Dirac--Jacobi structures can be seen as a generalization of Dirac structures.
	Indeed, we have the following.
	Let $E$ be a Courant algebroid and let $A\subset E$ be a vector subbundle.
	Then $A$ is an (almost) Dirac structure in the Courant algebroid $E$ if and only if $A$ is an (almost) Dirac--Jacobi structure in the Courant--Jacobi algebroid $(E,\bbR_M)$ corresponding to $E$ by Remark~\ref{rem:Courant_algebroid}.
	\end{remark}

A measure of how much an almost Dirac--Jacobi structure fails to be involutive is provided by its associated Courant--Jacobi tensor as we recall below.
\begin{proposition}[{\cite[Remark~4.5]{vitagliano2018dirac}}]
	\label{prop:Courant_tensor}
	Let $(E;L)$ be a Courant--Jacobi algebroid.
	For each almost Dirac--Jacobi structure $A\subset E$, its \emph{Courant--Jacobi tensor} $\Upsilon_A\in\Gamma(\wedge^3 A^\ast\otimes L)$ is well-defined by
	\begin{equation}
		\label{eq:prop:Courant_tensor}
		\Upsilon_A(u_1,u_2,u_3)=\ldab\ldsb u_1,u_2\rdsb, u_3\rdab,
	\end{equation}
	for all $u_1,u_2,u_3\in\Gamma(A)$, and measures how much $A$ fails to be involutive.
	Indeed, one gets that
	\begin{center}
		$A$ is Dirac--Jacobi $\Longleftrightarrow$ $A$ is involutive $\Longleftrightarrow$ $\Upsilon_A=0$.
	\end{center}
\end{proposition}

\begin{proof}
	In Definition~\ref{def:Courant-Jacobi_algebroid}, the identities~\eqref{eq:def:Courant-Jacobi_algebroids:product} and~\eqref{eq:def:Courant-Jacobi_algebroids:nabla} can be rephrased as
	\begin{equation}
		\label{eq:proof:prop:Courant_tensor}
		\nabla_e\ldab e_1,e_2\rdab=\ldab\ldsb e,e_1\rdsb,e_2\rdab+\ldab e_1,\ldsb e,e_2\rdsb\rdab\quad\text{and}\quad\nabla_e\ldab e_1,e_2\rdab=\ldab e,\ldsb e_1,e_2\rdsb+\ldsb e_2,e_1\rdsb\rdab 
	\end{equation}
	for all $e,e_1,e_2\in\Gamma(E)$.
	Since $A$ is Lagrangian, the latter imply that the RHS of Equation~\eqref{eq:prop:Courant_tensor} is skew-symmetric and $C^\infty(M)$ linear in $u_1,u_2$ and $u_3\in\Gamma(A)$.
	The rest of the proof is straightforward.
\end{proof}

\begin{example}
	\label{ex:DJ:omniLie}
	We show some special classes of (almost) Dirac--Jacobi structures in the omni-Lie algebroid $\bbD L:=DL\oplus J^1L$ of a line bundle $L\to M$ (see Example~\ref{ex:Jacobi_bialgebroid}).
	%standard Courant algebroid $(\bbT M,\ldab-,-\rdab,\ldsb-,-\rdsb,\text{pr}_{TM})$:
	\begin{enumerate}[label=(\arabic*)]
		\item
		\label{enum:ex:DJ:omniLie:1}
		$L$-valued Atiyah $2$-forms identify with almost Dirac--Jacobi structures transverse to $J^1L$
		\begin{equation*}
		\begin{aligned}
		\Omega^2_D(L)&\overset{\sim}{\longrightarrow}\{A\subset\bbD L\ \text{almost Dirac--Jacobi structure}\mid A\pitchfork J^1 L\},\\
		\varpi&\longmapsto\gr(\varpi):=\{\delta+\iota_\delta\varpi \mid \delta\in DL\}.
		\end{aligned}
		\end{equation*}
		so that $\varpi\in\Omega^2_D(L)$ is closed if and only if $\gr(\varpi)\subset\bbD L$ is Dirac--Jacobi.
		\item
		\label{enum:ex:DJ:omniLie:2}
		Bi-derivations of $L$ identify with almost Dirac--Jacobi structures transverse to $DL$
		\begin{equation*}
		%\label{eq:bivectors=TM-transverse_Lagrangians}
		\begin{aligned}
		\calD^2L&\overset{\sim}{\longrightarrow}\{ A\subset\bbD L\ \text{almost Dirac--Jacobi structure}\mid A\pitchfork DL\},\\
		J&\longmapsto\gr(J):=\{\iota_\alpha J+\alpha\mid\alpha\in J^1L\}.
		\end{aligned}
		\end{equation*}
		so that $J\in\calD^2 L$ is Jacobi if and only if $\gr(J)\subset\bbD L$ is Dirac--Jacobi.
		\item
		\label{enum:ex:DJ:omniLie:3}
		subbundles of $DL$ identify with almost Dirac--Jacobi structures in $\bbD L$ s.t.~$ A=\pr_D A{{}\oplus{}}{\pr_J A}$,
		\begin{equation*}
		\begin{aligned}
		\{\text{subbundles of $DL$}\}&\overset{\sim}{\longrightarrow}\{ A\subset\bbD L\ \text{almost Dirac--Jacobi structure}\mid A=\pr_D A{{}\oplus{}}{\pr_J A}\},\\
		D&\longmapsto D\oplus D^\circ.
		%\Gamma(Gr_{k}(M))&\overset{\sim}{\longrightarrow}\{ A\subset\bbT M\ \text{almost Dirac structure}\mid A=\pr_{TM} A{{}\oplus{}}{\pr_{T^\ast M} A}\},\\
		%D&\longmapsto D\oplus D^{0}.
		\end{aligned}
		\end{equation*}
		so that $D\subset DL$ is involutive if and only if $D\oplus D^\circ\subset\bbD L$ is Dirac--Jacobi.
		Above, $\pr_D:\bbD L\to DL$ and $\pr_J:\bbD L\to J^1L$ denote the standard projections. 
	\end{enumerate}
\end{example}

Because of the involutivity condition, each Dirac--Jacobi structure naturally becomes a Jacobi algebroid as recalled in the following.

\begin{proposition}
	\label{prop:almost_DJ_structure|Jacobi_algebroid}
	Let $A$ be an almost Dirac--Jacobi structure in a Courant--Jacobi algebroid $(E;L)$.
	\begin{enumerate}[label=\arabic*)]
		\item Each almost Dirac--Jacobi structure $B$ in $(E;L)$ which is transverse to $A$ determines an almost Jacobi algebroid structure on $(A;L)$ with structure maps (depending on $B$) given by
		\begin{equation*}
			[-,-]_A=\left.\pr_A\circ\ldsb-,-\rdsb\right|_{\Gamma(A)\times\Gamma(A)},\qquad\rho^A=\sigma\circ\nabla|_A,\qquad\nabla^A=\nabla|_A,
		\end{equation*}
		where $\pr_A:E\to A$ denotes the projection with respect to the splitting $E=A\oplus B$.
		\item If $A$ is a Dirac--Jacobi structure, then $(A;L)$ becomes a Jacobi algebroid with structure maps
		\begin{equation*}
			[-,-]_A=\left.\ldsb-,-\rdsb\right|_{\Gamma(A)\times\Gamma(A)},\qquad\rho^A=\sigma\circ\nabla|_A,\qquad\nabla^A=\nabla|_A.
		\end{equation*}
	\end{enumerate}
\end{proposition}

\section{On the Contact Geometry of the Shifted First Jet Bundle of a Line Bundle}
\label{sec:degree_2_NQ-manifolds}
%\label{sec:split-CJ_algebroids}

In developing the deformation theory of Dirac--Jacobi structure within a fixed Courant--Jacobi algebroid, an important ingredient is the identification~\cite{Grabowski2013gradedcontactCourant} of Courant--Jacobi algebroids with degree $2$ contact $\bbN Q$ manifolds (or equivalently with degree $2$ symplectic $\bbN Q$ manifolds that are homogeneous w.r.t.~an additional principal $\bbR^\times$ bundle structure).
Precisely, we will need to spell out this identification in the case of a split Courant--Jacobi algebroid (see Section~\ref{sec:splitCJalgbd:MC_elements}).
As a preliminary, in this section, we describe in detail the canonical contact structure of the shifted first jet bundle $J^1[2]L_\calQ$ of a line bundle generated in degree $0$ over an $\bbN$ manifold $\calQ$ (see Proposition~\ref{prop:Cartan_contact_structure}).
Actually, for the aims of the paper, from Section~\ref{sec:shifted_1st_jet_bundle_N-mfld:degree_1} on we will focus our attention only on the case of an $\bbN$ manifold of degree $1$, i.e.~$\calQ=A[1]$, for some vector bundle $A\to M$.
In particular, in this setting, Theorem~\ref{theor:contact_Legendre} constructs a contact version of the Legendre transform.

\subsection{The Canonical Contact Structure of the First Jet Bundle}
\label{sec:shifted_1st_jet_bundle_N-mfld}

Given a line bundle $L_\calQ\to\calQ$ generated in degree $0$ over an $\bbN$-manifold $\calQ$, the shifted total space $J^1[2]L_\calQ$ of the first jet bundle $J^1L_\calQ\overset{\pi}{\to}\calQ$ is endowed with a canonical contact structure, called the \emph{Cartan contact structure}.
Below we recall its construction.

%local coordinate system
Let $z^\alpha$ be a local coordinate chart for $\calQ$ and $\mu$ be a local frame of $L_\calQ\to\calQ$.
Then local fiberwise linear coordinates on $J^1[2]L_\calQ$ are given by the associated momenta $p,p_\alpha$, with $|p|=2$ and $|p_\alpha|=-|z_\alpha|+2$.
For each index $\alpha$ define the local vector field $D_\alpha$ on $J^1[2]L_\calQ$ by setting
\begin{equation*}
D_\alpha:=\frac{\partial}{\partial z^\alpha}+p_\alpha\frac{\partial}{\partial p}.
\end{equation*}
Then the vector fields $\frac{\partial}{\partial p},\ \frac{\partial}{\partial p_\alpha},\ D_\alpha$ form a local frame on $J^1[2]L_\calQ$, with degrees $|\frac{\partial}{\partial p}|=-2$, $|\frac{\partial}{\partial p_\alpha}|=|z^\alpha|-2$ and $|D_\alpha|=-|z^\alpha|$, and whose only non-trivial commutation relations are the following
\begin{equation}
\label{eq:commutation_relations}
[D_\alpha,\frac{\partial}{\partial p_\beta}]=-(-)^{|z^\alpha|}[\frac{\partial}{\partial p_\beta},D_\alpha]=-(-)^{|z^\alpha|}\delta^\alpha_\beta\frac{\partial}{\partial p}.
\end{equation}
As dual local coframe one gets $\rmd p-p_\alpha\rmd z^\alpha,\ \rmd p_\alpha,\ \rmd z^\alpha$, with degrees $|\rmd p-p_\alpha\rmd z^\alpha|=3$, $|\rmd p_\alpha|=3-|z^\alpha|$, $|\rmd z^\alpha|=|z^\alpha|+1$.

%Cartan contact structure
\begin{proposition}
	\label{prop:Cartan_contact_structure}
	On the shifted first jet bundle $J^1[2]L_\calQ$ there exists a canonical contact structure, called the \emph{Cartan contact structure}.
	It is given by the \emph{degree $3$} contact $1$-form $\theta$ on $J^1[2]L_\calQ$, with coefficients in the pull-back line bundle $\calL:=\pi^\ast L_\calQ\to J^1[2]L_\calQ$, which locally looks like
	\begin{equation}
		\label{eq:prop:Cartan_contact_form}
		\theta=(\rmd p-p_\alpha\rmd z^\alpha)\otimes\pi^\ast\mu.
	\end{equation}
	Equivalently, it is given by the contact distribution $\calH:=\ker\theta\subset T(J^1[2]L_\calQ)$ which locally looks like
	\begin{equation}
		\label{eq:prop:Cartan_contact_distribution}
	\calH=\left\langle D_\alpha:=\frac{\partial}{\partial z^\alpha}+p_\alpha\frac{\partial}{\partial p},\ \frac{\partial}{\partial p_\alpha}\right\rangle.
	\end{equation} 
\end{proposition}

\begin{proof}
	It is a straightforward computation in the local adapted coordinates $(z^\alpha,p,p_\alpha)$.
\end{proof}

%The Cartan contact structure on $J^1[2]L_\calQ$ is given by the contact distribution $\calH\subset T(J^1[2]L_\calQ)$ which locally looks like
%\begin{equation*}
%\calH=\left\langle D_\alpha:=\frac{\partial}{\partial z^\alpha}+p_\alpha\frac{\partial}{\partial p},\ \frac{\partial}{\partial p_\alpha}\right\rangle
%\end{equation*}
%In other terms, $\calH$ is the kernel of the \emph{degree $3$} contact $1$-form $\theta$ on $J^1[2]L_\calQ$ with coefficients in the pull-back line bundle $\calL:=\pi^\ast L_\calQ\to J^1[2]L_\calQ$ which locally looks like
%\begin{equation*}
%\theta=(\rmd p-p_\alpha\rmd z^\alpha)\otimes\pi^\ast\mu.
%\end{equation*}
In the line bundle approach to contact structures, each section of the underline line bundle determines the associated Reeb vector field.
This justifies the following.
\begin{definition}
	\label{def:Reeb_vf}
	For each homogeneous section $\lambda\in\Gamma(\calL)$, the associated (degree $|\lambda|-2$) \emph{Reeb vector field} $\calX_\lambda\in\frakX(J^1[2]L_\calQ)$ is uniquely determined by the following two conditions
	\begin{equation}
		\label{eq:def:Reeb_vf}
		\iota_{\calX_\lambda}\theta=(-)^{|\lambda|}\lambda\quad\text{and}\quad[\calX_\lambda,\Gamma(\calH)]\subset\Gamma(\calH).
	\end{equation}
\end{definition}
\begin{remark}
	\label{rem:Reeb_vf:coordinate_expression}
	Assume that locally $\lambda=f\pi^\ast\mu$, for some homogeneous function $f=f(z^\alpha,p,p_\alpha)$.
	Then, using Equations~\eqref{eq:prop:Cartan_contact_form} and the commutation relations~\eqref{eq:commutation_relations}, from Definition~\ref{def:Reeb_vf} it easily follows that
	\begin{equation}
		\label{eq:rem:Reeb_vf:coordinate_expression}
		\calX_{\lambda}=f\frac{\partial}{\partial p}+(-)^{|f||z^\alpha|}\left((D_\alpha f)\frac{\partial}{\partial p_\alpha}-(-)^{|z^\alpha|}\frac{\partial f}{\partial p_\alpha}D_\alpha\right).
	\end{equation}
\end{remark}

The identification of contact structures with non-degenerate Jacobi structures justifies the following definition of the canonical Jacobi structure on $\calL:=\pi^\ast L_\calQ\to J^1[2]L_\calQ$.
\begin{definition}
	The canonical contact structure on $J^1[2]L_\calQ$ determines the \emph{associated degree $-2$ non-degenerate Jacobi structure} $\{-,-\}$ on $\calL\to J^1[2]L_\calQ$ which is defined as follows
\begin{equation}
\label{eq:canonical_Jacobi_structure}
\{\lambda_1,\lambda_2\}=(-)^{|\lambda_1|+|\lambda_2|}\iota_{[\calX_{\lambda_1},\calX_{\lambda_2}]}\theta,
\end{equation}
for all homogeneous sections $\lambda_1,\lambda_2$ of the pull-back line bundle $\calL:=\pi^\ast L_\calQ\to J^1[2]L_\calQ$.
\end{definition}
\begin{remark}
	Assume that locally $\lambda_1=f_1\pi^\ast\mu$ and  $\lambda_2=f_2\pi^\ast\mu$, for some homogeneous functions $f_1=f_1(z^\alpha,p,p_\alpha)$ and $f_2=f_2(z^\alpha,p,p_\alpha)$.
	Now, in view of Equation~\eqref{eq:rem:Reeb_vf:coordinate_expression} and the commutation relations~\eqref{eq:commutation_relations}, one can compute
	\begin{equation*}
	[\calX_{\lambda_1},\calX_{\lambda_2}]=\left(\left(f_1\frac{\partial f_2}{\partial p}-\frac{\partial f_1}{\partial p}f_2\right)+(-)^{|f_1||z^\alpha|}\left(D_\alpha f_1\frac{\partial f_2}{\partial p_\alpha}-(-)^{|z^\alpha|}\frac{\partial f_1}{\partial p_\alpha}D_\alpha f_2\right)\right)\frac{\partial}{\partial p}+\calH.
	\end{equation*}
	Plugging the latter into Equation~\eqref{eq:canonical_Jacobi_structure} and also using of the fact that $\iota_{fX}=(-)^{|f|}\iota_X$, for all homogeneous $f\in C^\infty(J^1[2]L_\calQ)$ and $X\in\frakX(J^1[2]L_\calQ)$, one immediately gets that
	\begin{equation}
	\label{eq:canonical_Jacobi_structure:local_coordinates}
	\{\lambda_1,\lambda_2\}=\left(\left(f_1\frac{\partial f_2}{\partial p}-\frac{\partial f_1}{\partial p}f_2\right)+(-)^{|f_1||z^\alpha|}\left(D_\alpha f_1\frac{\partial f_2}{\partial p_\alpha}-(-)^{|z^\alpha|}\frac{\partial f_1}{\partial p_\alpha}D_\alpha f_2\right)\right)\otimes\pi^\ast\mu.
	\end{equation}
%	Above we use the fact that $\iota_{fX}=(-)^{|f|}\iota_X$, for all homogeneous $f\in C^\infty(J^1[2]L_\calQ)$ and $X\in\frakX(J^1[2]L_\calQ)$.
\end{remark}

\begin{remark}
	\label{rem:Hamiltonian_lift}
	Let us keep considering an $\bbN$-manifold $\calQ$ and a line bundle $L_\calQ\to\calQ$ generated in degree $0$.
	Fix a chart of local coordinate $z^\alpha$ on $\calQ$ and a local frame $\mu$ of $L_\calQ\to\calQ$.
	Then the local frame  $\mathbbm{1},\Delta_\alpha$ of the vector bundle $DL_\calQ\to\calQ$ is defined by setting, for each local function $f$ on $\calQ$,
	\begin{equation*}
		\mathbbm{1}(f\mu)=f\mu\quad\text{and}\quad\Delta_\alpha(f\mu)=\frac{\partial f}{\partial z^\alpha}\mu.
	\end{equation*}
	In particular, they commute and their degrees are $|\mathbbm{1}|=0$ and $|\Delta_\alpha|=-|z^\alpha|$.
	Therefore, an arbitrary derivation $\delta\in\calD L_\calQ\equiv\Gamma(DL_\calQ)$ of the line bundle $L_\calQ\to\calQ$ locally looks like
	\begin{equation*}
		\delta=f\mathbbm{1}+f^\alpha\Delta_\alpha,
	\end{equation*}
	for arbitrary local functions $f,f^\alpha$ on $\calQ$, and so it can be identified with the fiberwise-linear section $h_\delta$ of $\calL:=\pi^\ast L_\calQ\to J^1[2]L_\calQ$ locally given by $h_\delta=(fp+f^\alpha p_\alpha)\pi^\ast\mu$.
	This establishes a $C^\infty(\calQ)$-module embedding (called teh \emph{Hamiltonian lift})
	\begin{equation}
		\label{eq:LB-sections_from_derivations}
		\calD L_\calQ\longrightarrow\Gamma(\calL),\ \delta\longmapsto h_\delta.
	\end{equation}
	Similarly, each section $\lambda$ of  $L_\calQ\to\calQ$ identifies with the fiberwise constant section $\pi^\ast\lambda$ of the pull-back line bundle $\calL:=\pi^\ast L_\calQ\to J^1[2]L_\calQ$.
	This establishes a $C^\infty(\calQ)$-module embedding
	\begin{equation}
		\label{eq:LB-sections_from_pull-back}
		\Gamma(L_\calQ)\longrightarrow\Gamma(\calL),\ \lambda\longmapsto\pi^\ast\lambda.
	\end{equation}
\end{remark}

The $C^\infty(\calQ)$-module embeddings introduced in Remark~\ref{rem:Hamiltonian_lift} have the following property.

\begin{lemma}
	\label{lem:Hamiltonian_lift}
	For each $\delta,\delta^\prime\in\calD L_\calQ$ and $\lambda,\lambda^\prime\in\Gamma(L_\calQ)$, the following identities hold:
	\begin{equation*}
	\{h_\delta,h_{\delta^\prime}\}=-h_{[\delta,\delta^\prime]},\quad\{h_\delta,\pi^\ast\lambda\}=-\pi^\ast(\delta\lambda),\quad\{\pi^\ast\lambda,\pi^\ast\lambda^\prime\}=0.
	\end{equation*}
\end{lemma}
\begin{proof}
	It is an easy computation that uses the definition of the $C^\infty(\calQ)$-module embeddings~\eqref{eq:LB-sections_from_derivations} and~\eqref{eq:LB-sections_from_pull-back} together with the expression~\eqref{eq:canonical_Jacobi_structure:local_coordinates} of the Jacobi structure $\{-,-\}$ in Darboux coordinates.
\end{proof}

\subsection{The Double Vector Bundle Structure of the First Jet Bundle}
\label{sec:shifted_1st_jet_bundle_N-mfld:degree_1}

In Section~\ref{sec:shifted_1st_jet_bundle_N-mfld} we introduced the canonical contact structure on the shifted 1st jet bundle $J^1[2]L_\calQ$, where $L_\calQ\to\calQ$ is a line bundle generated in degree $0$ over an $\bbN$-manifold $\calQ$.
From now on we will focus our attention on the special case when the $\bbN$-manifold $\calQ$ is of degree $1$, i.e.~$\calQ=A[1]$ for some vector bundle $A\to M$.
Further, in this case, the line bundle $L_\calQ\to\calQ$ will be the pull-back of a line bundle $L\to M$ along $A[1]\to M$.

Let us set up our framework.
We consider a vector bundle $A\overset{p}{\to}M$ and a line bundle $L\to M$.
Then we can introduce the degree $1$ $\bbN$-manifold $A[1]$ and the line bundle $L_A:=p^\ast L=A[1]\times_M L\to A[1]$, generated in degree $0$, over $A[1]$, so that
\begin{equation*}
C^\infty(A[1])=\Omega^\bullet(A)\qquad \text{and}\qquad \Gamma(L_A)=C^\infty(A[1])\otimes_{C^\infty(M)}\Gamma(L)=\Omega^\bullet(A;L).
\end{equation*}
Going further, we also consider $J^1[2]L_A$, the shifted total space of the first jet bundle $J^1L_A\overset{\pi}{\to}A[1]$ and the line bundle $\calL:=\pi^\ast L_A\to J^1[2]L_A$, generated in degree $0$, over $J^1[2]L_A$.
\begin{remark}
	\label{rem:vaintrob:Jacobi}
	For future reference, let us point out here that mapping each almost Jacobi algebroid $(A;L)$ to its associated differential $\rmd_{A;L}$, as in Remark~\ref{rem:deRham_differential}, one gets the following identification
\begin{equation}
	\label{eq:vaintrob:almost}
	\{\text{almost Jacobi algebroid structures on $(A;L)$}\}\overset{\sim}{\longrightarrow}(\calD L_A)^1
\end{equation}
intertwining Jacobi algebroid structures on $(A;L)$ and degree $1$ graded derivations of $L_A\to A[1]$ that are \emph{homological}, i.e.~square to zero.
So, one gets the Jacobi version of a classical result~\cite{vaintrob1997lie}
\begin{equation}
	\label{eq:vaintrob}
	\{\text{Jacobi algebroid structures on $(A;L)$}\}\overset{\sim}{\longrightarrow}\operatorname{MC}(\calD L_A,[-,-])
\end{equation}
identifying Maurer--Cartan elements of the graded Lie algebra of graded derivations of $L_A\to A[1]$ (cf., e.g.,~\cite{le2017bfv}) with Jacobi algebroid structures on $(A;L)$.
\end{remark}

Let us see how the construction of the local adapted coordinates $(z^\alpha,p,p_\alpha)$ from Section~\ref{sec:shifted_1st_jet_bundle_N-mfld} specializes to the current framework.
Fix local coordinates $x^i$ on $M$, a local frame $e_a$ of $A\to M$, and a local non-zero section $\mu$ of $L\to M$.
Denote by $u^a$ the basis of fiberwise linear functions on $A[1]$ given by the dual frame $e^\ast_a$ of $A^\ast\to M$, and by $p,p_i,p_a$ the momenta on $J^1[2]L_A$ associated to $(x^i,u^a)$.
In particular, we have
\begin{equation}
	|x^i|=0,\qquad |u^a|=|p_a|=1,\qquad |p|=|p_i|=2,
\end{equation}
and the local expression~\eqref{eq:prop:Cartan_contact_form} for the canonical $\calL$-valued degree $3$ contact form $\theta$ on $J^1[2]L_A$ becomes
\begin{equation}
	\label{Cartan_contact form:local_coordinates:degree1}
	\theta=(\rmd p-p_i\rmd x^i- p_a\rmd u^a)\otimes\mu,
\end{equation}
where for simplicity we identify the local frame $\mu$ of $L\to M$ with the local frame of $\calL\to J^1[2]L_A$ obtained by pull-back along $J^1[2]L_A\overset{\pi}{\to}A[1]\overset{p}{\to}M$.
Further, the local expression~\eqref{eq:canonical_Jacobi_structure:local_coordinates} for the degree $-2$ non-degenerate Jacobi structure $\{-,-\}$ on $\calL\to J^1[2]L_A$ can be rewritten as follows
\begin{equation}
	\label{eq:canonical_Jacobi_structure:local_coordinates:degree1}
	\{f_1\mu,f_2\mu\}=\left(\left(f_1\frac{\partial f_2}{\partial p}-\frac{\partial f_1}{\partial p}f_2+D_i f_1\frac{\partial f_2}{\partial p_i}-\frac{\partial f_1}{\partial p_i}D_i f_2\right)+(-)^{|f_1|}\left(D_a f_1\frac{\partial f_2}{\partial p_a}+\frac{\partial f_1}{\partial p_a}D_a f_2\right)\right)\otimes\mu,
\end{equation}
for all homogeneous local functions $f_1$ and $f_2$ on $J^1[2]L_A$.

We want to show that $J^1L_A$ is a vector bundle not only over $A$ but also over its $L$-twisted dual $A^\dagger$ and, moreover, these two vector bundle structures are compatible so that they make $J^1L_A$ into a \emph{double vector bundle}~\cite[Chap.~9]{mackenzie2005general}.
%Let us consider the bundle map $A^\dagger:=A^\ast\otimes L\overset{q}{\to}M$, i.e.~the $L$-twisted dual of $A$ (cf.~Section~\ref{rem:GJ_bracket}).
For each choice of local coordinates $x^i$ on $M$, a local frame $e_a$ of $A\to M$, and a local frame $\mu$ of $L\to M$, we also construct the local frame $e_a^\ast\otimes\mu$ of $A^\dagger:=A^\ast\otimes L\overset{q}{\to}M$, and so obtain the corresponding basis of fiberwise linear functions on $A^\dagger$  denoted by $\widetilde{u}^a$.

\begin{proposition}
	\label{prop:DVB}
	Denote by $\tau:J^1L_A\to A^\dagger$ the smooth map defined by the following condition
	\begin{equation}
		\label{eq:prop:DVB:tau}
		\iota_X\tau(j^1(g\pi^\ast\lambda))=(X^{\sf v}g)\lambda,
	\end{equation}
	for all $X\in\Gamma(A)$, $g\in C^\infty(A)$, and $\lambda\in\Gamma(L)$, where $X^{\sf v}\in\frakX(A)$ denotes the \emph{vertical lift} of $X$.
	Then the first jet bundle $J^1L_A$ is a double vector bundle over $A$ and $A^\dagger$ as follows
	\begin{equation}
		\label{eq:prop:DVB:diagram}
	\begin{tikzcd}
	J^1L_A\arrow[rr, "\tau"]\arrow[d, swap, "\pi"]&&A^\dagger\arrow[d, "q"]\\
	A\arrow[rr, swap, "p"]&&M
	\end{tikzcd}
	\end{equation} 
\end{proposition}
\begin{proof}
	First of all, notice that the smooth map $\tau:J^1L_A\to A^\dagger$ is well-defined by Equation~\eqref{eq:prop:DVB:tau} and actually, in local adapted cooordinates, it acts as follows
	%In particular, in local coordinates, $\tau$ acts like $(x^i,u^a,p,p_i,p_\alpha) \longmapsto (\widetilde{x}^i,\widetilde{u}_a)$ where
	\begin{equation}
		\label{eq:prop:DVB:tau_local_coordinates}
		(x^i,u^a,p,p_i,p_a)\overset{\tau}{\longmapsto}(\widetilde{x}^i,\widetilde{u}^a)=(x^i,p_a).
	\end{equation}
	Therefore, one immediately gets that $\tau:J^1L_A\to A^\dagger$ is a VB morphism covering the map $p:A\to M$.
	Moreover, one can equip $J^1L_A\overset{\tau}{\to}A^\dagger$ with the vector bundle structure $(+_{A^\dagger},\rmh_{A^\dagger})$ such that locally
	\begin{gather*}
		(x^i,u^a,p,p_i,p_a)+_{A^\dagger}(x^i,u^{a\prime},p^\prime,p_i^\prime,p_a)=(x^i,u^a+u^{a\prime},p+p^\prime,p_i+p_i^\prime,p_a),\\
		\rmh_{A^\dagger}^t(x^i,u^a,p,p_i,p_a)=(x^i,tu^a,tp,tp_i,p_a).
	\end{gather*}
	Finally, one can easily check that these new structure maps are vector bundle morphisms wrt the given vector bundle structures on $J^1L_A\to A$ and $A^\dagger\to M$ and this completes the proof.
\end{proof}

As a straightforward consequence of Proposition~\ref{prop:DVB} we obtain the following identification.

\begin{corollary}
	\label{cor:bi-grading}
	The algebra $C^\infty(J^1[2]L_A)$ of smooth functions on the shifted first jet bundle identifies with the subalgebra of $C^\infty(J^1L_A)$ formed by those functions on $J^1L_A$ that are homogeneous wrt to both vector bundle structures $J^1L_A\to A$ and $J^1L_A\to A^\dagger$.
\end{corollary}
	
\begin{remark}	
	\label{rem:bi-grading}
	In view of Corollary~\ref{cor:bi-grading}, the $\bbN$-graded commutative algebra $C^\infty(J^1[2]L_A)$ is endowed with an $\bbN\times\bbN$ grading by $(\epsilon,\delta)$, where $\epsilon$ and $\delta$ are the homogeneity degrees wrt the VB structures $J^1L_A\to A$ and $J^1L_A\to A^\dagger$ respectively, so that for local adapted coordinates
		\begin{align*}
			\epsilon(x^i)&=0,&\epsilon(u^a)&=0,&\epsilon(p)&=1,&\epsilon(p_i)&=1,&\epsilon(p_a)&=1,\\
			\delta(x^i)&=0,&\delta(u^a)&=1,&\delta(p)&=1,&\delta(p_i)&=1,&\delta(p_a)&=0.
		\end{align*}
	Then one recovers the total $\bbN$ grading $|-|$ on $C^\infty(J^1[2]L_A)$ as the sum of $\epsilon$ and $\delta$, i.e.
	\begin{equation*}
		C^\infty(J^1[2]L_A)=\bigoplus_{k\in\bbN} C^\infty(J^1[2]L_A)^k\quad\text{and}\quad C^\infty(J^1[2]L_A)^k=\bigoplus_{\epsilon+\delta=k}C^\infty(J^1[2]L_A)^{(\epsilon,\delta)}.
	\end{equation*}
	Since the line bundles $L_A:=p^\ast L\to A[1]$ and $\calL:=\pi^\ast L_A\to J^1[2]L_A$ are generated in degree $0$, the $\bbN\times\bbN$ grading extends from the graded algebra $C^\infty(J^1[2]L_A)$ to the graded module $\Gamma(\calL)$, so that
	\begin{equation}
		\label{eq:bi-degree:direct_sum}
	\Gamma(\calL)=\bigoplus_k\Gamma(\calL)^k\quad\text{and}\quad\Gamma(\calL)^k=\bigoplus_{\epsilon+\delta=k}\Gamma(\calL)^{(\epsilon,\delta)}.
	\end{equation}
\end{remark}

Having introduced the $\bbN\times\bbN$ grading on $\Gamma(\calL)$, let us check now what is the bidegree of the canonical degree $-2$ non-degenerate Jacobi structure $\{-,-\}$ on $\calL\to J^1[2]L_A$. 
\begin{proposition}
	The canonical Jacobi bracket $\{-,-\}$ on $\calL\to J^1[2]L_A$ has bidegree $(-1,-1)$.
\end{proposition}

\begin{proof}
	It follows from Remark~\ref{rem:bi-grading} and the expression~\eqref{eq:canonical_Jacobi_structure:local_coordinates:degree1} for $\{-,-\}$ on local sections.
\end{proof}

Specializing the embeddings~\eqref{eq:LB-sections_from_derivations} and~\eqref{eq:LB-sections_from_pull-back} to our current framework, we get that, for each $k\in\bbN$,
\begin{itemize}
	\item the pull-back of functions along $\pi\colon J^1[2]L_A\to A[1]$ identifies degree $k$ functions on $A[1]$, i.e.~$k$-forms on $A$, with bidegree $(0,k)$ functions on $J^1[2]L_A$,
	\item the pull-back of sections along $\pi\colon J^1[2]L_A\to A[1]$ identifies degree $k$ sections of $L_A\to A[1]$, i.e.~$L$-valued $k$-forms on $A$, with bidegree $(0,k)$ sections of $\calL\to J^1[2]L_A$,
	\item the Hamiltonian lift identifies degree $k$ derivations of $L_A\to A[1]$ with bidegree $(1,k-1)$ sections of $\calL\to J^1[2]L_A$.
\end{itemize}
Therefore, we can state the following specializations of the embeddings~\eqref{eq:LB-sections_from_derivations} and~\eqref{eq:LB-sections_from_pull-back}.

\begin{proposition}
	\label{prop:sections_bidegree_(0,k)}
	There exist natural isomorphisms of graded algebras and graded modules
	\begin{align*}
	C^\infty(A[1])^\bullet=\Omega^\bullet(A)&\overset{\simeq}{\longrightarrow}C^\infty(J^1[2]L_A)^{(0,\bullet)}&&&\Gamma(L_A)^\bullet=\Omega^\bullet(A;L)&\overset{\simeq}{\longrightarrow}\Gamma(\calL)^{(0,\bullet)}.\\
	f&\longmapsto\pi^\ast f&&&\lambda&\longmapsto\pi^\ast\lambda
	\end{align*}
\end{proposition}

\begin{proposition}
	\label{prop:sections_bidegree_(1,k)}
	There exists a natural graded module isomorphism (the \emph{Hamiltonian lift})
	\begin{align*}
	\calD(L_A)^\bullet\longrightarrow\Gamma(\calL)^{(1,1+\bullet)},\ \delta\longmapsto h_\delta
	\end{align*}
%	\begin{align*}
%	\calD(L_A)^\bullet&\longrightarrow\Gamma(\calL)^{(1,1+\bullet)}\\
%	\delta&\longmapsto h_\delta
%	\end{align*}
	covering graded algebra isomorphism $C^\infty(A[1])^\bullet\overset{\simeq}{\longrightarrow}C^\infty(J^1[2]L_A)^{(0,\bullet)}$.
\end{proposition}

These latter embeddings are complemented by other two embeddings (the ones described in the following Propositions~\ref{prop:sections_bidegree_(k,0)} and~\ref{prop:sections_bidegree_(k,1)}) whose proofs requires and motivates the introduction of Legendre transform in the next section.
Finally, specializing Lemma~\ref{lem:Hamiltonian_lift}, one also gets the following
%The latter embeddings have the following property.
%
\begin{lemma}
	\label{lem:Hamiltonian_lift:degree1}
	For each $\delta,\delta^\prime\in\calD L_A$ and $\lambda,\lambda^\prime\in\Gamma(L_A)$, the following identities hold:
	\begin{equation*}
		\{h_\delta,h_{\delta^\prime}\}=-h_{[\delta,\delta^\prime]},\quad\{h_\delta,\pi^\ast\lambda\}=-\pi^\ast(\delta\lambda),\quad\{\pi^\ast\lambda,\pi^\ast\lambda^\prime\}=0.
	\end{equation*}
\end{lemma}

\subsection{The Legendre Transform of the First Jet Bundles}
\label{sec:contactLegendre}

In this section we will keep considering a vector bundle $A\to M$ and a line bundle $L\to M$.
However, now we apply the construction seen in Section~\ref{sec:shifted_1st_jet_bundle_N-mfld:degree_1} not only to $A$ but also to its $L$-twisted dual $A^\dagger$.
So, we obtain
\begin{itemize}
	\item the degree $1$ $\bbN$ manifolds $A[1]$ and $A^\dagger[1]$,
	\item the line bundles $L_A\to A[1]$ and $L_{A^\dagger}\to A^\dagger[1]$,
	\item the shifted first jet bundles $J^1[2]L_A$ and $J^1[2]L_{A^\dagger}$,
	\item the contact forms $\theta$ and $\widetilde\theta$ on $J^1[2]L_A$ and $J^1[2]L_{A^\dagger}$ with values in $\calL_A$ and $\calL_{A^\dagger}$ respectively.
\end{itemize}
Further, in view of Corollary~\ref{cor:bi-grading}, the algebras $C^\infty(J^1[2]L_A)$ and $C^\infty(J^1[2]L_{A^\dagger})$ are $\bbN\times\bbN$ graded and consist respectively of the homogeneous functions on the following double vector bundles
\begin{equation}
	\label{eq:sec:contactLegendre:DVBs}
	\begin{tikzcd}
		J^1L_A\arrow[rr, "\tau"]\arrow[d, swap, "\pi"]&&A^\dagger\arrow[d, "q"]\\
		A\arrow[rr, swap, "p"]&&M
	\end{tikzcd}
	\qquad\text{and}\qquad
	\begin{tikzcd}
		J^1L_{A^\dagger}\arrow[rr, "\widetilde\tau"]\arrow[d, swap, "\widetilde\pi"]&&A\arrow[d, "\widetilde{q}"]\\
		A^\dagger\arrow[rr, swap, "\widetilde{p}"]&&M
	\end{tikzcd}
\end{equation}
constructed according with Proposition~\ref{prop:DVB}.
Specifically, $C^\infty(J^1[2]L_A)^{\epsilon,\delta}$ consists of the polynomial functions on $J^1L_A$ of degree $\epsilon$ and $\delta$ wrt to the VB structures $J^1L_A\overset{\pi}{\longrightarrow} A$ and $J^1L_A\overset{\tau}{\longrightarrow} A^\dagger$ respectively.
Similarly, $C^\infty(J^1[2]L_{A^\dagger})^{\epsilon,\delta}$ consists of the polynomial functions on $J^1L_{A^\dagger}$ of degrees $\epsilon$ and $\delta$ wrt to the VB structures $J^1L_{A^\dagger}\overset{\widetilde\pi}{\longrightarrow} A^\dagger$ and $J^1L_{A^\dagger}\overset{\widetilde\tau}{\longrightarrow} A$ respectively.

This Section aims at constructing the \emph{Legendre transform} which gives a canonical DVB isomorphism between the DVBs in Equation~\eqref{eq:sec:contactLegendre:DVBs} and a canonical contactomorphism of the graded contact manifolds $J^1[2]L_A$ and $J^1[2]L_{A^\dagger}$ (see Theorem~\ref{theor:contact_Legendre}).
This can be seen as the contact analogue of the Legendre transform of the cotangent bundles (cf., e.g.,~\cite[Section 3.4]{Roytenberg2002gradedsymplectic}).
Our motivation for introducing the Legendre transform for the first jet bundles resides in the need to prove Propositions~\ref{prop:sections_bidegree_(k,0)} and \ref{prop:sections_bidegree_(k,1)} which complementing Propositions~\ref{prop:sections_bidegree_(0,k)} and \ref{prop:sections_bidegree_(1,k)} will play a central role in describing split Courant--Jacobi algebroids as Maurer--Cartan elements in Section~\ref{sec:split-CJ_algebroids}.

\begin{theorem}[The Legendre Transform]
	\label{theor:contact_Legendre}
There exists a unique contactomorphism $F:J^1[2]L_A\to J^1[2]L_{A^\dagger}$, the \emph{contact analogue of the Legendre transform}, satisfying the following two conditions:
\begin{enumerate}[label=(\arabic*)]
	\item\label{enumitem:prop:contact_Legendre:1} $F$ swaps the bi-degree, i.e.~$F^\ast (C^\infty(J^1[2]L_{A^\dagger})^{\epsilon,\delta})=C^\infty(J^1[2]L_A)^{\delta,\epsilon}$,
	\item\label{enumitem:prop:contact_Legendre:2} $F$ induces the identity map on both $A[1]$ and $A^\dagger[1]$.
\end{enumerate}
In particular, in local adapted coordinates, $F$ acts like $(x^i,u^a,p,p_i,p_a) \longmapsto (\widetilde{x}^i,\widetilde{u}^a,\widetilde{p},\widetilde{p}_i,\widetilde{p}_a)$ where
\begin{equation}
\label{eq:contact_Legendre}
\widetilde{x}^i=x^i,\qquad\widetilde{u}^a=p_a,\qquad\widetilde{p}=p- u^ap_a,\qquad\widetilde{p}_i=p_i,\qquad\widetilde{p}_a=u^a.
\end{equation}
\end{theorem}

\begin{proof}
	Choose local coordinates $x^i$ on $M$, a local frame $e_a$ of $A\to M$ and a local frame $\mu$ of $L\to M$.
	Denote by $u^a$ the basis of fiberwise linear functions on $A$ given by $e_a^\ast$ and by $\widetilde{u}_a$ the basis of fiberwise linear functions on $A^\dagger$ given by $e_a\otimes\mu^\ast$.
	Further, let be $p,p_i,p_a$ the momenta on $J^1L_A$ associated to $(x^i,u^a)$ and let $\widetilde{p},\widetilde{p}_i,\widetilde{p}_a$ be the momenta on $J^1L_{A^\dagger}$ associated to $(x^i,\widetilde{u}_a)$.
	Then we have the following local expressions for the Cartan contact forms $\theta$ and $\widetilde{\theta}$ on $J^1[2]L_A$ and $J^1[2]L_{A^\dagger}$ respectively 
	\begin{equation}
		\label{eq:contact_Legendre:intermediate:0}
		\theta=(\rmd p-p_i\rmd x^i- p_a\rmd u^a)\otimes\mu\qquad\text{and}\qquad\widetilde{\theta}=(\rmd\widetilde{p}-\widetilde{p}_i\rmd \widetilde{x}^i-\widetilde{p}_a\rmd \widetilde{u}^a)\otimes\mu.
	\end{equation}
	Now, as it is easy to see, Conditions~\ref{enumitem:prop:contact_Legendre:1} and~\ref{enumitem:prop:contact_Legendre:2} immediately imply that, in local coordinates, $F$ is given by $F(x^i,u^a,p,p_i,p_\alpha)=(\widetilde{x}^i,\widetilde{u}_a,\widetilde{p},\widetilde{p}_i,\widetilde{p}^a)$ where
	\begin{equation}
		\label{eq:contact_Legendre:intermediate:1}
		\widetilde{x}^i=x^i,\quad\widetilde{u}_a=p_a,\quad\widetilde{p}=Ap+A^ip_i+A_a^bu^ap_b,\quad\widetilde{p}_i=A_ip+A_i^jp_j+A_{ia}^bu^ap_b,\quad\widetilde{p}^a=u^a,
	\end{equation}
	for some local functions $A,A^i,A_a^b,A_i,A_i^j,A_{ia}^b$ only depending on the $x^i$'s.
	The latter means nothing but that $F\colon J^1[2]L_A\to J^1[2]L_{A^\dagger}$ comes from a DVB morphism $F\colon (J^1L_A,A,A^\dagger,M)\to(J^1L_{A^\dagger},A,A^\dagger,M)$ covering the identity maps $\id_A\colon A\to A$ and $\id_{A^\dagger}\colon A^\dagger\to A^\dagger$.
	\begin{equation}
	\begin{tikzcd}
	J^1L_A\arrow[rd, "F"]\arrow[rr, "\tau"]\arrow[dd, swap, "\pi"]&&A^\dagger\arrow[rd, equal]\arrow[dd, "q", near start]&\\
	&J^1L_{A^\dagger}\arrow[rr, "\widetilde{\pi}", near start, crossing over]&&A^\dagger\arrow[dd, "\widetilde{p}"]\\
	A\arrow[rd, equal]\arrow[rr, swap, "p", near start]&&M\arrow[rd, equal]&\\
	&A\arrow[rr, swap, "\widetilde{q}"]\arrow[from=uu, swap, crossing over, "\widetilde{\tau}", near start]&&M
	\end{tikzcd}
	\end{equation}
 
	Since $F$ acts like the identity on $A[1]$ and $A^\dagger[1]$, and so on $M$, it lifts to a line bundle isomorphism, still denoted by $F$, between the pullback line bundles $\calL_A\to J^1[2]L_A$ and $\calL_{A^\dagger}\to J^1[2]L_{A^\dagger}$.
	Now, plugging the local expression~\eqref{eq:contact_Legendre:intermediate:1} for $F$ in the local expression~\eqref{eq:contact_Legendre:intermediate:0} for $\widetilde\theta$, one gets
	\begin{align*}
	%F^\ast(\rmd\widetilde{p}-\widetilde{p}_i\rmd \widetilde{x}^i-\widetilde{p}^a\rmd \widetilde{u}_a)&
	F^\ast\widetilde{\theta}%&=F^\ast(\rmd\widetilde{p}-\widetilde{p}_i\rmd \widetilde{x}^i-\widetilde{p}^a\rmd \widetilde{u}_a)\\
	&=(\rmd(Ap+A^ip_i+A_a^bu^ap_b)-(A_ip+A_i^jp_j+A_{ia}^bu^ap_b)\rmd x^i-u^a\rmd p_a)\otimes\mu\\
	&=((\partial_iA-A_i)p+(\partial_iA^j-A_i^j)p_j+(\partial_iA_a^b-A_{ia}^b)u^ap_b)\rmd x^i\otimes\mu\\
	&\phantom{=\ }+A\rmd p\otimes\mu+A^i\rmd p_i\otimes\mu+A_a^bp_b\rmd u^a\otimes\mu-(A_a^bu^a+\delta_a^bu^a)\rmd p_b\otimes\mu,
	\end{align*}
	where we have used the fact that $\rmd (u^ap_b)=p_b\rmd u^a-u^a\rmd p_b$. 
	Therefore, $F:J^1[2]L_A\to J^1[2]L_{A^\dagger}$ is a contactomorphism, i.e.~$F^\ast\widetilde\theta$ coincides with $\theta$ up to a conformal factor, if and only if
	\begin{equation}
	\label{eq:contact_Legendre:intermediate:2}
	A=1,\quad A^i=0,\quad A_a^b=-\delta_a^b,\quad A_i=0,\quad A_i^j=\delta_i^j,\quad A_{ia}^b=0.
	\end{equation}
	Finally, from Equations~\eqref{eq:contact_Legendre:intermediate:1} and~\eqref{eq:contact_Legendre:intermediate:2} it follows that actually $F^\ast\widetilde{\theta}=\theta$ and, in local coordinates, $F$ has exactly the expression given in Equation~\eqref{eq:contact_Legendre}.
\end{proof}

\begin{remark}
	Since $\calL_A\to J^1[2]L_A$ and $\calL_{A^\dagger}\to J^1[2]L_{A^\dagger}$ are pullbacks of the line bundle $L\to M$ along $J^1[2]L_A\to A\to M$ and $J^1[2]L_{A^\dagger}\to A^\dagger\to M$ re\-spec\-tive\-ly, the Legendre transform $F$ lifts a line bundle isomorphism $\calL_A\to\calL_{A^\dagger}$ which we will still denote by $F$ and call Legendre transform.
\end{remark}

As anticipated, the introduction of the Legendre transform of first jet bundles was motivated by the need to prove the following embeddings complementing the ones in Propositions~\ref{prop:sections_bidegree_(k,0)} and~\ref{prop:sections_bidegree_(k,1)}.

\begin{proposition}
	\label{prop:sections_bidegree_(k,0)}
	There exist natural isomorphisms of graded algebras and graded modules
	\begin{align*}
	C^\infty(A^\dagger[1])^\bullet=\Omega^\bullet(A^\dagger)&\overset{\simeq}{\longrightarrow}C^\infty(J^1[2]L_A)^{(\bullet,0)}&&&\Gamma(L_{A^\dagger})^\bullet=\Omega^\bullet(A^\dagger;L)&\overset{\simeq}{\longrightarrow}\Gamma(\calL)^{(\bullet,0)}.\\
	f&\longmapsto F^\ast\widetilde{\pi}^\ast f&&&\lambda&\longmapsto F^\ast\widetilde{\pi}^\ast\lambda
	\end{align*}
\end{proposition}

\begin{proof}
	Apply first Proposition~\ref{prop:sections_bidegree_(0,k)} (to the line bundle $\calL_{A^\dagger}\to J^1[2]L_{A^\dagger}$) and then Theorem~\ref{theor:contact_Legendre}.
\end{proof}

Notice that the embeddings from Propositions~\ref{prop:sections_bidegree_(0,k)} and~\ref{prop:sections_bidegree_(k,0)} agree on $\Gamma(L)=\Gamma(L_A)\cap\Gamma(L_{A^\dagger})$, i.e.
\begin{equation*}
	\pi^\ast \lambda=F^\ast\widetilde\pi^\ast\lambda,\quad\text{for all}\ \lambda\in\Gamma(L).
\end{equation*}
\begin{proposition}
	\label{prop:sections_bidegree_(k,1)}
	There exists a natural graded module isomorphism
	\begin{align*}
	\calD(L_{A^\dagger})^\bullet\longrightarrow\Gamma(\calL)^{(1,1+\bullet)},\quad \delta\longmapsto F^\ast h_\delta,
	\end{align*}
	covering graded algebra isomorphism $C^\infty(A^\dagger[1])^\bullet\overset{\simeq}{\longrightarrow}C^\infty(J^1[2]L_A)^{(\bullet,0)}$.
\end{proposition}

\begin{proof}
	Apply first Proposition~\ref{prop:sections_bidegree_(1,k)} (to the line bundle $\widetilde{\pi}^\ast L_{A^\dagger}\to J^1[2]L_{A^\dagger}$) and then Theorem~\ref{theor:contact_Legendre}.
\end{proof}

As a consequence of these two embeddings, for each $k\in\bbN$, we can identify respectively:
\begin{itemize}
	\item degree $k$ functions on $A^\dagger[1]$, i.e.~$k$-forms on $A^\dagger$, and bidegree $(k,0)$ functions on $J^1[2]L_A$,
	\item degree $k$ sections of $L_{A^\dagger}$, i.e.~$L$-valued $k$-forms on $A^\dagger$, and bidegree $(k,0)$ sections of $\calL_A$,
	\item degree $k$ derivations of $L_{A^\dagger}$ and bidegree $(k-1,1)$ sections of $\calL_A$.
\end{itemize}
Moreover, these latter embeddings are compatible with the graded Lie algebra structures of $(\calD L_{A^\dagger})^\bullet$ and $\Gamma(\calL_{A^\dagger})^{(1,1+\bullet)}$ and with their natural representations on $\Gamma(L_{A^\dagger})^\bullet$ and $\Gamma(\calL_{A^\dagger})^{(\bullet,0)}$ respectively.

\begin{lemma}
	\label{lem:Hamiltonian_lift:Legendre}
	For each $\delta,\delta^\prime\in\calD L_{A^\dagger}$ and $\lambda,\lambda^\prime\in\Gamma(L_{A^\dagger})$, the following identities hold:
	\begin{equation*}
	\{F^\ast h_\delta,F^\ast h_{\delta^\prime}\}=-F^\ast h_{[\delta,\delta^\prime]},\quad\{F^\ast h_\delta,F^\ast \widetilde\pi^\ast\lambda\}=-F^\ast \widetilde\pi^\ast(\delta\lambda),\quad\{F^\ast \widetilde\pi^\ast\lambda,F^\ast\widetilde\pi^\ast\lambda^\prime\}=0.
 	\end{equation*}
\end{lemma}

\begin{proof}
		Apply first Lemma~\ref{lem:Hamiltonian_lift:degree1} (to the line bundle $\calL_{A^\dagger}\to J^1[2]L_{A^\dagger}$) and then Theorem~\ref{theor:contact_Legendre}.
\end{proof}

\section{Split Courant--Jacobi algebroids as degree \texorpdfstring{$2$}{2} contact \texorpdfstring{$\bbN Q$}{NQ}-manifolds}
\label{sec:split-CJ_algebroids}

Roytenberg proved in \cite[Theorem 4.5]{Roytenberg2002gradedsymplectic} that there is a one-to-one correspondence between Courant algebroids $(E,\ldab-,-\rdab,\ldsb-,-\rdsb)$ and degree $2$ symplectic $\bbN Q$ manifolds $(\calM,\omega,X)$.
The cohomological vector field $X$ is Hamiltonian and it is given by $X=X_\Theta$ for some degree $3$ function $\Theta\in C^\infty(\calM)$ satisfying the MC equation $\{\Theta,\Theta\}=0$.
Here $\{-,-\}$ denotes the degree $-2$ non-degenerate Poisson structure on $\calM$ corresponding to the degree $2$ symplectic structure.
%(resp.~Jacobi manifolds) can be seen as (generically non-coorientable) degree $2$ (resp.~$1$) contact $\bbN Q$ manifolds.
Similarly, there is a one-to-one correspondence between Courant--Jacobi algebroids and degree $2$ contact $\bbN Q$ manifolds $(\calM,\theta,X)$.
On the one hand, this extends what was done by Mehta~\cite{Mehta2013gradedcontactJacobi} for Jacobi manifolds and degree $1$ contact $\bbN Q$ manifolds.
On the other hand, it replicates what was done by Grabowski~\cite{Grabowski2013gradedcontactCourant} up to identifying contact $\bbN Q$ manifolds with symplectic $\bbN Q$ manifolds that are homogeneous w.r.t.~an additional principal $\bbR^\times$ bundle structure.

In this Section, since it will be crucial for our aims, we describe explicitly the degree $2$ contact $\bbN Q$ manifold corresponding to a split Courant--Jacobi algebroid (see~Theorem~\ref{theor:CJ_algebroid_as_MC-element}) and use it for constructing the associated curved $L_\infty$ algebra (see Theorem~\ref{theor:higher_derived_brackets:splitCJ}).

\subsection{Split Courant--Jacobi algebroids as degree $2$ contact $\bbN Q$-manifolds}
\label{sec:splitCJalgbd:MC_elements}

Let us start introducing the notion of \emph{split Courant--Jacobi algebroid}.
\begin{definition}
	\label{def:split_CJ_algebroid}
	A Courant--Jacobi algebroid $(A\oplus A^\dagger;L)$ with structure maps $(\ldab-,-\rdab,\ldsb-,-\rdsb,\nabla)$, is said to be a \emph{split} when the non-degenerate $L$-valued symmetric product on $A\oplus A^\dagger$ is given by
	\begin{equation}
		\label{eq:def:split_CJ_algebroid}
		\ldab(\xi_1+\alpha_1,\xi_2+\alpha_2\rdab=\alpha_1(\xi_2)+\alpha_2(\xi_1),
	\end{equation}
	for all $\xi_1+\alpha_1,\xi_2+\alpha_2\in\Gamma(A\oplus A^\dagger)$, so that, in particular, both $A$ and $A^\dagger$ are Lagrangian.
\end{definition}
In the rest of this section, where we only consider split Courant--Jacobi algebroids, the non-degenerate symmetric product $\ldab-,-\rdab$ will be always given by Equation~\eqref{eq:def:split_CJ_algebroid}.

\begin{remark}
	Let $A\to M$ be a vector bundle and $L\to M$ a line bundle.
	Sections of $A$ (resp.~of $A^\dagger$) can be identified, at the same time, with:
	\begin{itemize}
		\item degree $1$ sections of the line bundle $L_{A^\dagger}\to A^\dagger[1]$ (resp.~$L_A\to A[1]$) and
		\item degree $-1$ graded derivations of the line bundle $L_A\to A[1]$ (resp.~$L_{A^\dagger}\to A^\dagger[1]$) as follows
		\begin{equation*}
			\Gamma(A)\overset{\sim}{\to}(\calD L_A)^{-1},\ \xi\mapsto\iota_\xi,\qquad\Gamma(A^\dagger)\overset{\sim}{\to}(\calD L_{A^\dagger})^{-1},\ \alpha\mapsto\iota_\alpha.
		\end{equation*}
	\end{itemize}
	These two identifications are compatible in the sense that the following diagram commute
	\begin{equation}
		\label{eq:compatibility_embeddings}
		\begin{tikzcd}
			(\calD L_A)^{-1}\arrow[r, "h"]&\Gamma(\calL_A)^{(1,0)}\\
			\Gamma(A)=\Gamma(L_{A^\dagger})^1\arrow[u, "\iota_{(-)}"]\arrow[r, swap, "\widetilde\pi^\ast"]&\Gamma(\calL_{A^\dagger})^{(0,1)}\arrow[u, swap, "F^\ast"] 
		\end{tikzcd}
	\qquad\qquad
	\begin{tikzcd}
		(\calD L_{A^\dagger})^{-1}\arrow[r, "h"]&\Gamma(\calL_{A^\dagger})^{(1,0)}\arrow[d, "F^\ast"]\\
		\Gamma(A^\dagger)=\Gamma(L_A)^1\arrow[u, "\iota_{(-)}"]\arrow[r, swap, "\pi^\ast"]&\Gamma(\calL_A)^{(0,1)} 
	\end{tikzcd}
	\end{equation}
	Consequently, we can construct the following isomorphism
	\begin{align}
		\Gamma(A\oplus A^\dagger)=\Gamma(A)\oplus\Gamma(A^\dagger)&\hooklongrightarrow\Gamma(\calL_A)^{1,0}\oplus\Gamma(\calL_A)^{0,1}=\Gamma(\calL_A)^1,\nonumber\\
		\xi+\alpha&\longmapsto h_{\iota_\xi}+\pi^\ast\alpha=F^\ast\widetilde{\pi}^\ast \xi+F^\ast h_{\iota_\alpha}.
		\label{eq:degree_1_sections_calL}
	\end{align}
	and express the product~\eqref{eq:def:split_CJ_algebroid} on $\Gamma(A\oplus A^\dagger)$ in terms of the Jacobi bracket $\{-,-\}$ on $\Gamma(\calL_A)$ as follows
	\begin{equation*}
		\ldab u,v\rdab=-\{u,v\},
	\end{equation*}
	for all $u,v\in\Gamma(A\oplus A^\dagger)$, where we are understanding the isomorphism~\eqref{eq:degree_1_sections_calL}.
	In this entire section we will systematically understand this embedding~\eqref{eq:degree_1_sections_calL} of $\Gamma(A\oplus A^\dagger)$ as $\Gamma(\calL_A)^1$ inside $\Gamma(\calL_A)$.
\end{remark}

Now, as a specialization of the one-to-one correspondence between Courant--Jacobi algebroids and degree $2$ contact $\bbN Q$ manifolds, we establish a one-to-one correspondence between
\begin{itemize}
	\item split Courant--Jacobi algebroids $(A\oplus A^\dagger;L)$ and
	\item degree $2$ contact $\bbN Q$ manifolds $(J^1[2]L_A,\theta,X_\Theta)$.
\end{itemize}
Above $\theta$ is the canonical $\calL_A$-valued contact form on $J^1[2]L_A$ (cf.~Section~\ref{sec:degree_2_NQ-manifolds}), the cohomological contact vector field is Hamiltonian and so it is given by $X=X_\Theta$ for some MC element $\Theta$ of $(\Gamma(\calL_A)[2],\{-,-\})$, where $\{-,-\}$ is the canonical degree $-2$ non-degenerate Jacobi structure on $\calL_A\to J^1[2]L_A$.
Within the above one-to-one correspondence, the structure of split Courant--Jacobi algebroid on $(A\oplus A^\dagger;L)$ can be reconstructed via derived brackets from the corresponding $\Theta$ as detailed in the following.

\begin{theorem}
	\label{theor:CJ_algebroid_as_MC-element}
	There exists a canonical one-to-one correspondence between:
	\begin{itemize}
		\item split Courant--Jacobi algebroid structures $(\ldab-,-\rdab,\ldsb-,-\rdsb,\nabla)$ on $(A\oplus A^\dagger;L)$, and
		\item MC elements of $(\Gamma(\calL_A)[2],\{-,-\})$, i.e.~degree $3$ sections $\Theta$ of $\calL_A\to J^1[2]L_A$ s.t.~$\{\Theta,\Theta\}=0$. 
	\end{itemize}
	This one-to-one correspondence is established by the following two relations
	\begin{equation}
		\label{eq:theor:CJ_algebroid_as_MC-element}
		\ldsb u,v\rdsb=\{ \{u,\Theta\},v\},\quad \nabla_u\lambda=\{ \{u,\Theta\},\lambda\}\equiv\{\{\Theta,\lambda\},u\},
	\end{equation}
	for all $u,v\in\Gamma(A\oplus A^\dagger)$, and $\lambda\in\Gamma(L)$.
\end{theorem}

\begin{proof}
	Given a split Courant--Jacobi algebroid structure $(\ldab-,-\rdab,\ldsb-,-\rdsb,\nabla)$ on $(A\oplus A^\dagger;L)$, let us show how to construct the only MC element $\Theta$ of $(\Gamma(\calL_A)[2],\{-,-\})$ satisfying condition~\eqref{eq:theor:CJ_algebroid_as_MC-element}.

	The vector bundles $A$ and $A^\dagger$ are almost Dirac--Jacobi structures in the split Courant--Jacobi algebroid $(A\oplus A^\dagger;L)$.
	So, by Proposition~\ref{prop:Courant_tensor}, one can construct their Courant--Jacobi tensors $\Upsilon_A\in\Omega^3(A;L)=\Gamma(L_A)^3$ and $\Upsilon_{A^\dagger}\in\Omega^3(A^\dagger;L)=\Gamma(L_{A^\dagger})^3$ by setting
	\begin{equation}
		\label{eq:phi_psi}
		\Upsilon_A(\xi_1,\xi_2,\xi_3)=\ldab\ldsb\xi_1,\xi_2\rdsb,\xi_3\rdab,\qquad\text{and}\qquad\Upsilon_{A^\dagger}(\alpha_1,\alpha_2,\alpha_3)=\ldab\ldsb\alpha_1,\alpha_2\rdsb,\alpha_3\rdab,
	\end{equation}
	for all $\xi_1,\xi_2,\xi_3\in\Gamma(A)$ and $\alpha_1,\alpha_2,\alpha_3\in\Gamma(A^\dagger)$.
	Moreover, the almost Dirac--Jacobi structures $A$ and $A^\dagger$ are complementary.
	Therefore, by Proposition~\ref{prop:almost_DJ_structure|Jacobi_algebroid}, $(A;L)$ and $(A^\dagger;L)$ become almost Jacobi algebroids with structures $([-,-]_A,\nabla^A)$ and $([-,-]_{A^\dagger},\nabla^{A^\dagger})$ given by
	%\begin{equation*}
	%[\xi_1,\xi_2]_A=\pr_A\ldsb(\xi_1,0),(\xi_2,0)\rdsb\qquad\text{and}\qquad[\alpha_1,\alpha_2]_{A^\dagger}=\pr_{A^\dagger}\ldsb(0,\alpha_1),(0,\alpha_2)\rdsb,
	%\end{equation*}
	\begin{equation*}
		[\xi,\zeta]_A=\pr_A\ldsb\xi,\zeta\rdsb,\ \ \nabla^A_\xi\lambda=\nabla_\xi\lambda,\qquad\text{and}\qquad[\alpha,\beta]_{A^\dagger}=\pr_{A^\dagger}\ldsb\alpha,\beta\rdsb,\ \ \nabla^{A^\dagger}_\alpha\lambda=\nabla_\alpha\lambda,
	\end{equation*}
	for all $\xi,\zeta\in\Gamma(A)$, $\alpha,\beta\in\Gamma(A^\dagger)$, $\lambda\in\Gamma(L)$.
	As pointed out in Remark~\ref{rem:vaintrob:Jacobi}, these almost Jacobi alge\-br\-oid structures are fully encoded into the corresponding degree $1$ graded derivation $\rmd_{A,L}\in(\calD L_A)^1$ and $\rmd_{A^\dagger,L}\in(\calD L_{A^\dagger})^1$ (see also Remark~\ref{rem:deRham_differential}).
	The identities~\eqref{eq:def:Courant-Jacobi_algebroids:product} and~\eqref{eq:def:Courant-Jacobi_algebroids:nabla} (or their rephrasing as Equation~\eqref{eq:proof:prop:Courant_tensor}) allow to rewrite $\ldsb-,-\rdsb$ and $\nabla$ in terms of the quadruple $\Upsilon_A, \rmd_{A,L}, \rmd_{A^\dagger,L}, \Upsilon_{A^\dagger}$ as follows
	\begin{equation}
	\label{eq:proof:theor:CJ_algebroid_as_MC-element:intermediate}
	\begin{aligned}
	\ldsb X+\alpha,Y+\beta\rdsb&=[X,Y]_A-\iota_\beta\rmd_{A^\dagger,L}X+\scrL_\alpha Y+\iota_\beta\iota_\alpha\Upsilon_{A^\dagger}+\iota_Y\iota_X\Upsilon_A+\scrL_X \beta-\iota_Y\rmd_{A,L}\alpha+[\alpha,\beta]_{A^\dagger},\\
	\nabla_{X+\alpha}\lambda&=(\rmd_{A,L}\lambda)X+(\rmd_{A^\dagger,L}\lambda)\alpha,
	\end{aligned}
	\end{equation}
	for all $X,Y\in\Gamma(A)$, $\alpha,\beta\in\Gamma(A^\dagger)$, $\lambda\in\Gamma(L)$, where use the Cartan calculus on $\Omega^\bullet(A;L)$ and $\Omega^\bullet(A^\dagger;L)$.
	
	Now we construct a degree $3$ section $\Theta\in\Gamma(\calL)^3$, out of the quadruple $\Upsilon_A,\rmd_{A,L},\rmd_{A^\dagger,L},\Upsilon_{A^\dagger}$, by setting
	\begin{equation}
	\label{eq:Theta}
	\Theta:=-\pi^\ast\Upsilon_A+h_{\rmd_{A,L}}+F^\ast h_{\rmd_{A^\dagger,L}}-F^\ast\widetilde{\pi}^\ast\Upsilon_{A^\dagger}.
	\end{equation}
	Using Lemmas~\ref{lem:Hamiltonian_lift} and~\ref{lem:Hamiltonian_lift:Legendre}, and the fact that the canonical Jacobi bracket $\{-,-\}$ on $\calL_A\to J^1[2]L_A$ has bidegree $(-1,-1)$, one can compute that, for all $X,Y\in\Gamma(A)$, $\alpha,\beta\in\Gamma(A^\dagger)$ and $\lambda\in\Gamma(L)$,
	\begin{align*}
	\{\{\Theta,X\},Y\}&=-\{\{\pi^\ast\Upsilon_A,h_{\iota_X}\},h_{\iota_Y}\}+\{\{h_{\rmd_A},h_{\iota_X}\},h_{\iota_Y}\}=\pi^\ast\iota_Y\iota_X\Upsilon_A+h_{\iota_{[X,Y]_A}},
	\\
	\{\{\Theta,X\},\beta\}&=\{\{h_{\rmd_{A,L}},h_{\iota_X}\},\pi^\ast\beta\}+\{\{F^\ast h_{\rmd_{A^\dagger,L}},F^\ast\widetilde{\pi}^\ast X\},F^\ast h_{\iota_\beta}\}=\pi^\ast\scrL_X\beta-F^\ast\widetilde{\pi}^\ast\iota_\beta\rmd_{A^\dagger,L}X,\\
	\{\{\Theta,\alpha\},Y\}&=\{\{h_{\rmd_{A,L}},\pi^\ast\alpha\},h_{\iota_Y}\}+\{\{F^\ast h_{\rmd_{A^\dagger,L}},F^\ast h_{\iota_\alpha}\},F^\ast\widetilde{\pi}^\ast Y\}=-\pi^\ast\iota_Y\rmd_{A,L}\alpha+F^\ast\widetilde{\pi}^\ast\scrL_\alpha Y,
	\\
	\{\{\Theta,\alpha\},\beta\}&=\{\{F^\ast h_{\rmd_{A^\dagger,L}},F^\ast h_{\iota_\alpha}\},F^\ast h_{\iota_\beta}\}-\{\{F^\ast\widetilde{\pi}^\ast\Upsilon_{A^\dagger},F^\ast h_{\iota_\alpha}\},F^\ast h_{\iota_\beta}\}=F^\ast\widetilde{\pi}^\ast\iota_\beta\iota_\alpha\Upsilon_{A^\dagger}+F^\ast h_{\iota_{[\alpha,\beta]_{A^\dagger}}},\\
	\{\{\Theta,X\},\lambda\}&=\{\{h_{\rmd_{A,L}},h_{\iota_X}\},\pi^\ast\lambda\}=\pi^\ast\scrL_X\lambda=\pi^\ast\iota_X\rmd_{A,L}\lambda,\\
	\{\{\Theta,\alpha\},\lambda\}&=\{\{F^\ast h_{\rmd_{A,L}},F^\ast h_{\iota_\alpha}\},F^\ast\widetilde{\pi}^\ast\lambda\}=F^\ast\widetilde\pi^\ast\scrL_\alpha\lambda=F^\ast\iota_\alpha\rmd_{A^\dagger,L}\lambda,
	\end{align*}
	where we also use again the Cartan calculus on $\Omega^\bullet(A;L)$ and $\Omega^\bullet(A^\dagger;L)$ and the embedding~\eqref{eq:degree_1_sections_calL}.
	By the latter and Equation~\eqref{eq:proof:theor:CJ_algebroid_as_MC-element:intermediate}, one can rewrite $\ldsb-,-\rdsb$ and $\nabla$ in terms of the $\Theta$ and $\{-,-\}$ as follows
	\begin{equation}
		\label{eq:proof:theor:CJ_algebroid_as_MC-element:final}
		\ldsb u,v\rdsb=\{\{\Theta,u\},v\}\quad\text{and}\quad\nabla_u\lambda=\{\{\Theta,u\},\lambda\}\equiv\{\{\Theta,\lambda\},u\},
	\end{equation}
	for all $u,v\in\Gamma(A\oplus A^\dagger)$ and $\lambda\in\Gamma(L)$.
	These computations show that Equation~\eqref{eq:Theta} gives the only $\Theta\in\Gamma(\calL_A)$ satisfying condition~\eqref{eq:theor:CJ_algebroid_as_MC-element}.
	Further, using the graded Jacobi identity of the Jacobi bracket $\{-,-\}$ and Equation~\eqref{eq:proof:theor:CJ_algebroid_as_MC-element:final}, one can compute that, for all $\lambda\in\Gamma(L)$ and $u,v,w\in\Gamma(A\oplus A^\dagger)$,
	\begin{align}
	\{\{\{\{\Theta,\Theta\},u\},v\},w\}
	&=2(\ldsb\ldsb u,v\rdsb, w\rdsb-\ldsb u,\ldsb v,w\rdsb\rdsb+\ldsb v,\ldsb u,w\rdsb\rdsb),
	\label{eq:proof:theor:CJ_algebroid_as_MC-element:MC-equation:1}
	\\
	\{\{\{\{\Theta,\Theta\},u\},v\},\lambda\}
	&=2(\nabla_{\ldsb u,v\rdsb}\lambda-[\nabla_{u},\nabla_{v}]\lambda).
	\label{eq:proof:theor:CJ_algebroid_as_MC-element:MC-equation:2}
	\end{align}
	Since the bi-derivation $\{-.-\}$ of $\calL_A\to J^1[2]L_A$ is non-degenerate, the latter show that the MC equation $\{\Theta,\Theta\}=0$ is equivalent to properties~\eqref{eq:def:Courant-Jacobi_algebroids:Jacobi_identity} and~\eqref{eq:flat} of $\ldsb-,-\rdsb$ and $\nabla$.
	
	As last step in the proof, given a MC element $\Theta$ of $(\Gamma(\calL_A)[2],\{-,-\})$, let us show that it arises, by the previous construction, from a unique split Courant--Jacobi algebroid structure $(\ldab-,-\rdab,\ldsb-,-\rdsb,\nabla)$ on $(A\oplus A^\dagger;L)$.
	Since $\{-,-\}$ is a degree $-2$ Jacobi structure on $\calL_A\to J^1[2]L_A$, condition~\eqref{eq:theor:CJ_algebroid_as_MC-element} defines a vector bundle morphism $\nabla:A\oplus A^\dagger\to DL$ over $\id_M$ and a bracket $\ldsb-,-\rdsb$ on $\Gamma(A\oplus A^\dagger)$ satisfying the identitites~\eqref{eq:def:Courant-Jacobi_algebroids:Jacobi_identity} and~\eqref{eq:flat}.
	Finally, by Equation~\eqref{eq:proof:theor:CJ_algebroid_as_MC-element:MC-equation:1}, the MC equation for $\Theta$ implies that the bracket $\ldsb-,-\rdsb$ satisfies the Jacobi identity.
\end{proof}

\subsection{The curved $L_\infty$ algebra of a split Courant--Jacobi algebroid}
\label{sec:splitCJalgbd:L_infty_algebra}

Let us keep considering a split Courant--Jacobi algebroid $(A\oplus A^\dagger;L)$.
Using the same notations of Theorem~\ref{theor:CJ_algebroid_as_MC-element} and its proof, the Courant--Jacobi algebroid structure is fully encoded into the corresponding MC element $\Theta$ of $(\Gamma(\calL_A)[2],\{-,-\})$ which decomposes according to the bidegree as follows
\begin{equation}
	\label{eq:Theta:bidegree_decomposition}
	\Theta=-\underbrace{\pi^\ast\Upsilon_A\vphantom{h_{\rmd_{A^\dagger,L}}}}_{(0,3)}
	+\underbrace{h_{\rmd_{A,L}}\vphantom{h_{\rmd_{A^\dagger,L}}}}_{(1,2)}
	+\underbrace{F^\ast h_{\rmd_{A^\dagger,L}}}_{(2,1)}
	-\underbrace{F^\ast\widetilde\pi^\ast\Upsilon_{A^\dagger}\vphantom{h_{\rmd_{A^\dagger,L}}}}_{(3,0)}.
\end{equation}
Consequently, the split Courant--Jacobi algebroid $(A\oplus A^\dagger;L)$ gets attached with a \emph{curved V-data} (see Definition~\ref{def:V-data} and also~\cite{CS} where the terminology V-data has been introduced for the first time).
\begin{lemma}
	\label{lem:V-data}
	Each split Courant--Jacobi algebroid is attached with the curved V-data formed by
	\begin{itemize}
		\item the graded Lie algebra $\frakg:=(\Gamma(\calL_A)[2],\{-,-\})$,
		\item its abelian Lie subalgebra $\fraka:=\Gamma(L_A)[2]=\Omega^\bullet(A;L)[2]$,
		\item the natural projection $P\colon\frakg\to\fraka$ given by the restriction to the zero section of $J^1[2]L_A\to A[1]$,
		\item the Maurer--Cartan element $-\Theta$ of $\frakg$. 
	\end{itemize}
\end{lemma}

\begin{proof}
	Since $\{-,-\}$ is a degree $-2$ Jacobi structure on the line bundle $\calL_A\to J^1[2]L_A$, it turns out that $(\Gamma(\calL_A)[2],\{-,-\})$ is a graded Lie algebra.
	Recall that, specifically, $\{-,-\}$ has bidegree $(-1,-1)$ wrt the $\bbN\times\bbN$ bigrading of the line bundle $\calL_A\to J^1[2]L_A$.
	From Proposition~\ref{prop:sections_bidegree_(0,k)}, we know that $\fraka:=\Gamma(L_A)^\bullet=\Omega^\bullet(A;L)$ identifies with $\Gamma(\calL_A)^{\bullet,0}\subset \Gamma(\calL_A)$, the graded sapce of those sections that are constant along the fibers of $J^1[2]L_A\to A[1]$.
	Consequently, by bidegree reasons, one gets
	\begin{equation*}
		\{\Gamma(\calL_A)^{0,\bullet},\Gamma(\calL_A)^{0,\bullet}\}=0,
	\end{equation*}
	so $\fraka$ is an abelian Lie subalgebra of $\fraka$.
	The restriction to the zero section of $J^1[2]L_A\to A[1]$ simply gives the natural projection $P:\Gamma(\calL_A)^{\bullet,\bullet}\rightarrow\Gamma(\calL_A)^{0,\bullet}$ wrt the direct sum decomposition~\eqref{eq:bi-degree:direct_sum}.
	Hence, $\ker P\subset\Gamma(\calL_A)$ consists of the sections vanishing on $A[1]$, i.e.~those sections with only components of bidegree $(\epsilon,\delta)$ with $\epsilon\geq 1$.
	Consequently, by bidegree reasons, one gets
	\begin{equation*}
		\{\Gamma(\calL_A)^{\geq 1,\bullet},\Gamma(\calL_A)^{\geq 1,\bullet}\}\subset\Gamma(\calL_A)^{\geq 1,\bullet},
	\end{equation*}
	so $\ker P\subset\frakg$ is a Lie subalgebra.
	Finally, $\Theta$ is a MC element of $(\Gamma(\calL_A)[2],\{-,-\})$ by Theorem~\ref{theor:CJ_algebroid_as_MC-element}, and this shows that the given quadruple forms a set of curved V-data according to Definition~\ref{def:V-data}.
\end{proof}

In the next Theorem~\ref{theor:higher_derived_brackets:splitCJ}, applying Voronov's higher derived brackets technique~\cite{voronov2005higher} to the curved V-data from Lemma~\ref{lem:V-data}, we construct the \emph{curved $L_\infty[1]$-algebra of the split Courant--Jacobi algebroid $(A\oplus A^\dagger;L)$}.
This is actually $\fraka=\Omega^\bullet(A;L)[2]$ equipped with an $L_\infty[1]$-algebra structure $\{\frakm_k\}_{k\in\bbN}$ whose multibrackets, degree $1$ graded symmetric maps $\frakm_k:{\sf S}^k\fraka\to\fraka$, are expressed in terms of the geometry of the split Courant--Jacobi algebroid $(A\oplus A^\dagger;L)$ as detailed in Theorem~\ref{theor:higher_derived_brackets:splitCJ}.
So that, in particular, $\frakm_k=0$ for all $k>3$
Additionally, we also show that this $L_\infty[1]$ algebra structure makes the graded $\Omega^\bullet(A)$-module $\Omega^\bullet(A;L)[2]$ into an \emph{$LR_\infty[1]$ algebra} (see~\cite{vitagliano2014foliation}, and references therein, for the definition of the $LR_\infty$ algebras also called homotopy Lie--Rinehart algebra).
However, before we can state Theorem~\ref{theor:higher_derived_brackets:splitCJ}, we need to preliminarily introduce some notation.

%Some Notation
\begin{remark}
	For any $\omega\in\Omega^{k+1}(A)$ and $\alpha\in\Omega^{k+1}(A;L)$, the VB-morphisms $\omega^\sharp\colon A\to\wedge^k A^\ast$ and $\alpha^\sharp\colon A\otimes L^\ast\to\wedge^kA^\ast$, over $\id_M$, are defined by
\begin{equation*}
	(\omega^\sharp u_0)(u_1\wedge\ldots\wedge u_k)=\omega(u_0\wedge\ldots\wedge u_k)\quad\text{and}\quad\alpha^\sharp (u_0\otimes\nu)(u_1\wedge\ldots\wedge u_k)=\nu(\alpha(u_0\wedge\ldots\wedge u_k)),
\end{equation*}
for all $u_0,u_1,\ldots,u_k\in\Gamma(A)$ and $\nu\in\Gamma(L^\ast)$.
In particular, $\omega^\sharp=\alpha^\sharp=0$ if $k=-1$, while $\omega^\sharp=\omega$ and $\alpha^\sharp=\alpha$ if $k=0$.
Further, given $\alpha_1\in\Omega^{k_1+1}(A;L),\ldots,\alpha_n\in\Omega^{k_n+1}(A;L)$, one can define the VB morphism $\alpha_1^\sharp\wedge\ldots\wedge \alpha_n^\sharp:\wedge^n (A\otimes L^\ast)\to\wedge^{k_1+\ldots+k_n}A^\ast$ so that
\begin{equation*}
	\label{eq:sharp:property1}
	(\alpha_1^\sharp\wedge\ldots\wedge \alpha_n^\sharp)(\widetilde{u}_1\wedge\ldots\wedge \widetilde{u}_n)=\sum_{\sigma\in S_n}(-)^\sigma \alpha_1^\sharp(\widetilde{u}_{\sigma(1)})\wedge\ldots\wedge \alpha_n^\sharp(\widetilde{u}_{\sigma(n)}),
\end{equation*}
for all $\widetilde{u}_1,\ldots,\widetilde{u}_n\in \Gamma(A\otimes L^\ast)$.
Additionally, for any $\omega\in\Omega^{h+1}(A)$ and $\alpha\in\Omega^{k+1}(A;L)$, we define the VB morphisms $\omega^\sharp\wedge\alpha,\ \omega\wedge \alpha^\sharp:A\otimes L^\ast\to\wedge^{h+k+1}A^\ast$ by setting, for all $u\in\Gamma(A)$ and $\widetilde{u}\in\Gamma(A\otimes L^\ast)$,
\begin{equation*}
	(\omega^\sharp\wedge\alpha)u=(\omega^\sharp u)\wedge\alpha\quad\text{and}\quad(\omega\wedge\alpha^\sharp)\widetilde{u}=\omega\wedge(\alpha^\sharp\widetilde{u}).
\end{equation*}
Similarly, one can also define the VB morphisms $\omega_1^\sharp\wedge\eta,\ \omega\wedge \eta^\sharp:A\to\wedge^{h+k+1}A^\ast$, for all $\omega\in\Omega^{h+1}(A)$ and $\eta\in\Omega^{k+1}(A)$.
%Consequently, for all homogeneous $\alpha\in\Omega^\bullet(A)$ and $\beta\in\Omega^\bullet(A)$, the following identity holds:
%\begin{equation*}
%	\label{eq:sharp:property2}
%	(\alpha\wedge\beta)^\sharp=\alpha^\sharp\wedge\beta+(-)^{|\alpha|}\alpha\wedge\beta^\sharp.
%\end{equation*}
Consequently, the following Leibniz-like identities hold
\begin{equation*}
	\label{eq:sharp:property2}
	(\omega\eta)^\sharp=\omega^\sharp\wedge\eta+(-)^{|\omega|}\omega\wedge\eta^\sharp\quad\text{and}\quad(\omega\alpha)^\sharp=\omega^\sharp\wedge\alpha+(-)^{|\omega|}\omega\wedge\alpha^\sharp,
\end{equation*}
for all homogeneous $\omega,\eta\in\Omega^\bullet(A)$ and $\alpha\in\Omega^\bullet(A;L)$.
\end{remark}

Now we can construct the \emph{$LR_\infty[1]$ algebra of the split Courant--Jacobi algebroid $(A\oplus A^\dagger;L)$}.

\begin{theorem}
	\label{theor:higher_derived_brackets:splitCJ}
	A split Courant--Jacobi algebroid $(A\oplus A^\dagger;L)$ determines the curved $LR_\infty[1]$ algebra structure on the $\Omega^\bullet(A)$-module $\Omega^\bullet(A;L)[2]$ whose only non-trivial brackets are $\frakm_0,\frakm_1,\frakm_2$, and $\frakm_3$.
	%The graded $\Omega^\bullet(A)$-module $\Omega^\bullet(A;L)[2]$ is endowed with the curved $LR_\infty[1]$ algebra structure whose only non-trivial brackets, $\frakm_0,\frakm_1,\frakm_2$, and $\frakm_3$, are given as follows.
	\begin{enumerate}[label=(\arabic*)]
		\item
		\label{enumitem:prop:higher_derived_brackets:0}
		The $0$-ary bracket $\frakm_0$ is given by $\Upsilon_A\in\Omega^3(A;L)$.
		\item
		\label{enumitem:prop:higher_derived_brackets:1}
		The unary bracket $\frakm_1$ is given by $\rmd_{A,L}:\Omega^\bullet(A;L)\to\Omega^\bullet(A;L)$.
		\item
		\label{enumitem:prop:higher_derived_brackets:2}
		The binary bracket $\frakm_2$ is given by, for all homogeneous $\alpha,\beta\in\Omega^\bullet(A;L)$,
		\begin{equation}
			\label{eq:prop:higher_derived_brackets:2}
			\frakm_2(\alpha,\beta)=(-)^{|\alpha|}\ldsb\alpha,\beta\rdsb_{A^\dagger,L},
		\end{equation}
		\item
		\label{enumitem:prop:higher_derived_brackets:3}
		The ternary bracket $\frakm_3$ is expressed as follows in terms of $\Upsilon_{A^\dagger}$,
		\begin{equation}
			\label{eq:prop:higher_derived_brackets:3}
			\frakm_3(\alpha,\beta,\gamma)=-(-)^{|\beta|}(\alpha^\sharp\wedge\beta^\sharp\wedge\gamma^\sharp\otimes\id_L)\Upsilon_{A^\dagger},
		\end{equation}
		for all homogeneous $\alpha,\beta,\gamma\in\Omega^\bullet(A;L)$.
	\end{enumerate}
%	In particular, the $LR_\infty[1]$ algebra $\fraka_\Theta^P$ is flat iff $A$ is a Dirac--Jacobi algebroid.
%	Further, it reduces to a dgLa iff both $A$ and $A^\dagger$ are Dirac--Jacobi algebroids.
\end{theorem}

\begin{proof}
	By Voronov's technique~\cite{voronov2005higher}, the curved V-data in Lemma~\ref{lem:V-data} determine the curved $L_\infty[1]$-algebra structure $\{\frakm_k\}_{k\in\bbN}$ on $\fraka=\Omega^\bullet(A;L)[2]$ given as the following higher derived brackets
	\begin{equation}
		\label{eq:higher_derived_brackets}
		\frakm_k(\alpha_1,\ldots,\alpha_k)=-P\{\ldots\{\{\Theta,\alpha_1\},\alpha_2\},\ldots,\alpha_k\}, 
	\end{equation}
	for all $k\in\bbN$, and $\alpha_1,\ldots,\alpha_k\in\Omega^\bullet(A;L)$.
	Clearly, the latter are derivations of the graded $\Omega^\bullet(A)$-module $\Omega^\bullet(A;L)$ in each entry separately.
	Indeed,  $\{-,-\}$ is a bi-derivation of the $\Omega^\bullet(A)$-module $\Omega^\bullet(A;L)$ and the projection $P\colon \Gamma(\calL_A)^\bullet\to\Gamma(\calL_A)^{0,\bullet}\simeq\Omega^\bullet(A;L)$ is a graded module morphism covering a graded algebra morphism $\underline{P}\colon C^\infty(J^1[2]L_A)^\bullet\to C^\infty(J^1[2]L_A)^{0,\bullet}\simeq\Omega^\bullet(A)$ (cf., e.g., \cite[Remark 2.11]{le2017bfv} for the definition of a module morphism covering an algebra morphism).
	
	Recall that $\Theta\in\Gamma(\calL_A)$ has total degree $3$ with bidegree decomposition~\eqref{eq:Theta:bidegree_decomposition}, the degree $-2$ Jacobi bracket $\{-,-\}$ on $\calL_A\to J^1[2]L_A$ has bidegree $(-1,-1)$, Proposition~\ref{prop:sections_bidegree_(0,k)} identifies $\Omega^\bullet(A;L)=\Gamma(L_A)$ with $\Gamma(\calL_A)^{0,\bullet}$, and $P:\Gamma(\calL_A)^{\bullet,\bullet}\rightarrow\Gamma(\calL_A)^{0,\bullet}$ is the projection wrt the direct sum decomposition~\eqref{eq:bi-degree:direct_sum}.
	Therefore, one can easily compute:
	\begin{align}
		\frakm_0&=\pi^\ast\Upsilon_A,\\
		\label{eq:prop:higher_derived_brackets:1:bis}
		\frakm_1(\alpha)&=-\{h_{\rmd_{A,L}},\pi^\ast\alpha\}=\pi^\ast(\rmd_{A,L}\alpha),\\
		\label{eq:prop:higher_derived_brackets:2:bis}
		\frakm_2(\alpha,\beta)&=-\{\{F^\ast h_{\rmd_{A^\dagger,L}},\pi^\ast\alpha\},\pi^\ast\beta\},\\
		\label{eq:prop:higher_derived_brackets:3:bis}
		\frakm_3(\alpha,\beta,\gamma)&=\{\{\{F^\ast\widetilde{\pi}^\ast\Upsilon_{A^\dagger},\pi^\ast\alpha\},\pi^\ast\beta\},\pi^\ast\gamma\},\\
		\frakm_k&=0\quad\text{(for all $k\geq 4$)},
	\end{align}
	for all $\alpha,\beta,\gamma\in\Omega^\bullet(A;L)=\Gamma(L_A)^\bullet$.
	Above, in Equation~\eqref{eq:prop:higher_derived_brackets:1:bis}, we have also used Lemma~\ref{lem:Hamiltonian_lift:degree1}.
	
	Now it only remains to check that the RHS of Equations~\eqref{eq:prop:higher_derived_brackets:2:bis} and~\eqref{eq:prop:higher_derived_brackets:3:bis} agree with  respectively the RHS of Equations~\eqref{eq:prop:higher_derived_brackets:2} and~\eqref{eq:prop:higher_derived_brackets:3}.
	However, since the latter are all graded symmetric multi-derivations of graded $\Omega^\bullet(A)$-module $\Omega^\bullet(A;L)[2]$, it is enough to check that they agree on the components of degree $0$ and $1$ of $\Omega^\bullet(A;L)$.
	Actually, one can easily check that:
	\begin{align*}
		\frakm_2(\alpha,\lambda)&=-\{\{F^\ast h_{\rmd_{A^\dagger}},F^\ast h_{\iota_\alpha}\},F^\ast\widetilde{\pi}^\ast\lambda\}=-F^\ast \widetilde{\pi}^\ast\calL_\alpha\lambda=-\pi^\ast\nabla^{A^\dagger}_\alpha\lambda,\\
		\frakm_2(\alpha,\beta)&=-\{\{F^\ast h_{\rmd_{A^\dagger}},F^\ast h_{\iota_\alpha}\},F^\ast h_{\iota_\beta}\}=-F^\ast h_{[\scrL_\alpha,\iota_\beta]}=-F^\ast h_{[\alpha,\beta]_{A^\dagger}}=-\pi^\ast[\alpha,\beta]_{A^\dagger},\\
		\frakm_3(\alpha,\beta,\gamma)&=\{\{\{F^\ast\widetilde{\pi}^\ast\Upsilon_{A^\dagger},F^\ast h_{\iota_\alpha}\},F^\ast h_{\iota_\beta}\},F^\ast h_{\iota_\gamma}\}=F^\ast\widetilde{\pi}^\ast (\iota_\gamma\iota_\beta\iota_\alpha\Upsilon_{A^\dagger})=\pi^\ast(\iota_\gamma\iota_\beta\iota_\alpha\Upsilon_{A^\dagger}),%=\pi^\ast(\alpha^\sharp\wedge\beta^\sharp\wedge\gamma^\sharp\otimes\id_L)\Upsilon_{A^\dagger},
	\end{align*}
	for all $\alpha,\beta,\gamma\in\Gamma(A^\dagger)$ and $\lambda\in\Gamma(L)$, where we have used Lemma~\ref{lem:Hamiltonian_lift:Legendre}, the compatility conditions~\eqref{eq:compatibility_embeddings} and the Cartan calculus on $\Omega^\bullet(A^\dagger;L)$.
	The latter shows that Equations~\eqref{eq:prop:higher_derived_brackets:2} and~\eqref{eq:prop:higher_derived_brackets:3} hold.
\end{proof}

\begin{remark}
	\label{rem:joana_paulo}
	As privately pointed out to me by Paulo Antunes and Joana Nunes da Costa, the same arguments used in~\cite[Theorems 4.1 and 4.3]{antunes2020split} can be easily adapted to prove that the construction from Theorem~\ref{theor:higher_derived_brackets:splitCJ} establishes a one-to-one correspondence between:
	\begin{itemize}
		\item split Courant--Jacobi algebroid structures on $(A\oplus A^\dagger;L)$ and
		\item curved $LR_\infty$ algebra structures $\frakm_0,\frakm_1,\frakm_2,\frakm_3$ on $\Omega^\bullet(A;L)[2]$.
	\end{itemize}
\end{remark}

\begin{remark}
	Since $\Omega^\bullet(A)$ is the graded algebra of functions on the graded manifold $A[1]$ and  $\Omega^\bullet(A;L)$ is the graded module of sections of the graded line bundle $L_A\to A[1]$, the $LR_\infty[1]$ algebra of a split Courant--Jacobi algebroid $(A\oplus A^\dagger;L)$ constructed in Theorem~\ref{theor:higher_derived_brackets:splitCJ} is also an instance of a \emph{Jacobi structure up to homotopy}, and equivalently a \emph{Kirillov structure up to homotopy}~\cite{bruce2016kirillov}. 
\end{remark}

\section{\texorpdfstring{The $L_\infty$ Algebra and the Deformation Problem of a Dirac--Jacobi Structure}{The L-infty Algebra and the Deformation Problem of Dirac--Jacobi Structure}}
\label{sec:deformation_theory:DJ}
%introduction to the subsection

In this final section, we address the deformation problem of a  Dirac--Jacobi structure $A$ within a fixed Courant--Jacobi algebroid $(E;L)$.
First, we construct the \emph{deformation $L_\infty$ algebra} of $A$, which is an $L_\infty$ algebra unique up to $L_\infty$ isomorphisms.
Specifically, for each complementary almost Dirac--Jacobi structure $B$, Theorem~\ref{theor:deformation_L_infty_algebra} constructs a cubic $L_\infty$ algebra and Theorem~\ref{theor:GMS} shows that different choices of $B$ give rise to canonically isomorphic $L_\infty$ algebras.
Further, in Theorem~\ref{theor:DJ_deformation:MC-elements}, we prove that this $L_\infty$ algebra actually controls the deformation problem of the Dirac--Jacobi structure $A$ within $(E;L)$.
In conclusion, we also identify the infinitesimal deformations of a Dirac--Jacobi structure $A$ and find sufficient criteria for the existence of obstructions.

\subsection{The Deformation $L_\infty$ Algebra of a Dirac--Jacobi Structure}
\label{subsection:CJ_algebroids_and_DJ_structures:deformation_theory}

%framework of the overall section
Let $(E;L)$ be a Courant--Jacobi algebroid, over a manifold $M$, and let $A\subset E$ be a Dirac--Jacobi structure.
Assume to have chosen an almost Dirac--Jacobi structure $B\subset E$ complementary to $A$, so that $E=A\oplus B$ with both $A^\perp=A$ and $B=B^\perp$.
%\textcolor{gray}{Denote by $\pr_A\colon E\to A$ and $\pr_B\colon E\to B$ the natural projections.
%Denote by $\pr_A\colon A\oplus A^\dagger\to A$ and $\pr_{A^\dagger}\colon A\oplus A^\dagger\to A^\dagger$ the natural projections.}
Then the non-degenerate $L$-valued symmetric product $\ldab-,-\rdab$ on $E$ induces the following VB isomorphism over $\id_M$
\begin{gather}
	\label{eq:VB-iso_A_B^ast}
	B\overset{\sim}{\longrightarrow} A^\dagger,\ u\longmapsto\ldab u,-\rdab|_A,\\
	\label{eq:VB-iso_A+A^ast}
	\phi_B:E=A\oplus B\overset{\sim}{\longrightarrow} A\oplus A^\dagger,\ u+v\longmapsto u+\ldab v,-\rdab|_A.
\end{gather}
The latter, together with $\id_L\colon L\to L$, allows us to transfer the Courant--Jacobi algebroid structure from $(E;L)$ to $(A\oplus A^\ast;L)$ so that, in particular, the non-degenerate $L$-valued symmetric pairing on $A\oplus A^\dagger$ is given by
\begin{equation*}
	\ldab \xi+\alpha,\zeta+\beta\rdab=\alpha(\zeta)+\beta(\xi),
\end{equation*}
for all $\xi+\alpha,\zeta+\beta\in \Gamma(A\oplus A^\dagger)$.
Hence, once equipped with this transferred structure, $(A\oplus A^\dagger;L)$ becomes a split Courant--Jacobi algebroid.
In the rest of this section, for each choice of a complementary almost--Dirac--Jacobi structure $B$, we will systematically understand the identification, induced by $(\phi_B,\id_B)$, of the given Courant--Jacobi algebroid $(E;L)$ with the transferred split Courant--Jacobi algebroid $(A\oplus A^\dagger;L)$.

In view of Theorem~\ref{theor:CJ_algebroid_as_MC-element}, the split Courant--Jacobi algebroid structure thus obtained on $(A\oplus A^\dagger;L)$ is fully encoded by the corresponding MC element of $(\Gamma(\calL_A)[2],\{-,-\})$ that we denote by $\Theta_B$ to stress its dependence on the complementary almost Dirac--Jacobi structure $B$.
Explicitly, one has
\begin{equation*}
	\ldab u, v\rdab=-\{u,v\},\quad\ldsb u, v\rdsb=\{\{u,\Theta_B\},v\},\quad\nabla_u\lambda=\{ \{u,\Theta_B\},\lambda\}\equiv\{\{\Theta_B,\lambda\},u\},
\end{equation*}
for all $u,v\in\Gamma(A\oplus A^\dagger)\simeq \Gamma(\calL_A)^1$ and $\lambda\in \Gamma(L)\simeq\Gamma(\calL_A)^0$.
Further, each choice of a complementary almost Dirac--Jacobi structure determines a (flat) V-data as we get by specializing Lemma~\ref{lem:V-data} to the current situation.
%the (flat) V-data of a Dirac--Jacobi structure
\begin{lemma}
	\label{lem:V-data:DJ}
	Let $(E;L)$ be a Courant--Jacobi algebroid and $A\subset E$ be a Dirac--Jacobi structure.
	Each Lagrangian subbundle $B\subset E$ complementary to $A$ determines the (flat) V-data formed by
	\begin{itemize}
		\item the graded Lie algebra $\frakg:=(\Gamma(\calL_A)[2],\{-,-\})$,
		\item its abelian Lie subalgebra $\fraka:=\Gamma(L_A)[2]=\Omega^\bullet(A;L)[2]$,
		\item the natural projection $P:\frakg\to\fraka$ given by the restriction to the zero section of $J^1[2]L_A\to A[1]$,
		\item the Maurer--Cartan element $-\Theta_B$ of $\frakg$.% satisfying $P(\Theta_B)=0$. 
	\end{itemize}
\end{lemma}

\begin{proof}
	In view of Lemma~\ref{lem:V-data}, the quadruple $(\frakg,\fraka,P,-\Theta_B)$ forms a curved V-data.
	So, we only have to check that $\Theta_B\in\ker P$.
	According to the $\bbN\times\bbN$ bi-degree of $\Gamma(\calL_A)$, the MC element $\Theta_B$ corresponding to the transferred split Courant--Jacobi algebroid $(A\oplus A^\dagger;L)$ decomposes as follows
	\begin{equation*}
		\Theta_B=-\underbrace{\pi^\ast\Upsilon_A\vphantom{h_{\rmd_{A^\dagger,L}}}}_{(0,3)}
		+\underbrace{h_{\rmd_{A,L}}\vphantom{h_{\rmd_{A^\dagger,L}}}}_{(1,2)}
		+\underbrace{F^\ast h_{\rmd_{A^\dagger,L}}}_{(2,1)}
		-\underbrace{F^\ast\widetilde\pi^\ast\Upsilon_{A^\dagger}\vphantom{h_{\rmd_{A^\dagger,L}}}}_{(3,0)},
	\end{equation*}
	where we keep using the same notation from Theorem~\ref{theor:CJ_algebroid_as_MC-element}.
	So, one immediately gets $P(\Theta_B)=\Upsilon_A$.
	Since $A$ is a Dirac--Jacobi structure, $\Upsilon_A=0$ by Proposition~\ref{prop:Courant_tensor}), and so $P(\Theta_B)=0$.
\end{proof}

The next Proposition constructs the (flat) $LR_\infty[1]$-algebra associated with $A$ by the choice of a complementary almost Dirac--Jacobi structure $B$.
We keep using the notations from Theorem~\ref{theor:higher_derived_brackets:splitCJ} and understanding the VB-isomorphism $B\simeq A^\dagger$ induced by $\ldab-,-\rdab$ and the identification $E\simeq A\oplus A^\dagger$.
However, before we can state Theorem~\ref{theor:higher_derived_brackets:splitCJ}, we need to preliminarily introduce some notation concerning the $L_{\infty}[1]$-algebra.

\begin{remark}
	\label{rem:dgL[1]a}
	Let $V$ be a graded vector space.
	Even though in this note we always work with $L_\infty[1]$-algebras, let us recall that $L_\infty$-algebra structures $\{\mu_k\}$ on $V$ correspond bijectively to $L_\infty[1]$-algebra structures $\{\frakm_k\}$ on $V[1]$ via the following relation (see~\cite[Remark 1.1]{fiorenza2007cones}):
	\begin{equation}
		\label{eq:decalage_isomorphism}
		\mu_k(v_1,\ldots,v_k)=(-)^k(-)^{\sum_i(k-i)|v_i|}\frakm_k(v_1[1],\ldots,v_k[1]),
	\end{equation}
	for all $k\in\bbN$ and homogeneous $v_1,\ldots,v_k\in V$.
	Consequently, differential graded Lie algebra (dgLa) structures on $V$ correspond bijectively to \emph{dgL[1]a structures} on $V[1]$, i.e.~those $L_\infty[1]$ algebra structures $\{\frakm_k\}$ on $V[1]$ whose only non-trivial brackets are $\frakm_1$ and $\frakm_2$.
\end{remark}

Now we are ready to construct the \emph{deformation $L_\infty[1]$ algebra} associated with a Dirac--Jacobi structure $A$ by the choice of a complementary almost Dirac--Jacobi structure $B$.

\begin{theorem}
	\label{theor:deformation_L_infty_algebra}
	Let $(E;L)$ be a Courant--Jacobi algebroid and let $A\subset E$ be a Dirac--Jacobi structure.
	For each almost Dirac--Jacobi structure $B$ complementary to $A$, the graded $\Omega^\bullet(A)$-module $\Omega^\bullet(A;L)[2]$ is endowed with an $LR_\infty[1]$ algebra structure $\{\frakm_k^B\}$ with only non-trivial brackets, $\frakm^B_1,\frakm^B_2$, and $\frakm^B_3$.
	\begin{enumerate}[label=(\arabic*)]
		\item
		\label{enumitem:prop:higher_derived_brackets:1_bis}
		The unary bracket $\frakm^B_1$ is the de Rham differential $\rmd_{A,L}:\Omega^\bullet(A;L)\to\Omega^\bullet(A;L)$.
		\item
		\label{enumitem:prop:higher_derived_brackets:2_bis}
		The binary bracket $\frakm^B_2$ is given by, for all homogeneous $\alpha,\beta\in\Omega^\bullet(A;L)$,
		\begin{equation}
		\label{eq:prop:higher_derived_brackets:2_bis}
		\frakm^B_2(\alpha,\beta)=(-)^{|\alpha|}[\alpha,\beta]_{A^\dagger,L}.
		\end{equation}
		\item
		\label{enumitem:prop:higher_derived_brackets:3_bis}
		The ternary bracket $\frakm^B_3$ is expressed as follows in terms of $\Upsilon_{A^\dagger}$
		\begin{equation}
		\label{eq:prop:higher_derived_brackets:3_bis}
		\frakm^B_3(\alpha,\beta,\gamma)=-(-)^{|\beta|}(\alpha^\sharp\wedge\beta^\sharp\wedge\gamma^\sharp\otimes\id_L)\Upsilon_{A^\dagger},
		\end{equation}
		for all homogeneous $\alpha,\beta,\gamma\in\Omega^\bullet(A;L)$.
	\end{enumerate}
	Above, in the LHS of Equations~\eqref{eq:prop:higher_derived_brackets:2_bis} and~\eqref{eq:prop:higher_derived_brackets:3_bis}, we are understanding the VB isomorphism~\eqref{eq:VB-iso_A_B^ast}.
	Moreover, this $LR_\infty[1]$ algebra $(\Omega^\bullet(A;L)[2],\{\frakm_k^B\})$ reduces to a dgL[1]a, i.e.~$\frakm^B_3=0$, iff $B$ is involutive.
\end{theorem}

\begin{proof}
	It follows immediately from Theorem~\ref{theor:higher_derived_brackets:splitCJ} and Lemma~\ref{lem:V-data:DJ}.
\end{proof}

\begin{remark}
	In view of Remarks~\ref{rem:Courant_algebroid} and~\ref{rem:Dirac_structures}, we can also apply Theorem~\ref{theor:deformation_L_infty_algebra} to a Dirac structure $A$ in a Courant algebroid $E$.
	In this way, we recover the cubic $L_\infty$ algebra constructed in~\cite{fregier_zambon_2015} for any choice of an almost Dirac structure $B\subset E$ complementary to $A$.
\end{remark}

\color{black}

\subsection{Change of complementary almost Dirac--Jacobi structure}
\label{sec:L_infty-isomorphism}

A priori the $L_\infty[1]$-algebra attached to a Dirac--Jacobi structure $A$ in a Courant--Jacobi algebroid $(E;L)$ (see Theorem~\ref{theor:deformation_L_infty_algebra}) depends on the choice of a complementary almost Dirac--Jacobi structure $B$.
Actually, it does not, up to $L_\infty[1]$-algebra isomorphisms, as we are going to show below.
Indeed, using the results of~\cite{CS} on the equivalence of higher derived brackets, we construct a canonical $L_\infty[1]$-algebra isomorphism between the $L_\infty[1]$-algebra structures on $\Omega^\bullet(A)[2]$ coming from two different choices of $B$.
Additionally, in this way, we get rather efficient formulas (from a computational point of view) for the Taylor coefficient of this canonical $L_\infty[1]$-algebra isomorphism.

Let $(E;L)$ be a Courant--Jacobi algebroid and let $A\subset E$ be an almost Dirac--Jacobi structure.
For $i=0,1$, fix an almost Dirac--Jacobi structure $B_i\subset E$ complementary to $A$, so that $E=A\oplus B_i$.
Equip $A\oplus A^\dagger$ with the non-degenerate $L$-valued symmetric product $\ldab-,-\rdab$ coming from the duality pairing.
For $i=0,1$, the VB isomorphism over $\id_M$
\begin{equation}
	\label{eq:phi_i}
	\phi_i:E=A\oplus B_i\overset{\sim}{\longrightarrow} A\oplus A^\dagger,\ (u,v)\longmapsto (u,\ldab v,-\rdab|_A),
\end{equation}
is an orthogonal transformation, and there exists unique split Courant--Jacobi algebroid structure on $A\oplus A^\dagger$, with Loday bracket $\ldsb-,-\rdsb_i$ on $\Gamma(A\oplus A^\dagger)$ and VB morphism $\nabla^i\colon A\oplus A^\dagger\to DL$, such that the map~\eqref{eq:phi_i} becomes a Courant--Jacobi algebroid isomorphism.
Denote by $\Theta_i$ the MC element of $(\Gamma(\calL_A)[2],\{-,-\})$ that corresponds to the split Courant--Jacobi algebroid structure $(\ldab-,-\rdab,\ldsb-,-\rdsb_i,\nabla^i)$ on $A\oplus A^\dagger$ in view of Theorem~\ref{theor:CJ_algebroid_as_MC-element}.
Therefore, for $i=0,1$, by Lemma~\ref{lem:V-data:DJ} one gets the following (flat) V-data
\begin{equation*}
	(\Gamma(\calL_A)[2],\Omega^\bullet(A;L)[2],P,-\Theta_i),
\end{equation*}
where $\frakg$ the graded Lie algebra is $(\Gamma(\calL_A)[2],\{-,-\})$, $\fraka\subset\frakg$ is the abelian Lie subalgebra $\Gamma(L_A)=\Omega^\bullet(A;L)\simeq\Gamma(\calL_A)^{0,\bullet}$, and $P$ is the projection $\Gamma(\calL_A)^{\bullet,\bullet}\longrightarrow\Gamma(\calL_A)^{0,\bullet}$.

As pointed out in Theorem~\ref{theor:deformation_L_infty_algebra}, these V-data allow to cook up, by Voronov's technique of higher derived brackets, the $L_\infty[1]$ algebra structure on $\Omega^\bullet(A;L)[2]$ associated with $A$ by the choice of $B_i$ that we are going to denote by $\{\frakm_k^i\}$.
Further, for $0,1$, this $L_\infty[1]$ algebra structure $\{\frakm^i_k\}$ on $\Omega^\bullet(A;L)[2]$ corresponds to a codifferential  $\calQ_i$ of ${\sf S}(\Omega^\bullet(A;L)[2])$ (cf., e.g., Proposition~\ref{prop:coalgebra_coderivation}).

Since both almost Dirac--Jacobi structures $B_0$ and $B_1$ are complementary to $A$, i.e.~$E=A\oplus B_i$, for $i=0,1$, there exists a unique $\epsilon\in\Omega^2(A^\dagger;L)\simeq\Gamma(\calL_A)^{2,0}$ such that
\begin{equation}
	B_1=\gr(\epsilon)\equiv\{\iota_\alpha\epsilon+\alpha\mid\alpha\in A^\dagger\}\subset A\oplus A^\dagger\simeq A\oplus B_0=E,
\end{equation}
where we are understanding the VB isomorphism $\varphi_0:E\overset{\sim}{\longrightarrow} A\oplus A^\dagger$.
Then, setting ${\sf m}:=\{\epsilon,-\}$, we get that ${\sf m}$ is a degree $0$ graded derivation of both
\begin{itemize}
	\item the graded $C^\infty(J^1[2]L_A)$-module $\Gamma(\calL_A)$ and
	\item the graded Lie algebra $(\Gamma(\calL_A)[2],\{-,-\})$.
\end{itemize}
In particular, notice that ${\sf m}$ has bidegree $(1,-1)$, i.e.~${\sf m}(\Gamma(\calL_A)^{\epsilon,\delta})\subset\Gamma(\calL_A)^{\epsilon+1,\delta-1}$.

Denote by ${\sf M}$ the degree $0$ coalgebra coderivation of ${\sf S}(\Omega^\bullet(A;L)[2])$ such that its Taylor coefficients $\{{\sf M}_k\}$ are given by
\begin{equation}
	\label{eq:M_k}
	{\sf M}_k(\omega_1\odot\ldots\odot\omega_k)=P\{\{\ldots,\{{\sf m}\omega_1,\omega_2\},\ldots\},\omega_k\},
\end{equation}
for all $\omega_1,\dots,\omega_k\in\Omega^\bullet(A;L)[2]$.
Recall that ${\sf M}$ is uniquely determined by Equation~\eqref{eq:M_k} (cf., e.g., Proposition~\ref{prop:coalgebra_coderivation}).
Hence, in particular, for all $\omega_1,\omega_2\in\Omega^2(A;L)$ and $\alpha\in\Omega^1(A;L)$,
\begin{equation}
	\label{eq:M_2:bis:app}
	({\sf M}_2(\omega_1\odot \omega_2))^\sharp=\omega_1^\sharp\epsilon^\flat \omega_2^\sharp+\omega_2^\sharp\epsilon^\flat \omega_1^\sharp,\qquad{\sf M}_2(\omega_1\odot \alpha)=\omega_1^\sharp(\epsilon^\flat(\alpha)).
\end{equation}
Moreover, ${\sf M}_k=0$ for $k\neq 2$, and this implies that ${\sf M}({\sf S}^n(\Omega^\bullet(A;L)[2]))\subset{\sf S}^{n-1}(\Omega^\bullet(A;L)[2])$, for all $n$.
So, ${\sf M}$ is pronilpotent and its flow is given by 
%the $1$-parameter group
\begin{equation*}
	e^{t{\sf M}}:=\sum_{k=0}^\infty\frac{t^k}{k!}{\sf M}^k
\end{equation*}
%of graded coalgebra automorphisms $e^{\eps{\sf M}}:=\sum_{k=0}^\infty\frac{\eps^k}{k!}{\sf M}^k$ of ${\sf S}(\Omega^\bullet(A)[2])$.
which is a well-defined one-parameter group of graded coalgebra automorphisms of ${\sf S}(\Omega^\bullet(A;L)[2])$.

Now we are ready to prove that different choices of the complementary almost Dirac--Jacobi structure lead to canonically $L_\infty$ isomorphic $L_\infty[1]$ algebras, and so the deformation $L_\infty[1]$ algebra of a Dirac--Jacobi structure is unique up to $L_\infty$ isomorphisms.

\begin{theorem}%[{\cite[Corollary 3.8]{gualtieri2020deformation}}]
	\label{theor:GMS}
	Keeping the notations introduced above, the coalgebra automorphism $e^{\sf M}$% of ${\sf S}(\Omega^\bullet(A)[2])$
	gives rise to a codifferential coalgebra isomorphism 
	%$e^{\sf M}\colon({\sf S}(\Omega^\bullet(A;L)[2]),\calQ_0)\longrightarrow({\sf S}(\Omega^\bullet(A;L)[2]),\calQ_1)$,
	\begin{equation*}
		e^{\sf M}\colon({\sf S}(\Omega^\bullet(A;L)[2]),\calQ_0)\longrightarrow({\sf S}(\Omega^\bullet(A;L)[2]),\calQ_1),
	\end{equation*}
	and so it corresponds to an $L_\infty[1]$-algebra isomorphism (cf., e.g., Remark~\ref{rem:L_infinity_algebra_iso})
	\begin{equation*}
		\{(e^{\sf M})_k\}_{k\in\bbN}\colon(\Omega^\bullet(A;L)[2],\{\frakm_k^0\})\longrightarrow(\Omega^\bullet(A;L)[2],\{\frakm_k^1\}).
	\end{equation*}
\end{theorem}

\begin{proof}
	First of all, it is easy to see that ${\sf m}:=\{\epsilon,-\}$ satisfies the conditions~\ref{enumitem:CS:1} and~\ref{enumitem:CS:2} in Section~\ref{sec:equivalence_higher_derived_brackets:CS}.
	Further, for each $\lambda\in \Gamma(\calL_A)$, one gets that ${\sf m}^k\lambda=0$ for all but finitely many $k$.
	This means that the bidegree $(+1,-1)$ derivation ${\sf m}$ is \emph{pronilpotent} and so its flow is given by 
	\begin{equation*}
		e^{t{\sf m}}:=\sum_{k=0}^\infty\frac{t^k}{k!}{\sf m}^k
	\end{equation*}
	%the $1$-parameter group of automorphisms $e^{\eps{\sf m}}:=\sum_{k=0}^\infty\frac{\eps^k}{k!}{\sf m}^k$ both of the graded module and of the graded Lie algebra.
	which is a well-defined $1$-parameter group of automorphisms of both the graded $C^\infty(J^1[2]L_A)$-module $\Gamma(\calL_A)$ and the graded Lie algebra $(\Gamma(\calL_A)[2],\{-,-\})$.
	Consequently, for every $t\in\bbR$, we can introduce a new MC element $\Theta(t)$ of $(\Gamma(\calL_A)[2],\{-,-\})$ by setting
	\begin{equation}
		\label{eq:proof:theor:GMS:1}
		\Theta(t):=e^{t{\sf m}}\Theta_0.
	\end{equation}
	Since $e^t{\sf m}$ preserves $\ker P$ and $\Theta_0\in\ker P$, we get that also $\Theta(t)\in\ker P$, for every $t\in\bbR$.
	Therefore, keeping the same notations of Lemma~\ref{lem:V-data:DJ}, we get a new (flat) V-data $(\frakg,\fraka,P,\Theta(t))$ and, by Theorem~\ref{theor:deformation_L_infty_algebra}, also the associated $L_\infty[1]$-algebra structure $\{\frakm_k(t)\}_{k\in\bbN}$ on $\fraka=\Omega^\bullet(A;L)[2]$ with corresponding codifferential $\calQ(t)$ of ${\sf S}(\Omega^\bullet(A;L)[2])$.
	Obviously, for $t=0$ we get back
	$$\Theta(0)=\Theta_0,\quad\text{and}\quad\calQ(0)=\calQ_0.$$
	
	We want to prove that $\Theta(1)=\Theta_1$ (and so also $\calQ(1)=\calQ_1$).
	Using the Equation~\eqref{eq:canonical_Jacobi_structure:local_coordinates:degree1} expressing the Jacobi bracket $\{-,-\}$ in Darboux coordinates on $J^1[2]L_A$, it is a straightforward computation to show that $e^{\sf m}$ on $\Gamma(A\oplus A^\dagger)\simeq\Gamma(\calL_A)^1\subset\Gamma(\calL_A)$ agrees with $\phi_1\circ\phi_0^{-1}$, or equivalently
	\begin{equation*}
		e^{\sf m}\circ\varphi_0=\varphi_1\quad\text{on}\ \Gamma(E).
	\end{equation*}
	Consequently, one immediately gets that, for any $u,v\in\Gamma(A\oplus A^\dagger)$,  %$$(\phi_1\circ\phi_0^{-1})(\xi+\alpha)=\xi+\iota_\alpha\epsilon+\alpha=\xi+\alpha+{\sf m}(\xi+\alpha)=e^{\sf m}(\xi+\alpha)$$, for all $\xi+\alpha\in\Gamma(A\oplus A^\ast)$.
	\begin{equation}
		\label{eq:proof:theor:GMS:2}
		\{\{e^{\sf m}\Theta_0,e^{\sf m}u\},e^{\sf m}v\}=e^{\sf m}\{\{\Theta_0,u\},v\}=e^{\sf m}\ldsb u,v\rdsb_0=\ldsb e^{\sf m}u,e^{\sf m}v\rdsb_1=\{\{\Theta_1,e^{\sf m}u\},e^{\sf m}v\}.
	\end{equation}
	Above we have also used the fact that $\phi_i$ is a Courant--Jacobi isomorphism from $(E;L)$ to $(A\oplus A^\dagger;L)$ where the latter is equipped with the split Courant--Jacobi algebroid structure $(\ldab-,-\rdab,\ldsb-,-\rdsb_i,\nabla^i)$.
	Since the Jacobi structure $\{-,-\}$ on $\calL_A\to J^1[2]L_A$ is non-degenerate, Equation~\eqref{eq:proof:theor:GMS:2} implies that
	\begin{equation}
		\label{eq:proof:theor:GMS:3}
		\Theta_1=e^{\sf m}\Theta_0.
	\end{equation}
	Therefore, comparing Equations~\eqref{eq:proof:theor:GMS:1} and~\eqref{eq:proof:theor:GMS:3}, we get that $\Theta(1)=\Theta_1$ and so also $\calQ(1)=\calQ_1$.

	Finally, applying Theorem~\ref{theor:CS} to $U(t)=e^{t{\sf M}}$ and $\Theta(t):=e^{t{\sf m}}\Theta_0$, for $t=1$, we get that the graded coalgebra automorphism $e^{\sf M}$ of ${\sf S}(\Omega^\bullet(A;L)[2])$
	gives a codifferential graded coalgebra isomorphism
	\begin{equation*}
		e^{\sf M}:({\sf S}(\Omega^\bullet(A;L)[2]),\calQ_0)\longrightarrow ({\sf S}(\Omega^\bullet(A;L)[2]),\calQ_1)
	\end{equation*}
	corresponding to an $L_\infty[1]$-algebra isomorphism $(\Omega^\bullet(A;L)[2],\{\frakm^0_k\})\longrightarrow(\Omega^\bullet(A;L)[2],\{\frakm^1_k\})$.
\end{proof}

\begin{remark}
	In view of Remarks~\ref{rem:Courant_algebroid} and~\ref{rem:Dirac_structures}, Theorem~\ref{theor:GMS} generalizes to Dirac--Jacobi structures the result first obtained in~\cite{gualtieri2020deformation} for Dirac structures using their corresponding pure spinors for the Clifford algebra of the ambient Courant algebroid.
	Therefore, in the case of Dirac structures, Theorem~\ref{theor:GMS} provides an alternative proof entirely based on the equivalence of higher derived brackets.
\end{remark}

\subsection{The Deformation Space of a Dirac--Jacobi Structure}

Given a Dirac--Jacobi structure $A\subset E$ and a complementary almost Dirac--Jacobi structure $B\subset E$, we now turn to the geometric information encoded in the Maurer-Cartan elements of the associated deformation $LR_\infty[1]$-algebra.

Recall that the product $\ldab-,-\rdab$ induces VB isomorphisms $B\overset{\sim}{\to}A^\dagger$ and $\varphi_B\colon E=A\oplus B\overset{\sim}{\to}A\oplus A^\dagger$ (see Equations~\eqref{eq:VB-iso_A_B^ast} and~\eqref{eq:VB-iso_A+A^ast}).
Further, $\varphi_B$ transform the Courant--Jacobi algebroid $(E;L)$ into a split Courant--Jacobi algebroid $(A\oplus A^\dagger;L)$ whose structure is encoded by an MC element $\Theta_B$ of $(\Gamma(\calL_A)[2],\{-,-\})$.
Moreover, understanding the identification $\varphi_B$, it is easy to see that the relation
\begin{equation}
	\label{eq:relation:MC_to_DJ}
	A^\prime=\gr(-\eta)=\{\xi-\iota_\xi\eta\mid\xi\in A\}\subset A\oplus A^\dagger\simeq_{\varphi_B} A\oplus B=E
\end{equation}
establishes a one-to-one correspondence between degree $0$ elements $\eta$ of $\Omega^\bullet(A)[2]$, i.e.~$2$-forms $\eta\in\Omega^2(A)$, and ``small'' deformations of $A$ seen as an almost Dirac--Jacobi structure, i.e.~almost Dirac structures $A^\prime\subset E$ that are \emph{close to $A$ w.r.t. $B$}, in the sense that they are still transverse to $B$.

As proven in the next Theorem~\ref{theor:DJ_deformation:MC-elements}, within the above one-to-one correspondence, the almost Dirac--Jacobi structures $A^\prime\subset E$ is involutive if and only if the $2$-form $\eta\in\Omega^2(A;L)$ is an MC element of $(\Omega^\bullet(A;L)[2],\{\frakm^B_k\})$.
Let us recall here that an \emph{MC element} the $L_\infty[1]$-algebra $\Omega^\bullet(A;L)[2],\{\frakm^B_k\})$ is a degree $0$ element $\eta$ of $\Omega^\bullet(A;L)[2]$, i.e.~a $2$-form $\eta\in\Omega^2(A;L)$, satisfying the following \emph{MC equation}
\begin{equation}
	\label{eq:MC-eqn}
	\frakm^B_1(\eta)+\frac{1}{2}\frakm^B_2(\eta,\eta)+\frac{1}{6}\frakm^B_3(\eta,\eta,\eta)=0.
\end{equation}
So, Theorem~\ref{theor:DJ_deformation:MC-elements} shows that the MC-elements of the associated deformation $LR_\infty[1]$ algebra encode the ``small'' deformations of the Dirac--Jacobi subbundle $A$, i.e.~those Dirac--Jacobi structures $A^\prime\subset E$ that are close to $A$ w.r.t.~$B$.
In other words, for each complementary almost Dirac--Jacobi structure $B\subset E$, the associated $L_\infty[1]$-algebra $(\Omega^\bullet(A;L)[2],\{\frakm_k^B\})$ controls the ``corresponding'' deformation problem of the given Dirac--Jacobi structure $A\subset E$.

\begin{theorem}
	\label{theor:DJ_deformation:MC-elements}
	Let $(E;L)$ be a Courant--Jacobi algebroid and $A\subset E$ a Dirac--Jacobi structure.
	Fix an almost Dirac--Jacobi structure $B\subset E$ transverse to $A$.
	Then the relation $A^\prime=\operatorname{Gr}(\eta)$ (cf.~Equation~\eqref{eq:relation:MC_to_DJ})
%	\begin{equation}
%	P=\text{Gr}(-\eta)
%	\end{equation}
	establishes a canonical one-to-one correspondence between:
	\begin{itemize}
		\item Dirac--Jacobi structures $A^\prime\subset E$ that are transverse to $B$, and
		\item MC-elements $\eta$ of the $LR_\infty[1]$-algebra $(\Omega^\bullet(A;L)[2],\{\frakm^B_k\}_{k\in\bbN})$.
	\end{itemize}
\end{theorem}

\begin{proof}
	It is evident that the relation $A^\prime=\text{Gr}(\eta)$ establishes a one-to-one correspondence between:
	\begin{itemize}
		\item $L$-valued $2$-forms $\eta$ on $A$, i.e.~degree $0$ elements $\eta$ of $\Omega^\bullet(A;L)[2]$, and
		\item almost Dirac--Jacobi structures $A^\prime\subset E$ transverse to $B$, so that $E=A^\prime\oplus B$.
	\end{itemize}
	So, it only remains to prove that, within this one-to-one correspondence, $\eta$ satisfies the MC equation~\eqref{eq:MC-eqn} if and only if $A^\prime$ is involutive.
	
	Fix $\eta\in\Omega^2(A;L)=\Gamma(L_A)^2\simeq\Gamma(\calL_A)^{0,2}$ and denote by $\Delta$ the associated degree $0$ Hamiltonian derivation of the graded Jacobi manifold $(J^1[2]L_A,\calL_A,\{-,-\})$ given by
	$$\Delta:=\{\eta,-\}.$$
	Since $\Delta$ has bidegree $(-1,1)$, it is pronilpotent and generates, as its flow, the one-parameter automorphism group $e^{t\Delta}$ of the graded Jacobi manifold $(J^1[2]L_A,\calL_A,\{-,-\})$ such that, for all $\lambda\in\Gamma(\calL_A)$,
	\begin{equation*}
	(e^{t\Delta})^\ast\lambda=\sum_{k=0}^\infty\frac{t^k}{k!}\Delta^k\lambda,
	\end{equation*}
	where $(e^{t\Delta})^\ast\colon\Gamma(\calL_A)\to\Gamma(\calL_A)$ denotes the pull-back of sections along $e^{t\Delta}\colon\calL_A\to\calL_A$.
	In particular, $(e^{-\Delta})^\ast$ is an automorphism of the graded Lie algebra $(\Gamma(\calL_A)[2],\{-,-\})$, and so $(e^{-\Delta})^\ast\Theta_B$ is still an MC element of the graded Lie algebra $(\Gamma(\calL_A)[2],\{-,-\})$ and encodes a new modified structure of split Courant--Jacobi algebroid on $(A\oplus A^\dagger;L)$ (see Theorem~\ref{theor:CJ_algebroid_as_MC-element}).
	
	Using the expression~\eqref{eq:canonical_Jacobi_structure:local_coordinates:degree1}, one can easily compute that, for all $\xi+\alpha\in\Gamma(A\oplus A^\dagger)\simeq\Gamma(\calL_A)^1$,
	\begin{equation*}
		\Delta(\xi+\alpha)=\{\eta,\xi+\alpha\}=\iota_\xi\eta,
	\end{equation*}
	Since $\Delta$ has bidegree $(-1,1)$, one also gets that $\Delta^k$ vanishes on $\Gamma(A\oplus A^\dagger)\simeq\Gamma(\calL_A)^1$ for all $k>1$.
	Hence the pull-back of sections $(e^\Delta)^\ast:\Gamma(\calL_A)\to\Gamma(\calL_A)$ induces the following bijection
	\begin{equation*}
		(e^\Delta)^\ast\colon\Gamma(A\oplus A^\dagger)\overset{\sim}{\longrightarrow}\Gamma(A\oplus A^\dagger),\ \xi+\alpha\longmapsto\xi+\alpha+\eta^\flat(\xi).
	\end{equation*}
	so that, in particular, $(e^\Delta)^\ast$ transforms $\Gamma(A)\simeq\Gamma(\calL)^{0,1}$ into $\Gamma(\text{Gr}(\eta))$, i.e.
	\begin{equation}
		\label{eq:theor:DJ_deformation:MC-elements}
		(e^\Delta)^\ast(\Gamma(A))=\gr(\eta).
	\end{equation}

	Now, since $\Theta_B$ has degree $3$ and $\Delta$ has bidegree $(-1,1)$, one gets that $\Delta^k\Theta_B=0$, for all $k>3$, and so one can rewrite the MC equation~\eqref{eq:MC-eqn} as follows
	\begin{equation*}
	0=P((e^{-\Delta})^\ast\Theta_B).
	\end{equation*}
	In view of Lemma~\ref{lem:V-data:DJ}, the latter means exactly that $A$ is involutive wrt the split Courant--Jacobi algebroid structure on $(A\oplus A^\dagger;L)$ encoded by $(e^{-\Delta})^\ast\Theta_B$.
	Further, since $e^{-\Delta}$ is a Jacobi automorphism with inverse $e^{\Delta}$, from Equations~\eqref{eq:theor:CJ_algebroid_as_MC-element}) and~\eqref{eq:theor:DJ_deformation:MC-elements} it follows that  the following conditions are equivalent:
	\begin{itemize}
		\item $A$ is involutive wrt the split Courant--Jacobi algebroid structure encoded by $(e^{-\Delta})^\ast\Theta$,
		\item $\gr(\eta)$ is involutive wrt the Courant--Jacobi algebroid structure encoded by $\Theta_B$.
	\end{itemize}
	Consequently, $\eta$ is an MC element iff $\gr(\eta)$ is involutive w.r.t~$\Theta_B$.

\end{proof}

\begin{remark}
	\label{rem:deformation_Dirac}
	Theorem~\ref{theor:DJ_deformation:MC-elements} can be seen as the Dirac--Jacobi analogue of the result proved for Dirac structures in~\cite[Theorem 6.1]{liu1997manintriples} (when the complement $B$ is Dirac) and in~\cite[Lemma 2.6]{fregier_zambon_2015} (when the complement is not necessarily involutive).
\end{remark}

\begin{example}
	\label{ex:deformations_precontact/Jacobi_structures}
	Theorem~\ref{theor:DJ_deformation:MC-elements} recovers the well know deformation theory of two special Dirac--Jacobi structures within the Courant--Jacobi algebroid $(E;L)$ given by the omni-Lie algebroid $(\bbD L;L)$.
	
	1) Assume that $A$ is the Dirac--Jacobi structure $DL$, i.e.~the graph of the zero precontact structure.
	Then as $B$ we can choose the complementary \emph{abelian} Dirac--Jacobi structure $J^1L$.
	In this case,
	\begin{itemize}
		\item the Dirac--Jacobi structures transverse to $DL$ are nothing but the precontact structures with underlying line bundle $L\to M$ (cf.~Example~\ref{ex:DJ:omniLie}~\ref{enum:ex:DJ:omniLie:1}),
		\item the deformation $L_\infty[1]$ algebra of $DL$ (up to décalage isomorphism~\eqref{eq:decalage_isomorphism}) boils down to the shifted der-complex $(\Omega^\bullet(L)[1],\rmd_D)$ of Atiyah $L$-valued form (cf.~Example~\ref{ex:der-differential}).
	\end{itemize}
	Now, applying Theorem~\ref{theor:DJ_deformation:MC-elements}, we recover that precontact structures are encoded as $2$-cocycles of  $(\Omega^\bullet(L),\rmd_D)$ in agreement with the identification~\eqref{eq:precontact_structures}.
	
	2) Assume that $A$ is the \emph{abelian} Dirac--Jacobi structure $J^1L$, i.e.~the graph of the zero Jacobi structure.
	Then as $B$ we can pick the complementary Dirac--Jacobi structure $DL$.
	Under these assumptions,
	\begin{itemize}
		\item the Dirac--Jacobi structures transverse to $DL$ are the same thing as the Jacobi structures on $L\to M$ (cf.~Example~\ref{ex:DJ:omniLie}~\ref{enum:ex:DJ:omniLie:2}),
		\item the deformation $L_\infty[1]$ algebra of $J^1L$ (up to the décalage isomorphism~\eqref{eq:decalage_isomorphism}) boils down to the dgLa $((\calD^\bullet L)[1],[-,-]_{\sf{SJ}})$ of multiderivations of $L\to M$ (cf.~Example~\ref{ex:SJ_brackets}).
	\end{itemize}
	So, from Theorem~\ref{theor:DJ_deformation:MC-elements}, we recover that the Jacobi structures are encoded as MC elements of the dgLa $((\calD^\bullet L)[1],[-,-]_{\sf{SJ}})$ in agreement with the identification~\eqref{eq:Jacobi_structures}.
\end{example}

\subsection{Infinitesimal Deformations and Obstructions}

We developed just above the deformation theory of Dirac-Jacobi structures within a fixed Courant--Jacobi algebroid.
In this section, we use it to identify also the infinitesimal deformations and find sufficient criteria for the existence of obstructions.
So let us fix a Dirac--Jacobi structure $A$ in a Courant--Jacobi algebroid $(E;L)$.
\begin{definition}
	\label{def:smooth_deformation}
	A \emph{smooth deformation} of $A$ is a smooth one-parameter family $A_t$ of Dirac--Jacobi structures in $(E;L)$ with $A_0=A$.
\end{definition}

Assume to have a smooth deformation $A_t$ of the Dirac--Jacobi structure $A$ in $(E;L)$.
Upon choosing an almost Dirac--Jacobi structure $B\subset E$ transverse to $A$, we can construct the $L_{\infty}[1]$-algebra $(\Omega^\bullet(A;L)[2],\{\frakm_k\})$, which is the one introduced in Theorem~\ref{theor:deformation_L_infty_algebra} and which governs the deformation problem of the Dirac--Jacobi structure $A$ in $(E;L)$ (see Theorem~\ref{theor:DJ_deformation:MC-elements}) 
Being interested in small deformations of the Dirac--Jacobi structure $A$, one can assume that all $A_t$'s in a smooth deformation of $A$ are transverse to $B$.
Consequently, in view of Theorem~\ref{theor:DJ_deformation:MC-elements}, there is a unique smooth one-parameter family $\eta_t$, with $\eta_0=0$, of MC elements of $(\Omega^\bullet(A;L)[2],\{\frakm_k\})$, such that $A_t=\gr(\eta_t)$.
Differentiating the MC equation for $\eta_t$ at $t=0$, one obtains
\begin{equation*}
	0=\left.\frac{\rmd}{\rmd t}\right|_{t=0}\left(\rmd_{A,L}\eta_t+\frac{1}{2}\frakm_2(\eta_t,\eta_t)+\frac{1}{6}\frakm_3(\eta_t,\eta_t,\eta_t)\right)=\rmd_{A,L} \left(\left.\frac{\rmd}{\rmd t}\right|_{t=0}\eta_t\right).
\end{equation*}
This means that $\dot\eta_0=\left.\frac{\rmd}{\rmd t}\right|_{t=0}\eta_t\in\Omega^2(A;L)$ is $\rmd_{A,L}$-closed and justifies the following definition.

\begin{definition}
	\label{def:infinitesimal_deformation}
	An \emph{infinitesimal deformation} of $A$ is a $2$-cocycle in the complex $(\Omega^\bullet(A;L),\rmd_{A,L})$.
\end{definition}

So each smooth deformation gives rise, as its derivative at $t=0$, to an infinitesimal deformation.
The converse is generally false: there may exist \emph{obstructed infinitesimal deformations}, i.e.~infinitesimal deformations of the Dirac--Jacobi structure $A$ which do not arise from smooth deformations.
If this is the case, the deformation problem of the Dirac--Jacobi structure $A$ is said to be \emph{obstructed}.
This reflects the fact that the space of  $(E;L)$-Dirac--Jacobi structures may fail to be smooth around $A$.
The $L_{\infty}[1]$-algebra $(\Omega^\bullet(A;L)[2],\{\frakm_k\})$ controlling the deformation problem of $A$ gives a criterion for the existence of obstructions.
Indeed, obstructions can be detected by means of the \emph{Kuranishi map}
\begin{equation}
	\label{eq:kuranishi_map}
	\mathrm{Kur}\colon H^2(\Omega^\bullet(A;L),\rmd_{A,L}) \to H^3(\Omega^\bullet(A;L),\rmd_{A,L}),\quad [w]\mapsto [\frakm_2(w,w)].
\end{equation}

\begin{proposition}
	\label{prop:kuranishi_criterion}
	Let $\eta$ be an infinitesimal deformation of $A$.
	If $\operatorname{Kur}[\eta]\neq 0$, then $\eta$ is obstructed.
\end{proposition}

\begin{proof}
	Assume that $\eta$ is not obstructed, i.e.~it arises from a smooth deformation $A_t$ of $A$.
	Let $\eta_t$ be the smooth $1$-parameter family of MC elements of $(\Omega^\bullet(A;L)[2],\{\frakm_k\})$, with $\eta_0=0$, such that $A_t=\gr(\eta_t)$.
	Then one gets $\eta=\left.\frac{\rmd}{\rmd t}\right|_{t=0}\eta_t$.
	Taking the second derivative of the MC equation for $\eta_t$ at $t=0$, one gets
	\begin{equation*}
		0=\left.\frac{\rmd^2}{\rmd^2 t}\right|_{t=0}\left(\rmd_{A,L}\eta_t+\frac{1}{2}\frakm_2(\eta_t,\eta_t)+\frac{1}{6}\frakm_3(\eta_t,\eta_t,\eta_t)\right)=\rmd_{A,L} \left(\left.\frac{\rmd^2}{\rmd^2 t}\right|_{t=0}\eta_t\right)+\frakm_2\left(\left.\frac{\rmd}{\rmd t}\right|_{t=0}\eta_t,\left.\frac{\rmd}{\rmd t}\right|_{t=0}\eta_t\right).
	\end{equation*}
	So $\frakm_2(\eta,\eta)$ is $\rmd_{A,L}$-exact, with primitive given by $-\left.\frac{\rmd^2}{\rmd^2t}\right|_{t=0}\eta_t$, and this concludes the proof.
\end{proof}

\appendix

\section{\texorpdfstring{A reminder on $L_\infty$ Algebras}{A reminder on L-infty Algebras}}
\label{app:L_infty_algebras}

Even though the paper requires some familiarity with the language of $L_\infty$ algebras~\cite{lada1995strongly} (or equivalently of $L_\infty[1]$ algebras), in this section we review the identification of (morphisms of) $L_\infty[1]$ algebras with (morphisms of) codifferential coalgebras and recall, without proofs, the relevant results about the equivalence of higher derived brackets~\cite{CS}.

\color{black}
\subsection{$L_\infty[1]$-Algebra (Morphisms) as Codifferential Coalgebra (Morphisms)}

Let $C$ be a graded coalgebra with coproduct $\mu:C\to C\otimes C$.
A \emph{degree $k$ coderivation} of $C$ is a degree $k$ graded linear map $X:C\to C$ s.~t.
\begin{equation*}
\mu\circ X=(X\otimes\id+X\circ\id)\circ\mu.
\end{equation*}
The space $\operatorname{CoDer}^k(C)$ of degree $k$ graded coderivations of $C$ has a natural structure of vector space.
Then $\operatorname{CoDer}^\bullet(C):=\oplus_{k\in\bbZ}\operatorname{CoDer}^k(C)$, the $\bbZ$-graded vector space of graded coderivations of $C$, has a natural structure of graded Lie algebra with Lie bracket $[-,-]$ given by the usual graded commutator
\begin{equation*}
[X,Y]=X\circ Y-(-)^{|X||Y|}Y\circ X,
\end{equation*}
for all homogeneous $X,Y\in\operatorname{CoDer}^\bullet(V)$.

Let $V=\oplus_{k\in\bbZ}V^k$ be a graded vector space.
Then its graded symmetric algebra $\sf{S}^\bullet V:=\oplus_{k\in\bbN}\sf{S}^kV$ inherits from the tensor algebra ${\sf T}^\bullet V=\oplus_{k\in\bbN}V^{\otimes k}$ the structure of a graded coalgebra with coproduct $\mu$ given by
\begin{equation*}
\mu(v_1\odot\ldots\odot v_n)=\sum_{i=1}^{n-1}\sum_{\sigma\in{\sf S}_{i,n-i}}\epsilon(\sigma,{\bf v})(v_{\sigma(1)}\odot\ldots\odot v_{\sigma(i)})\otimes (v_{\sigma(i+1)}\odot\ldots\odot v_{\sigma(n)}).
\end{equation*}

\begin{proposition}
	\label{prop:coalgebra_coderivation}
	For any graded vector space $V$, there is a degree $0$ graded linear isomorphism
	\begin{equation*}
	\operatorname{CoDer}^\bullet({\sf S} V)\longrightarrow \Hom^\bullet({\sf S}V,V)=\oplus_{n\in\bbN}\Hom^\bullet({\sf S}^n V,V)
	\end{equation*}
	mapping each $\calQ\in\operatorname{CoDer}^k({\sf S}V)$ to the family $\{\calQ_n\}_{n\in\bbN}\in\oplus_{n\in\bbN}\Hom^k({\sf S}^nV,V)$ given as follows
	\begin{equation*}
	\begin{tikzcd}
	{\sf S}^nV\arrow[rrr, bend right=15, swap, "\calQ_n"] \arrow[r, hook, "\text{incl}"]&{\sf S}^\bullet V\arrow[r, "\calQ"]&{\sf S}^\bullet V\arrow[r, two heads, "\pr_1"]&V
	\end{tikzcd}
	\end{equation*}
	where $\pr_k:{\sf S}^\bullet V\to{\sf S}^kV$ denotes the projection.
	In particular, $\calQ\in\operatorname{CoDer}^k({\sf S}V)$ can be reconstructed out of the family $\{\calQ_n\}_{n\in\bbN}\in\oplus_{n\in\bbN}\Hom^k({\sf S}^nV,V)$ as follows
	\begin{equation*}
	\calQ(v_1\odot\ldots\odot v_n)=\sum_{i=1}^n\sum_{\sigma\in{\sf S}_{i,n-i}}\epsilon(\sigma;{\bf v})\calQ_i(v_{\sigma(1)}\odot\ldots\odot v_{\sigma(i)})\odot v_{\sigma(i+1)}\odot\ldots\odot v_{\sigma(n)}
	\end{equation*}
	for all homogeneous $v_1,\ldots,v_n\in V$.
	Additionally, the linear isomorphism induces a bijection between
	\begin{itemize}
		\item \emph{codifferentials} $\calQ$ of ${\sf S}V$, i.e.~$\calQ\in\operatorname{CoDer}^1({\sf S}V)$ such that $[\calQ,\calQ]\equiv 2\calQ^2=0$,
		\item $L_\infty[1]$ algebra structures $\{\calQ_n\}_{n\in\bbN}$ on $V$.
	\end{itemize}
\end{proposition}

Let $C$ and $C^\prime$ be graded coalgebras with coproducts respectively $\mu$ and $\mu^\prime$.
A degree $k$ graded coalgebra morphism $C\to C^\prime$ is a degree $k$ graded linear map $\Phi:C\to C^\prime$ such that
\begin{equation*}
(\Phi\otimes\Phi)\circ\mu=\mu^\prime\circ\Phi.
\end{equation*}
The space $\Hom^k(C,C^\prime)$ of degree $k$ graded coalgebra morphisms has a natural structure of vector space, and so one can also construct $\Hom^\bullet(C,C^\prime):=\oplus_{k\in\bbZ}\Hom^k(C,C^\prime)$, the graded space of graded coalgebra morphisms.

The identity map $\id_C$ gives a coalgebra morphism $C\to C$, and the composition of two coalgebra morphisms is still a coalgebra morphism.
So, one can also introduce the obvious notion of graded coalgebra isomorphism.

\begin{proposition}
	\label{prop:coalgebra_morphism}
	For any graded vector spaces $V$ and $W$, there exists a (degree $0$) graded linear isomorphism
	\begin{equation*}
	\Hom^\bullet({\sf S} V,{\sf S} W)\longrightarrow \Hom^\bullet({\sf S}V,W)=\oplus_{n\in\bbN}\Hom^\bullet({\sf S}^n V,W)
	\end{equation*}
	mapping each $\Phi\in\Hom^k({\sf S}V,{\sf S}W)$ to the family $\{\Phi_n\}_{n\in\bbN}\in\oplus_{n\in\bbN}\Hom^k({\sf S}^nV,W)$ given as follows
	\begin{equation*}
	\begin{tikzcd}
	{\sf S}^nV\arrow[rrr, bend right=15, swap, "\Phi_n"] \arrow[r, hook, "\text{incl}"]&{\sf S}^\bullet V\arrow[r, "\Phi"]&{\sf S}^\bullet W\arrow[r, two heads, "\pr_1"]&W
	\end{tikzcd}
	\end{equation*}
	In particular, $\Phi$ can be reconstructed out of the family $\{\Phi_n\}_{n\in\bbN}$ as follows
	\begin{equation*}
	\Phi(v_1\odot\ldots\odot v_n)=\sum_{i=1}^n\sum_{p_1+\ldots+p_i=n}\sum_{\sigma\in{\sf S}_n}\frac{\epsilon(\sigma;{\bf v})}{i!p_1!\cdots p_i!}\Phi_{p_1}(v_{\sigma(1)}\odot\ldots\odot v_{\sigma(p_1)})\odot\cdots\odot \Phi_{p_i}(v_{\sigma(p_{i-1}+1)}\odot\ldots\odot v_{\sigma(n)})
	\end{equation*}
	for all homogeneous $v_1,\ldots,v_n\in V$.
	Additionally, it turns out that the graded coalgebra morphism $\Phi:{\sf S}V\to{\sf S}W$ is an invertible iff the graded linear map $\Phi_1:V\to W$ is invertible.
\end{proposition}

\begin{remark}
	\label{rem:L_infinity_algebra_iso}
	Let $(V,\{\calQ_n\}_{n\in\bbN})$ and $(V^\prime,\{\calQ^\prime_n\}_{n\in\bbN})$ be $L_\infty[1]$-algebras, with $\calQ$ and $\calQ^\prime$ the corresponding codifferentials of respectively ${\sf S}V$ and ${\sf S}V^\prime$.
	Then a \emph{$L_\infty[1]$-algebra (iso)morphism} $(V,\{\calQ_n\}_{n\in\bbN})\longrightarrow(V^\prime,\{\calQ^\prime_n\}_{n\in\bbN})$ is a degree $0$ codifferential graded coalgebra (iso)morphism $\Phi:({\sf S}V,\calQ)\to({\sf S}V^\prime,\calQ^\prime)$, i.e.~a degree $0$ graded coalgebra (iso)morphism $\Phi:{\sf S}V\to{\sf S}V^\prime$ such that $\calQ^\prime\circ\Phi=\Phi\circ\calQ$.
\end{remark}

\subsection{Equivalences of Higher Derived Brackets}
\label{sec:equivalence_higher_derived_brackets:CS}
Let us start recalling the notion of V-data.
\begin{definition}[{\cite[Definition 1.7]{fregier_zambon_2015}}]
	\label{def:V-data}
	A \emph{V-data} $(\frakh,\fraka,P,\Theta)$ consists of:
	\begin{itemize}
		\item a graded Lie algebra $\frakh$, with Lie bracket $[-,-]$,
		\item an abelian graded Lie subalgebra $\fraka\subset\frakh$ (so that $[\fraka,\fraka]=0$),
		\item a projection $P:\frakh\to\fraka$ such that $\ker P\subset\frakh$ is a graded Lie subalgebra,
		\item a MC element $\Theta$ of $\frakh$ such that $P(\Theta)=0$. 
	\end{itemize}
	Removing the condition $\Theta\in\ker P$, one obtains what is called a \emph{curved V-data}.
\end{definition}
Following~\cite{voronov2005higher}, one can use a  V-data $(\frakh,\fraka,P,\Theta)$ to cook up an $L_\infty[1]$-algebra structure on $\fraka$.

\begin{proposition}[{\cite{voronov2005higher}}]
	\label{prop:higher_derived_brackets:Dirac}
	A V-data $(\frakh,\fraka,P,\Theta)$ determines an $L_\infty[1]$-algebra structure on $\fraka$ whose multibrackets $\calQ_k\in\Hom^1({\sf S}^k\fraka,\fraka)$ are given by the following higher derived brackets, for all $a_1,\ldots,a_k\in\fraka$,
	\begin{equation*}
	\calQ_k(a_1\odot\cdots\odot a_k)=P[[\ldots[\Theta,a_1],\ldots],a_n].
	\end{equation*}
	Clearly, the latter corresponds to a codifferential $\calQ$ of the ${\sf S}\fraka$, the symmetric coalgebra of $\fraka$, in view of Proposition~\ref{prop:coalgebra_coderivation}.
	
\end{proposition}

Let us fix, from now on, a set of V-data $(\frakh,\fraka,P,\Theta)$ and construct the associated $L_\infty[1]$-algebra structure $\{\calQ_k\}_{k\in\bbN}$ on $\fraka$ with corresponding codifferential $\calQ$ of ${\sf S}\fraka$.
Let us also fix a degree $0$ graded Lie algebra derivation ${\sf m}$ of $\frakh$ which can be integrated to a $1$-parameter group of graded Lie algebra automorphisms $\phi_\varepsilon$ of $\frakh$, \emph{the flow of ${\sf m}$}, i.e.~the solution of the following Cauchy problem
\begin{equation*}
\frac{\rmd}{\rmd\varepsilon}\phi_\varepsilon={\sf m}\circ\phi_\varepsilon,\quad\phi_0=\id.
\end{equation*}
From now on we will assume that the derivation ${\sf m}$ satisfies the following conditions:
\begin{enumerate}[label=\arabic*)]
	\item\label{enumitem:CS:1} ${\sf m}$ preserves $\ker P$, i.e.~$P\circ{\sf m}\circ P=P\circ {\sf m}$,
	\item\label{enumitem:CS:2} $\lambda_\epsilon=0$ is the only solution to the Cauchy problem $\frac{\rmd}{\rmd\varepsilon}\lambda_\varepsilon=P{\sf m}\lambda_\varepsilon$, $\lambda_0=0,$
\end{enumerate}
Notice that Conditions~\ref{enumitem:CS:1} and~\ref{enumitem:CS:2} imply that the flow $\varphi_\varepsilon$ preserves $\ker P$, i.e.~$P\circ\phi_\varepsilon\circ P=P\circ\phi_\varepsilon$.

For each $\varepsilon$ we can introduce the new MC element $\Theta(\varepsilon)$ of $\frakh$ given by $\Theta(\varepsilon)=\phi_\varepsilon\Theta$.
Since $\Theta(\varepsilon)\in\ker P$, we also get a new V-data $(\frakh,\fraka,P,\Theta(\varepsilon))$, and the associated $L_\infty[1]$-algebra structure $\{\calQ(\varepsilon)_k\}_{k\in\bbN}$ on $\fraka$ with corresponding codifferential $\calQ(\varepsilon)$ of ${\sf S}\fraka$.
Obviously, $\Theta(0)=\Theta$, and so $\calQ(0)=\calQ$.

In view of Proposition~\ref{prop:coalgebra_coderivation}, there exists a unique degree $0$ graded coalgebra coderivation ${\sf M}$ of ${\sf S}\fraka$ such that the corresponding family $\{{\sf M}_n\}_{n\in\bbN}\in\oplus_{n\in\bbN}\Hom^0({\sf S}^n\fraka,\fraka)$ is given as follows, for all $a_1,\ldots,a_n\in\fraka$,
\begin{equation*}
{\sf M}_n(a_1\odot\cdots\odot a_n)=P[[\ldots[{\sf m}a_1,a_2],\ldots],a_n].
\end{equation*}
Since ${\sf m}$ satisfies Conditions~\ref{enumitem:CS:1} and~\ref{enumitem:CS:2}, the following is a consequence of Proposition 3.2 in~\cite{CS}.
\begin{proposition}
	\label{prop:coderivation_flow:CS}
	${\sf M}$ integrates to a $1$-parameter group of coalgebra automorphisms $U(\varepsilon)$ of ${\sf S}\fraka$, the \emph{flow of ${\sf M}$}, i.e.~there is a unique solution of the following Cauchy problem
	\begin{equation*}
	\frac{\rmd}{\rmd\varepsilon}U(\varepsilon)={\sf M}\circ U(\varepsilon),\quad U(0)=\id.
	\end{equation*}
\end{proposition}

Finally we are in condition to state the main result about the equivalence of higher derived brackets.

\begin{theorem}[{\cite[Theorem 3.2]{CS}}]
	\label{theor:CS}
	For each $\varepsilon$, the graded coalgebra automorphism $U(\varepsilon)$ of ${\sf S}\fraka$ gives a codifferential graded coalgebra isomorphism
	\begin{equation*}
	U(\varepsilon):({\sf S}\fraka,\calQ(0))\longrightarrow ({\sf S}\fraka,\calQ(\varepsilon))
	\end{equation*}
	or equivalently an $L_\infty[1]$-algebra isomorphism $\{U(\varepsilon)_n\}_{n\in\bbN}$ from $(\fraka,\{\calQ(0)_n\}_{n\in\bbN})$ to $(\fraka,\{\calQ(\varepsilon)_n\}_{n\in\bbN})$.
\end{theorem}

\small
\addtocontents{toc}{\SkipTocEntry}
\section*{Acknowledgements}

The author is grateful to Paulo Antunes, Joana Nunes da Costa and Luca Vitagliano for useful discussions and helpful suggestions.
He has been supported by an FWO postdoctoral fellowship during the preparation of this paper.
Further, he is member of the National Group for Algebraic and Geometric Structures, and their Applications (GNSAGA – INdAM) and is partially supported by CMUP, which is financed by national funds through FCT – Fundação para a Ciência e a Tecnologia, I.P., under the project with reference UIDB/00144/2020.
The author also acknowledge the partial support by the FWO research project G083118N (Belgium) and the hospitality by the Centro de Matemática da Universidade de Coimbra received in the early phases of this work.

\bibliographystyle{plain}
\bibliography{CourantAlgebroids}

\end{document}